\documentclass{article}
\usepackage[utf8]{inputenc}
\usepackage{amssymb,amsmath}
\usepackage{amsmath}
\usepackage{amssymb}
\usepackage{amsthm}
\usepackage{tikz-cd}
\usepackage{float}
\usepackage{todonotes}
\usepackage{algorithm,algorithmic}
\usepackage{xcolor}
\usepackage{comment}
\usepackage{enumitem}
\usepackage{physics}
\usepackage{url}
\usepackage{rotating,lscape}

\newtheorem{theorem}{Theorem}
\newtheorem{proposition}[theorem]{Proposition}

\newtheorem{lemma}[theorem]{Lemma}

\theoremstyle{definition}
\newtheorem{definition}[theorem]{Definition}
\newtheorem{example}[theorem]{Example}
\newtheorem{remark}[theorem]{\bf Remark}
\newtheorem{condition}[theorem]{Condition}

\newcommand{\RR}{\mathbb{R}}
\newcommand{\RRp}{\mathbb{R}_{\geq 0}}
\newcommand{\RRpp}{\mathbb{R}_{> 0}}
\newcommand{\CC}{\mathbb{C}}
\newcommand{\NN}{\mathbb{N}}
\newcommand{\ZZ}{\mathbb{Z}}
\newcommand{\QQ}{\mathbb{Q}}

\newcommand{\vf}{\boldsymbol{f}}
\newcommand{\vF}{\boldsymbol{F}}
\newcommand{\vX}{\boldsymbol{X}}
\newcommand{\vz}{\boldsymbol{z}}
\newcommand{\vZ}{\boldsymbol{Z}}
\newcommand{\vy}{\boldsymbol{y}}
\newcommand{\vx}{\boldsymbol{x}}
\newcommand{\vk}{\boldsymbol{k}}
\newcommand{\vc}{\boldsymbol{c}}

\newcommand{\valpha}{\boldsymbol{\alpha}}

\newcommand{\vxbar}{\boldsymbol{\bar{x}}}
\newcommand{\vkbar}{\boldsymbol{\bar{k}}}

\newcommand{\vfhat}{{\boldsymbol{\hat{f}}}}
\newcommand{\vFhat}{{\boldsymbol{\hat{F}}}}
\newcommand{\vfcheck}{{\boldsymbol{\check{f}}}}

\newcommand{\vfbar}{{\boldsymbol{\bar{f}}}}
\newcommand{\vgbar}{{\boldsymbol{\bar{g}}}}
\newcommand{\vFbar}{{\boldsymbol{\bar{F}}}}

\newcommand{\vftilde}{{\boldsymbol{\tilde f}}}
\newcommand{\vFtilde}{{\boldsymbol{\tilde F}}}

\newcommand{\vect}[1]{\boldsymbol{#1}}

\newcommand{\ord}[1]{ {o}( #1 )}
\newcommand{\Ord}[1]{ {O}( #1 )}
\newcommand{\DP}[2]{ \ensuremath{ \frac{\partial #1 }{\partial #2 } } }
\newcommand{\D}[2]{ \ensuremath{ \frac{d #1 }{d #2 } } }
\newcommand{\DD}[2]{ \ensuremath{ \frac{d^2 #1 }{d {#2}^2 } } }
\newcommand{\transpose}[1]{\ensuremath{ #1^\mathsf{T}}}

\usepackage{hyperref}

\newcommand{\cor}[1]{{\color{black} #1}}

\usepackage[left=3cm, right=3cm, top=2cm]{geometry}

\title{Reduction of Chemical Reaction Networks with Approximate Conservation Laws}

\author{Aurélien Desoeuvres$^{1}$, Alexandru Iosif$^7$, Christoph L\"uders$^6$, \\
  Ovidiu Radulescu$^{1,*}$, Hamid Rahkooy$^8$, Matthias Sei\ss$^2$, Thomas Sturm$^{3,4,5}$ }

 \date{
    $^1$ University of Montpellier and CNRS LPHI,  Montpellier, France\\
    $^2$ University of Kassel, Kassel, Germany \\
    $^3$ CNRS, INRIA, and the University of Lorraine, Nancy, France\\
    $^4$ Max Planck Institute for Informatics, Saarbr\"ucken, Germany  \\
    $^5$ Saarland University, Saarbrücken, Germany \\
    $^6$ University of Bonn, Bonn, Germany \\
    $^7$ Rey Juan Carlos University, Madrid, Spain\\
     $^8$ University of Oxford, Oxford, United Kingdom.\\
    $^*$ corresponding author \href{mailto:ovidiu.radulescu@umontpellier.fr}{ovidiu.radulescu@umontpellier.fr} \\
 \bf \today
 }

\begin{document}

\maketitle

\begin{abstract}
Model reduction of fast-slow chemical reaction networks based on the quasi-steady state approximation fails when the fast subsystem has 
first integrals. We call these first integrals approximate conservation laws.
In order to define fast subsystems and identify approximate 
conservation laws, we use ideas from tropical geometry. 
We prove that any approximate conservation 
law evolves slower than all the species involved in it and therefore 
represents a supplementary slow variable in an extended system.
By elimination of some variables of the extended system, we obtain 
networks without approximate conservation laws, which can be 
reduced by standard singular perturbation methods. 
The field of applications of approximate conservation laws covers the 
{\em quasi-equilibrium approximation}, well known in biochemistry.
\cor{We discuss reductions of slow-fast as well as  multiple 
timescale systems. }
Networks with multiple timescales have hierarchical relaxation. 
At a given timescale, our multiple timescale reduction method defines three subsystems composed of
(i) slaved fast variables satisfying algebraic equations, (ii) slow driving variables
satisfying reduced ordinary differential equations, and (iii) quenched much slower variables that are constant. The algebraic equations satisfied by fast variables define chains of nested normally hyberbolic invariant manifolds. In such chains, faster 
manifolds are of higher dimension and contain the slower manifolds.    
Our reduction methods are introduced algorithmically for networks with \cor{monomial reaction rates and} 
linear, monomial or polynomial approximate conservation laws. 
\cor{We propose symbolic algorithms
to reshape and rescale the networks
such that geometric 
singular perturbation theory can be applied to them, 
test the applicability of the theory, and finally 
reduce the networks. }
As a proof of concept, we apply this method to a model 
of \cor{the} TGF-$\beta$ signaling pathway. 

\end{abstract}

\section{Introduction} 

The study of {\em chemical reaction networks} (CRN) was motivated by important applications in physics and chemistry, concerning models of non-equilibrium thermodynamics \cite{schlogl1972chemical}, 
 catalytic reactions \cite{yablonskii1991kinetic}, combustion \cite{semenov1956some}, etc.
More recently, CRNs were used to model cell and tissue physiology \cite{theret2020integrative} needed
for the understanding of the fundamental mechanisms of living systems and for fighting 
disease. 

Chemical reaction networks can be characterized by reaction {\em stoichiometry} and reaction rates \cite{aris1963independence}. Stoichiometry tells us how many molecules of each species are consumed and
produced in a reaction. For instance, the reaction $A_1 + A_2 \rightarrow A_3$ consumes one molecule of each $A_1$ and $A_2$
and produces one molecule of $A_3$. A useful construct is the stoichiometric matrix $\vect{S}$ in which each column 
represents the net numbers of molecules
of each species produced by a particular reaction. A three species CRN made of the 
reactions  $A_1 + A_2 \rightarrow A_3$ and $A_1 + A_1 \rightarrow A_1$ has the stoichiometric matrix 
$$\vect{S} = \begin{pmatrix} -1 & -1 \\ -1 & 0 \\ 1 & 0 \end{pmatrix}.$$ With each reaction we also associate a 
positive function
of species concentrations, called reaction rate, representing the number of 
occurrences of the reaction
per unit time and volume. In this paper we assume that reaction rates are 
monomials. 
Therefore, the deterministic kinetics of CRNs is described by sets of polynomial 
ordinary differential equations. 
For instance, if the reaction rates in the above example are $R_1=k_1 x_1 x_2$ and $R_2 = k_2 x_1^2$, 
the CRN kinetics is described by $\dot x_1 = - k_1 x_1 x_2 - \frac{1}{2} k_2 x_1^2$, $\dot x_2 =  -k_1 x_1 x_2$ and
$\dot x_3 =  k_1 x_1 x_2 $. 

A lot of effort has been dedicated to studying the behavior of {\em mass-action} networks \cite{Feinberg:19a}. In such networks, the probability that two species react is proportional to their abundances, and therefore
the reaction rates are monomials \cor{in the} concentrations of reactant species with exponents 
equal to the number of molecules entering the reaction. The above example is of this type, but will no
longer be of this type if $R_1$ is changed to \cor{$k_1 x_1 x_2^2$}, for instance. 
This constraint leads to algebraic properties exploited in {\em chemical reaction network theory} (CRNT), which has been initiated by Horn, Jackson and Feinberg \cite{feinberg1974dynamics,Feinberg:19a}. 
In order to cover more general models,  we do not impose the mass action  constraint on the reaction rates.
\cor{Instead, to avoid that some species become negative as a result
of the CRN kinetics, we use a weaker constraint: if a reaction 
consumes a species, then its rate is proportional
 to a strictly positive power of the concentration of this species. }
In spite of significant progress
towards elucidating the properties of CRNs, important models are left aside because of their size and complexity of their dynamics. For such examples, algorithmic model reduction, which 
transforms complex networks into simpler ones that can be more easily analysed,
becomes a necessity. A few attempts of developing model reduction algorithms used concepts of CRNT such as 
complex balance \cite{rao2013graph,feliu2013simplifying}, but they were limited to networks functioning at the steady state or based on ad hoc identification 
of the balanced complexes. Model reduction methods based on the theory of singular perturbations, 
employing concepts such as the intrinsic low dimensional manifold or  
quasi-steady state reduction, are often used in chemistry 
and systems biology \cite{surovtsova2009accessible,gerdtzen2004non}.
\cor{However, these methods 
lack general algorithms for finding appropriate  small parameters and scalings, needed
for the quasi-steady state reduction, and are limited to two time-scales (slow-fast systems).  
}

Recently, we  proposed a symbolic method for algorithmic reduction of chemical reaction networks with multiple \cor{(more than two)} timescales \cite{kruff2021algorithmic}. This method combines
{\em tropical geometry} ideas 
for identifying the time and concentration scales 
 \cite{samal2015geometric,radulescu2020tropical}, 
dominance principles based on comparison of 
orders of magnitude \cite{gr08,radulescu2008robust,gorban2010asymptotology} 
and
singular perturbation results \cite{fenichel1979geometric,cardin2017fenichel} to justify the reduction. 
For a given timescale, the method defines three subsystems: a {\em slaved equilibrated} subsystem, a {\em driving evolving} subsystem and a {\em quenched}
subsystem. 
The variables of the slaved subsystem satisfy quasi-steady equations,
defined as equilibria of the fast truncated ODEs in which all the remaining variables are considered fixed.  
The entire construction requires
 the  hyperbolicity of the quasi-steady state, which is needed in the classical geometrical 
singular perturbation theory of Fenichel  \cite{fenichel1979geometric}.
Although general in its implementation, this reduction method fails in a number of cases. A major cause of failure is the degeneracy of the quasi-steady state, when the fast dynamics  
has a continuous variety of steady states. Typically, this happens when the fast truncated ODEs have first integrals, 
i.e.~quantities that are conserved on any trajectory, \cor{whose values depend} on the initial conditions. The quasi-steady states are no longer hyperbolic, because
the Jacobian matrix of the {\em fast} part of the dynamics is in this case singular. 
This type of singular perturbations, called {\em critical}, is known  since the work of Vasil'eva and Butuzov in the 70's \cite{vasil1980singularly}, but its origins can be 
traced back to the early theory of enzymatic reactions, as quasi-equilibrium is an 
instance of critical singular perturbations. Vasil'eva and Butuzov
\cite{vasil1980singularly} propose asymptotic expansions of the solutions of 
singularly perturbed systems in the critical case  
based on their method of {\em boundary series}, which are two timescale expansions. 
Their method works in the case when the problem has rigorously two timescales, but
can not be applied in the case when there are more than two timescales. 
We show in this paper that critical singular perturbations may have more timescales
than are apparent after rescaling parameters and variables. 
\cor{We also provide algorithmic methods 
to compute these extra timescales that correspond
to approximate conservation laws.}

Exact conservation laws, i.e.~first integrals of the full dynamics, were
already used for model order reduction. If such quantities exist, \cor{the model
can be reduced by eliminating} 
a number of variables and equations equal to the number of
independent conservation laws \cite{lemaire2014defining,mahdi2017conservation}. 
In the present work, we introduce the {\em approximate conservation laws}
that are quantities conserved by the fast dynamics. 
\cor{In models with multiple timescales, approximate conservation laws
provide extra slow variables. Model reduction takes place by elimination 
of the fast variables.}
We thus provide algorithmic reduction methods covering the case of non-hyperbolic fast dynamics with conservation. 
Our algorithms are inspired from the \cor{well-known}
{\em quasi-equilibrium} approximation \cite{gorban2010asymptotology,radulescu2012reduction}. 
Contrary to the quasi-equilibrium approximation that uses only  linear approximate conservation laws, here we are also exploiting non-linear conserved quantities. Similar ideas were developed in \cor{\cite{schneider2000model,del2008modular,auger2008aggregation}},
but without a full algorithmic solution. Like in \cite{kruff2021algorithmic} we use tropical geometry to find appropriate time and concentration scales. 

\cor{The topic of conservation laws has a broad interest and its relation to 
symmetry was widely studied in classical and quantum mechanics (see Noether's theorems \cite{noether1918invarianten}).} Referring to the broad range of phenomena in 
physics, from nuclear forces to gravitation, having {\em near symmetry}, R. Feynman concludes that
``God made the laws only nearly symmetrical so that we should not be jealous of His perfection'' \cite{feynman1963}.
Approximate continuous symmetries (Lie-B\"acklund symmetries) were studied for  differential equations
with 
or without a Lagrangian (see \cite{baikov1988approximate} and  \cite{kara1999approximate}, respectively). Beyond their utility in the theory of 
regular and singular perturbations, 
\cor{exact and approximate} symmetries can be used to gain insight into the dynamics of complex \cor{chemical reaction networks.}
 Approximate conservation laws, valid for certain concentrations and not valid for other concentrations
 of biochemical species, imply that the same biochemical system can have multiple behaviors depending
 on the internal or external stimuli. 
 In biology,  ``imperfect'' conservation allows living systems to be 
 flexible, to evolve and adapt to changes of the environment. 
 This also leads to multiple dynamical phenomena: slow metastable states, 
 bifurcations in the fast dynamics of the system and itineracy when the system 
 switches from one metastable state to another \cite{gr08,auger2008aggregation,rabinovich2008transient}.

 The structure of this paper is as follows. Section 2 introduces the class of 
 models we are dealing with. These are systems of ODEs whose r.h.s.~are 
 integer coefficient polynomials in species concentrations and reaction rate constants.
 Theorem 1 shows that these models are always endowed with a stoichiometric matrix.
 Section 3 introduces the concept of exact and approximate conservation laws, without
 methods for computing them. 
 The methods for symbolic computation of approximate conservation laws are
 presented in  \cite{desoeuvres2022complete}.
 Section 4 introduces our reduction method based on tropically constrained formal scalings and approximate conservation laws. An important result here is that approximate conservation laws
 are always slower than all the species involved in their structure. This implies
 that they can be used as new, unconditionally slow variables, providing 
 robust reductions. In \ref{sec:SM1} we illustrate
 the method via a case study from molecular biology of intracellular signal transduction.

\section{Models}
In this paper we consider CRNs with $n$ species $A_1,\ldots,A_n$,  
whose concentrations $\vx = (x_1,\ldots,x_n)$
follow a system of ODEs of the form
\begin{equation} \label{eq:fi}
\dot x_1 = f_1(\vk,\vx), \, \dots , \, \dot x_n = f_n(\vk,\vx)  ,
\end{equation}
where 
$$f_i(\vk,\vx) = \sum_{j=1}^r S_{ij} k_j   
\vx^{\vect{\alpha}_j} \in \ZZ[\vk,\vx]=\ZZ[k_1,\dots,k_r,x_1,\dots,x_n].$$ 
The monomials $\vx^{\vect{\alpha}_j}= x_1^{\alpha_{j1}}\ldots x_n^{\alpha_{j_n}}$ appearing in the 
right hand sides of \eqref{eq:fi} are defined 
by $r$ multi-indices $\vect{\alpha}_j=(\alpha_{j1},\ldots,\alpha_{jn}) \in \NN^n$ and  
for each monomial the variable $k_j$ represents a rate constant. The variables $\vk=(k_1,\dots,k_r)$ take values in $\RRpp^r$ and 
the integer coefficients $S_{ij}$ form a matrix $\vect{S}=(S_{ij}) \in \ZZ^{n\times r}$, which is called
the stoichiometric matrix. 
The concentration variables $\vx=(x_1,\ldots,x_n)$  take values in $\RRpp^n$. For reasons that will become clear in  Section~\ref{sec:trop} in Remark~\ref{rem:rrpp}, we exclude zero concentrations. 
We denote the vector of right hand sides of  \eqref{eq:fi} by 
$$\vect{F}(\vk,\vx) = \transpose{(f_1(\vk,\vx),f_2(\vk,\vx),\ldots,
f_n(\vk,\vx))}.$$ 



Mass action networks belong to this class of models. 
For a mass action reaction 
$$\bar{\alpha}_{j1} A_1 +  \ldots + \bar{\alpha}_{jn} A_n \rightarrow 
\beta_{j1} A_1 +  \ldots + \beta_{jn} A_n$$ 
we have that $S_{ij} = \beta_{ji} - \bar{\alpha}_{ji}$ and that 
the reaction rate is  $k_j \vx^{\vect{\bar{\alpha}}_j}$, i.e.~the stoichiometric coefficients $\vect{\bar{\alpha}}_{j}$ and the 
multi-indices $\vect{\alpha}_{j}$ coincide. 
But in general this must not be the case, meaning that the multi-indices $\vect{\alpha}_{j}$ and $\vect{\bar{\alpha}}_j$ in $S_{ij} = \beta_{ji} - \bar{\alpha}_{ji}$ and $k_j \vx^{\vect{\alpha}_j}$
 are not necessarily equal. 
\cor{In order to keep the positive orthant $\RRpp^n$
invariant, we consider that whenever a reaction consumes 
a species, its rate tends to zero when the concentration 
of this species approaches zero. This is equivalent to  
\begin{condition}\label{cond:pos}
$\alpha_{ji} > 0$  for all
$i,j$ such that $S_{ij}<0$. 
\end{condition}
The condition~\ref{cond:pos} is automatically
satisfied by mass action reactions, but covers also 
 truncated ODEs systems introduced  in Section~\ref{sec:exdef}
 and used in Section~\ref{sec:model_reduction}.  
}

\section{Approximate Conservation Laws}

\subsection{A Classical Example and some Definitions}
\label{sec:exdef}

\begin{example}\label{MM}
Let us consider the irreversible 
Michaelis-Menten mechanism that is paradigmatic for enzymatic reactions. We choose rate constants corresponding to the so-called quasi-equilibrium, studied by Michaelis and Menten. 
The reaction network for this model is
$$
  S + E   \underset{k_2}{\overset{k_1}{\rightleftharpoons}} ES    \overset{k_3 \delta}{\longrightarrow} E + P,
$$
where $S$ is a substrate, $E$ is an enzyme, $ES$ is an enzyme-substrate complex and $k_1$, $k_2$, $k_3$ 
are rate constants. Here $0 < \delta < 1$ is a small positive scaling parameter, indicating that the third rate constant  is
small. 

According to mass-action kinetics, the concentrations
$x_1=[S]$, $x_2=[ES]$ and $x_3=[E]$ satisfy the system of 
ODEs
\begin{eqnarray}\label{eqsMM}
\dot{x_1} &=& - k_1 x_1 x_3 + k_2 x_2 , \notag \\
\dot{x_2} &=& k_1 x_1 x_3 - k_2 x_2 - \delta k_3 x_2 , \notag \\
\dot{x_3} &=& -k_1 x_1 x_3 + k_2 x_2 + \delta k_3 x_2.
\end{eqnarray}
We consider the truncated system of ODEs
\begin{eqnarray}\label{eqsMMtruncated}
\dot{x_1} &=& - k_1 x_1 x_3 + k_2 x_2, \notag \\
\dot{x_2} &=& k_1 x_1 x_3 - k_2 x_2 , \notag \\
\dot{x_3} &=& -k_1 x_1 x_3 + k_2 x_2 ,
\end{eqnarray}
 which is obtained by setting $\delta = 0$ in \eqref{eqsMM}. 
The truncated system \eqref{eqsMMtruncated} describes the
dynamics of the model on fast timescales of order $\Ord{\delta^0}$. 

The steady state of the fast dynamics is obtained by equating to zero the r.h.s.~of \eqref{eqsMMtruncated}. The resulting
condition is called  {\bf quasi-equilibrium (QE)}
because it means that the complex
formation rate $k_1 x_1 x_3$ is equal
to the complex dissociation rate $k_2 x_2$.
In other words, the reversible reaction 
$$ S + E \underset{k_2}{\overset{k_1}{\rightleftharpoons}} ES$$ functions at 
equilibrium. The QE condition is reached only at the end of the
fast dynamics and is satisfied
with a precision of order $\Ord{\delta}$ during
the slow dynamics \cite{gorban2010asymptotology}. Because of its approximate validity and purely
kinetic origin, QE is different from the similar concept of detailed balance \cite{Boltzmann}.

We introduce the linear combinations of variables $x_4 = x_1 + x_2$ and $x_5 = x_2 + x_3$ corresponding 
to the total substrate and total enzyme concentrations, respectively.
Addition of the last two equations of \eqref{eqsMM} leads to 
$\dot{x}_5=0$, which means that for solutions of the full system $x_5$ is constant for all times.
We will call such a quantity an {\bf exact conservation law}.  

Addition of the first two and the last two equations of \eqref{eqsMMtruncated} lead to $\dot{x}_4=0$ and
$\dot{x}_5=0$. This means that $x_4$ is constant for 
solutions of the truncated dynamics, valid at short times $t = \Ord{\delta^{0}}$ and is not constant at larger times $t = \Ord{\delta^{-1}}$.
We call such a quantity an  {\bf approximate conservation law}. 
The quantity $x_5$ is both an exact and approximate conservation law. 
\end{example}

More generally, let us consider system \eqref{eq:fi}.
This model depends on the parameters 
$\vk=(k_1, \dots, k_r)$.
After rescaling it by powers of 
a scaling parameter $\delta$ with $0 < \delta < 1$,  it becomes (for details see Section~\ref{sec:formal_scaling})
\begin{eqnarray} \label{fullOriginalSystem}
 \dot{\bar x}_1 & =&  \delta^{b_1} \big(\bar f_1^{(1)}(\vkbar,\vxbar) + \delta^{b'_1} \bar f_1^{(2)}(\vkbar,\vxbar,\delta) \big) , \notag \\
 & \vdots & \notag \\
 \dot{\bar x}_n & = & \delta^{b_n} \big( \bar f_n^{(1)}(\vkbar,\vxbar) + \delta^{b'_n} \bar f_n^{(2)}(\vkbar,\vxbar,\delta)\big),
\end{eqnarray}
where $ b'_i > 0$. The rescaled model \eqref{fullOriginalSystem}
is rewritten in rescaled variables
$\vxbar=(\bar x_1,\dots,\bar x_n)$ and
depends on the rescaled parameters $\vkbar=(\bar k_1,\dots,\bar k_r)$.

The functions $\bar f_i^{(2)}(\vkbar,\vxbar,\delta)$ are infinitely differentiable and their first derivatives in $\delta$ vanish at $\delta=0$. 
We also have
$$f_i(\vk,\vx) = 
\delta^{b_i} \big(\bar f_i^{(1)}(\vkbar,\vxbar) + \delta^{b'_i} \bar f_i^{(2)}(\vkbar,\vxbar,\delta) \big).$$
For small $\delta$ and $b_i'>0$ the terms in $\delta^{b'_i} \bar f_i^{(2)}(\vkbar,\vxbar,\delta)$ are dominated by the terms in $\bar f_i^{(1)}(\vkbar,\vxbar)$. This justifies
to introduce  
the {\bf truncated system} as the system of ODEs obtained
by keeping only the lowest order dominant terms in \eqref{fullOriginalSystem}, namely
\begin{equation} \label{TruncatedSystem}
 \dot{x}_1 =  f_1^{(1)}(\vk,\vx), \, \ldots , \,
 \dot{x}_n = f_n^{(1)}(\vk,\vx),
\end{equation}
where
$f_i^{(1)}(\vk,\vx) = \delta^{b_i} 
\bar f_i^{(1)}(\vkbar,\vxbar)$ for all $1\leq i \leq n$.

\begin{definition} 
\label{def:exactandapproxcons}$\,$
\begin{enumerate}[label=(\alph*)]
    \item \label{def:exactandapproxconsa} A function $\phi(\vx)$ is an {\bf  exact conservation law}, unconditionally 
    on the parameters, if it is a first integral of the full system \eqref{eq:fi}, i.e.~if 
$$\sum_{i=1}^n \DP{\phi}{x_i}(\vx) f_i(\vk,\vx) = 0$$ for all $\vk \in \RRpp^r$,
$\vx \in \RRpp^n$. 
\item \label{def:exactandapproxconsb} A function $\phi(\vx)$ is an {\bf approximate  conservation law}, unconditionally  on the parameters, if it is a first integral of the truncated system \eqref{TruncatedSystem}, i.e.~if 
$$ \sum_{i=1}^n \DP{\phi}{x_i}(\vx)  f_i^{(1)}(\vk,\vx) = 0 $$
for all $\vk \in \RRpp^r$,
$\vx \in \RRpp^n$.
\item An exact \textnormal{(}approximate\textnormal{)} conservation law of the form $c_1 x_1 + \dots + c_n x_n $ with coefficients $c_i \in \RR$
is called an {\bf exact \textnormal{(}approximate\textnormal{)} linear conservation law}. If $c_i \geq 0$ for 
$1\leq i \leq n$, the linear conservation law is
called {\bf semi-positive}.
\item An exact \textnormal{(}approximate\textnormal{)} conservation law of the form $x_1^{m_1} \cdots x_n^{m_n} $ with $m_i \in \ZZ$ is called an {\bf exact \textnormal{(}approximate\textnormal{)} rational monomial conservation law}. 
 If $m_i \in \ZZ_+$ for 
$1\leq i \leq n$, the conservation law is
called {\bf monomial}. For simplicity, in this
paper, we will call both types {\bf monomial}.
\item An exact \textnormal{(}approximate\textnormal{)} conservation law of the form $\sum_{i=1}^s a_i x_1^{m_{1i}} \cdots x_n^{m_{ni}}$ with $m_{ji} \in \ZZ_+$
and $a_i \in \RR$ is called an {\bf exact \textnormal{(}approximate\textnormal{)} polynomial conservation law}. 
\end{enumerate}
\end{definition}
\begin{remark}
In the above definitions $\phi(\vx)$ does not contain $\vk$, and exists for all admissible $\vk$,
which is a slight limitation. This condition facilitates the calculation of time scale orders needed for model
reduction (see Section~\ref{sec:acslow}). 
Of course, some conservation laws may depend on $\vk$, and/or exist only for special values of $\vk$, but finding
them and their existence conditions is generally a much 
more difficult problem. Parametric conservation laws are considered in our parallel paper \cite{desoeuvres2022complete}. 
From now on, for the sake of simplicity and when it is clear from the context, we will refer to the  categories in \eqref{def:exactandapproxconsa} and \eqref{def:exactandapproxconsb} of Definition~\ref{def:exactandapproxcons} as conservation law and approximate conservation law, respectively.  
Let us note that an approximate conservation law can also be exact. 
\end{remark}






Some variables may not appear in a conservation law $\phi(\vx)$. This means that in the case of a linear conservation law some coefficients $c_i$ may be zero, or in the case of a nonlinear conservation law some partial derivatives $\DP{\phi}{x_i}(\vx)$ may vanish. If $r$ is the number of all non-zero quantities $\DP{\phi}{x_i}(\vx)$, then we say that the conservation law {\bf depends on $r$ variables}.

\begin{definition}
\label{def:irreducibleConsLaw}
An exact or approximate conservation law depending on 
$r$ variables 
is called  \cor{{\bf simple} 
if it can not} be split 
into the sum or the product of two conservation laws
such that at least one of them depends on 
a number  \cor{$r'$ of variables, with $1 \leq r' < r$}.
\end{definition} 

\begin{definition}
\label{def:degenerateSteadyStates}
For $\vk \in \RRpp^r$  a steady state $\vx$ is a positive solution of $\vect{F}(\vk,\vx)=0$
and we denote 
the {\bf steady state variety} by $\mathcal{S}_{\vk}$. A steady state $\vx$ is 
{\bf degenerate} or {\bf non-degenerate}
if the Jacobian $D_{\vx} \vect{F}(\vk,\vx)$ is singular or regular, respectively. 
\end{definition}

Degeneracy of steady states implies that
$\mathcal{S}_{\vk}$ is not discrete. Reciprocally, 
  if the local dimension at a point $\vx_0 \in \mathcal{S}_{\vk}$ is strictly positive, then $\vx_0$ is degenerate (Theorem~10 of \cite{desoeuvres2022complete}).



\begin{definition}
\label{def:completeConsLaw}
A set $$\vect{\Phi}(\vx)=
\transpose{\left( \phi_1(\vx),\ldots,\phi_s(\vx) \right)}$$ of exact
conservation laws is called {\bf complete} if
the Jacobian matrix 
$$\vect{J}_{\vect{F},\vect{\Phi}}(\vk,\vx)=D_{\vx}\transpose{(\vect{F}(\vk,\vx), \vect{\Phi}(\vx) )}$$ 
has rank $n$ for any $\vk \in \RRpp^r$, $\vx\in \RRpp^n$ 
 satisfying $\vect{F}(\vk,\vx)=0$.
The set is called {\bf independent} if the Jacobian matrix of  
$ \transpose{\vect{\Phi}(\vx)}$ with respect to $\vx$
has rank $s$ for any $\vk \in \RRpp^r$, $\vx\in \RRpp^n$ such that 
$\vect{F}(\vk,\vx)=0$. In the case of a set $\vect{\Phi}(\vx)$ of approximate conservation laws completeness is defined with $\vect{F}(\vk,\vx )$
replaced by 
$$\vect{F}^{(1)}(\vk,\vx)=\transpose{\left( f_1^{(1)}(\vk ,\vx), \ldots,
  f_n^{(1)}(\vk ,\vx )\right)}.$$
\end{definition}

If a CRN has a complete set of conservation laws
$\vect{\Phi}(\vx)$, then the set of positive
solutions of $\vect{F}(\vk,\vx)=0$, $\vect{\Phi}(\vx)=\vect{c}_0$ is finite (see Proposition~12 of \cite{desoeuvres2022complete}).


\begin{remark}
If the components of $\vect{\Phi}(\vx)$
are linear in $\vx$ and comes from a stoichiometric matrix,
the set $$ \{ \vect{\Phi}(\vx)=\vect{c}_0 \} \cap \RRp^n$$ is called {\bf stoichiometric compatibility class} or {\bf reaction simplex} \cite{wei1962structure,feinberg1974dynamics}. Stoichiometric compatibility classes of systems with complete sets of linear conservation laws contain 
a finite number of steady states (see Proposition~12 of \cite{desoeuvres2022complete}). 
We note that some authors call completeness of linear conservation laws  {\bf non-degeneracy} \cite{feliu2012preclusion}. 
\end{remark}

\begin{remark}
Since our concern is the number of strictly
positive solutions in $\vx$ of the system $\vect{F}(\vk,\vx)=0$, $\vect{\Phi}(\vx)=\vect{c}_0$, in Definition~\ref{def:completeConsLaw} it would be
more natural to consider the rank of  
$\vect{J}_{\vect{F},\vect{\Phi}}(\vk,\vx)$
on $\mathcal{S}_{\vk} \cap \{\vect{\Phi}(\vx)=\vect{c}_0 \} \cap \RRpp^n$.
In fact, as this rank does not depend on $\vect{c}_0$, it is simpler and equivalent to impose its value
on $\mathcal{S}_{\vk} \cap \RRpp^n$.
\end{remark}

\begin{remark}
The  independent linear conservation laws 
$$\vect{\Phi}(\vx) =(x_1+x_2,\, x_2+x_3) $$  of
Example~\ref{MM} are complete. More precisely, the Jacobian of $\transpose{(\vect{F}(\vx),\vect{\Phi}(\vx))}$, where $\vect{F}(\vx)$ is the vector of right hand sides of \eqref{eqsMMtruncated},
has the $3 \times 3$ minor 
$$M:=\mathrm{det}(D_{\vx} \transpose{(-k_1x_1x_3+k_2x_2,\vect{\Phi}(\vx))}) = - k_2 - k_1 x_1 - k_1 x_3.$$ 
This minor can not be zero for positive $\vx$, $\vk$ on the steady state variety defined by the equation
$-k_1 x_1 x_3 + k_2 x_2 =0$ and therefore the rank 
of $\vect{J}_{\vect{F},\vect{\Phi}}(\vk,\vx) $ is three. 

Furthermore, all stoichiometric compatibility classes
defined by $x_1+x_2=c_{01}$, $x_2+x_3=c_{02}$, $\vx > 0$
contain a unique steady state 
\begin{eqnarray*}
x_1 &=&  \big(k_1(c_{01} - c_{02}) - k_2 + \sqrt{\Delta} \big)/(2 k_1), \\
x_2 &=&  \big(k_1(c_{01} + c_{02}) + k_2 - \sqrt{\Delta}\big)/(2 k_1), \\
x_3 &=&  \big(k_1(c_{02} - c_{01}) - k_2 + \sqrt{\Delta}\big)/(2 k_1),
\end{eqnarray*}
where
$\Delta = (c_{01}-c_{02})^2k_1^2+k_2^2 + 2 k_1k_2(c_{01}+c_{02})$.
\end{remark}

\begin{remark}
The following example shows that linear conservation laws are not always complete.
\end{remark}
\begin{example}\label{example2}
One checks easily that the system 
$$\dot x_1 = 1 - x_1 - x_2, \quad \dot x_2 = x_1 + x_2 -1$$
has the linear conservation law $\Phi(\vx)= (x_1 + x_2)$ and that 
the Jacobian of $\transpose{(\vect{F}(\vx),\vect{\Phi}(\vx) )}$ is constant and  
has rank $1$. Trivially the Jacobian has everywhere rank $1$ and so $\vect{\Phi}(\vx)$ is not complete. 
The explicit solutions of the ODE system 
are $x_1(t) = (1-c_0)t + c_1$ and $ x_2(t)=(c_0-1)t + c_0 - c_1$.
One can easily show that all invariant curves are of the 
form $x_1+x_2=c_0$. We conclude that there are no further first integrals and so the system has no 
complete set of conservation laws. 
We can also note that the intersection of the 
steady state variety $x_1 + x_2 = 1$ with a stoichiometric
compatibility class $x_1+x_2 = c_0$
is either empty or continuous.

\end{example}

\begin{remark}
The notions of completeness and independence in
Definition~\ref{def:completeConsLaw} are effective
and can be tested algorithmically 
using parametric rank computations (see Algorithms~3-5
in \cite{desoeuvres2022complete}).
\end{remark}

\section{Model Reduction Using Approximate Conservation Laws}
\label{sec:model_reduction}
\subsection{Formal Scaling Procedure}\label{sec:formal_scaling}

We consider CRN models described by a system of ODEs as in \eqref{eq:fi}, 
 where we assume that numerical values of the model parameters
are given. We denote these parameters by $\vk^*$.

For the scaling of the model parameters $\vk^*$ we choose a new parameter $\epsilon_*\in]0,1[$ and rescale $\vk^*$
by powers of $\epsilon_*$, that is
\begin{equation} 
k_i^*=\bar k_i \epsilon_*^{e_i}  \quad \mathrm{for} \ 1 \leq i \leq r,
\end{equation}
where the exponents $e_i $ are in $\QQ$. Furthermore, the prefactors $\bar{k}_i$ have 
order $\Ord{\epsilon_*^0}$. More precisely, we have  
\begin{equation} \label{eq:prefactor}
(\epsilon_*)^{\eta} < \bar{k}_i \leq (\epsilon_*)^{-\eta} \quad \mathrm{for} \ 1 \leq i \leq r,
\end{equation}
where $\eta$ is a positive parameter smaller than one. 
A possible choice of the exponents is
\begin{equation} \label{rounding}
e_i = \frac{\text{round} (g\log_{\epsilon_*}(k_i^*))}{g} \in \QQ,  
\end{equation}
where $g$ is a strictly positive integer controlling the precision of the 
rounding step and {\it round} stands for round half down. This choice leads to prefactors $\bar{k}_i$ satisfying \eqref{eq:prefactor} with $\eta=1/(2g)$. 

We further rescale the variables $x_k = y_k \epsilon_*^{d_k}$, where $d_k \in \QQ$, and transform \eqref{eq:fi} into the rescaled system
\begin{equation} \label{rescaled}
S_{\epsilon_*}
\left\{
\begin{array}{rcl}
\dot y_1 &=& \displaystyle \sum_{j=1}^r  \epsilon_*^{\psi_{1j}} S_{1j} \bar k_j^*  \vy^{\alpha_j}, \\
&\vdots& \\
\dot y_n &=&  \displaystyle \sum_{j=1}^r  \epsilon_*^{\psi_{nj}} S_{nj} \bar k_j^*  \vy^{\alpha_j}, 
\end{array}\right.
\end{equation}
where $\psi_{i,j} = e_j +\langle\vect{d},\alpha_j\rangle - d_i$ for $1\leq i \leq n$ and $\vect{d}=(d_1,d_2,\ldots,d_n)$.

The parameter and concentration orders $e_i$ and $d_k$
should be understood as orders of magnitude. For instance, if $\epsilon^*=10^{-1}$, a parameter or concentration of order $d=2$ equals roughly $10^{-2}$. Thus, small orders 
mean large parameters or concentration values. 

At this stage, the variable rescaling is completely arbitrary, but in the next subsection the rescaling
exponents $d_k$ 
will satisfy important constraints.

From now on, we will transform the numerical parameters
$\epsilon_*$, $\vk^*$ into variables $\epsilon$, $\vk=(k_1,\ldots,k_r)$ such that
$k_i = \epsilon^{e_i} \bar k_i$
and consider the family of ODE systems indexed by $\epsilon$
\begin{equation} \label{family}
S_{\epsilon}
\left\{
\begin{array}{rcl}
\dot y_1 &=& \displaystyle \sum_{j=1}^r  \epsilon^{\psi_{1j}} S_{1j} \bar k_j  \vy^{\alpha_j}, \\
&\vdots& \\
\dot y_n &=& \displaystyle \sum_{j=1}^r  \epsilon^{\psi_{nj}} S_{nj} \bar k_j  \vy^{\alpha_j}. 
\end{array}\right.
\end{equation}
Our initial model $S_{\epsilon_*}$ is a member of this family, obtained from \eqref{family} for $\epsilon=\epsilon_*$ and $\vk = \vk^*$. We are interested in characterizing the behaviour of the solutions of this family of ODEs 
when $\epsilon \to 0$. This limit will correspond 
to a reduced model 
that is a good approximation to the initial model if $\epsilon^*$ is close to zero for $\vk = \vk^*$
and if several conditions ensuring the convergence of the 
solutions of $S_\epsilon$ are satisfied (see Section~\ref{sec:reduction}). 
This remark suggests that $\epsilon^*$ should be chosen
as small as possible, smaller than an upper bound depending on the convergence
rate of the solutions of $S_\epsilon$ 
in the limit $\epsilon \to 0$. However, 
rounding effects in \eqref{rounding} imply that
when $\epsilon^*$ is too small, the parameter 
orders $e_i$ become zero 
for all $1 \leq i \leq r $.
In practice, we need to differentiate 
parameters whose ratios are large (or small) enough,
although it is generally difficult to establish
what is the meaning of enough. However, this means that
$\epsilon^*$ must be chosen larger than a 
lower bound depending on the common denominator $d$
and on the minimum (or maximum) ratio of parameters
that need to be consider as distinct.


In order to transform the exponents
$\psi_{ij}$ into positive integer orders
and thus
render the equations suitable for application of singular perturbation theory,  
we further perform the time scaling $\tau = \epsilon^{\mu} t$, where
$$\mu = \min \{\psi_{ij} \mid 1\leq i \leq n,\, 1\leq  j \leq r, \, S_{ij} \neq 0 \} $$ and separate lowest order terms in each equation. 
Letting 
$$a_{ij} = \psi_{ij} - \mu \geq 0, \ a_i = \min \{ a_{ij} \mid 1\leq j \leq r, \, S_{ij} \neq 0 \} \geq 0 \ \mathrm{and} \ a'_{ij}=a_{ij}-a_i > 0,$$ 
we obtain  
\begin{equation}  \label{timerescaled}
 y'_i =  \epsilon^{a_i} (\sum_{a_{ij} =a_i}  
 S_{ij} \bar k_j  \vy^{\alpha_j} + 
 \sum_{a_{ij} \neq a_i}  S_{ij} \bar k_j  \epsilon^{a'_{ij}} \vy^{\alpha_j}) \quad \mathrm{for} \ 1 \leq i \leq n.
\end{equation}


Defining $\delta = \epsilon^{1/o}$, where $o \in \NN$ is the least common
multiple of denominators 
of $a_{ij} \in \QQ$, the equations in \eqref{timerescaled} become
\begin{equation}  \label{timerescaleddelta}
 y'_i =  \delta^{b_i} (\bar f^{(1)}_i(\vkbar,\vy) + \delta^{b'_i}  \bar f^{(2)}_i(\vkbar,\vy,\delta)),
\end{equation}
where 
$$\bar f^{(1)}_i(\vkbar,\vy)  =  \sum_{a_{ij} = a_i}  S_{ij} \bar k_j   \vy^{\alpha_j} \quad \mathrm{and} \quad \bar f^{(2)}_i(\vkbar,\vy,\delta)  =  
\sum_{a_{ij} \neq a_i}  S_{ij} \bar k_j  \delta^{b'_{ij}} \vy^{\alpha_j},$$
and all the powers of $\delta$ are positive integers as follows:
$$ 
b_i=o a_i , \
b'_i = o \min \{ a_{ij} \mid 1\leq j \leq r, \, S_{ij} \neq 0, \, a_{ij} \neq a_i \}- b_i >0, \
b'_{ij} =o a'_{ij} - b_i - b'_i >0 .
$$ 


We call the system obtained by retaining only the minimal order dominant terms in
\eqref{timerescaleddelta} the {\bf truncated system} 
\begin{equation}  \label{truncated}
 y'_i =  \delta^{b_i} \bar f^{(1)}_i(\vkbar,\vy) \quad \mathrm{for} \ 1 \leq i \leq n.
\end{equation}
Let us define the  {\bf truncated stoichiometric matrix} $\vect{S}^{(1)}$ as the matrix whose 
entries are
\begin{equation}\label{S1}
    S_{ij}^{(1)} = 
    \left\{ 
    \begin{array}{ll}
     S_{ij}    &  \text{if }  a_{ij} = a_i, \\
      0   &  \text{if not}.
    \end{array}
    \right.
\end{equation}
It then follows that 
$$\bar f^{(1)}_i(\vkbar,\vy)  =  \sum_{j=1}^r  S_{ij}^{(1)} \bar k_j   \vy^{\alpha_j} \quad \mathrm{for} \ 1 \leq i \leq n.$$
For several calculations it is convenient to return to the variables $\vx$ and the time $t$. 
In the variables $x_i = y_i \delta^{o d_i}$,  the full system reads
\begin{equation}  \label{fullx}
\dot x_i =    f^{(1)}_i(\vk,\vx)+ f^{(2)}_i(\vk,\vx) \quad \mathrm{for} \ 1 \leq i \leq n,
\end{equation}
where $$f^{(1)}_i(\vk,\vx)= \delta^{od_i + b_i + o \mu} \bar f^{(1)}_i(\vkbar,\vy) = \sum_{j=1}^r  S_{ij}^{(1)}  k_j   \vx^{\alpha_j}$$
and $$f^{(2)}_i(\vk,\vx)=  \sum_{j=1}^r  S_{ij}^{(2)}  k_j   \vx^{\alpha_j}$$
with $S_{ij}^{(2)}=S_{ij}-S_{ij}^{(1)}$.
In the same variables, the truncated system is
\begin{equation}  \label{truncatedx}
\dot x_i =    f^{(1)}_i(\vk,\vx) \quad \mathrm{for} \ 1 \leq i \leq n.
\end{equation}
\begin{remark}\label{rem:f1eqbarf1}
It is useful to notice that $f^{(1)}_i$ and $\bar f^{(1)}_i$ are in fact identical as polynomials in
$\RR(\vk,\vx)$ for all
$1\leq i \leq n $. 
\end{remark}
\begin{remark}
The variables $x_i$ (and $y_i$) change significantly on timescales 
given, in the same units as $t$, by the reciprocals of $\dot x_i / x_i = \dot y_i / y_i$. Because
$\dot x_i / x_i$ scales with 
 $\delta^{\mu_i}$, where $\mu_i = b_i + o \mu$, we call
 $\mu_i$
the timescale order of $x_i$ which is also the timescale  of $y_i$. 
Changing time units to the units of $\tau$, $x'_i / x_i = y'_i / y_i = \Ord{\delta^{b_i}}$; in these units
the timescales of the variables $x_i$ and $y_i$ have orders  $b_i$.
\end{remark}

By definition $$\min \{ a_{i} \mid 1\leq i \leq n \} = \min \{\psi_{ij} \mid 1\leq i \leq n,\, 1\leq  j \leq r, \, S_{ij} \neq 0 \} - \mu =0$$ and therefore
$\min \{ b_{i} \mid 1\leq i \leq n \} = 0$ and
up to a relabelling of the variables $y_i$ one can assume that
$b_1=0 \leq b_2 \leq \ldots \leq b_n$. 
As the powers $\delta^{b_i}$ indicate the timescales of the variables $y_i$ (in the same units as $\tau$), the
most rapid variable is $y_1$ and the slowest variable is $y_n$. Of course, several variables can have
the same timescale order, i.e. the same value of $b_i$. Let us regroup
the variables $y_i$ into vectors $\vz_k$. More precisely, $\vz_k =(y_{i_k},y_{i_k+1},\ldots,y_{i_k+n_k-1})$
regroups all variables such that $b_{i_k}=b_{i_k+1}=\ldots=b_{i_k+n_k-1}=b_k$, where $n_1 + n_2 + \ldots + n_m = n$.
We then obtain
\begin{equation}  \label{final}
 \vz'_k =  \delta^{b_k} ( \vfbar^{(1)}_k(\vkbar,\vz) + \delta^{b'_k} \vfbar^{(2)}_k(\vkbar,\vz,\delta)),
\end{equation}
where $b_1=0 < b_2 < \ldots < b_m$, $0< b'_k $,
$\vfbar^{(1)}_k(\vkbar,\vz) \in (\ZZ[\vkbar,\vz])^{n_k}$ and 
$\vfbar^{(2)}_k(\vkbar,\vz,\delta) \in (\ZZ[\vkbar,\vz,\delta])^{n_k}$.

We  regroup variables $x_i$
into vectors $\vx_k = (y_{i_k} \delta^{o d_{i_k}},y_{i_k+1}\delta^{o d_{i_{k+1}}},\ldots,
y_{i_k+n_k-1} \delta^{o d_{i_k+n_k-1}})$ for $1\leq k  \leq m$. Although the variables $x_i$ in the same group
 $\vx_k$ can have different orders, their timescales have the same order $\mu_k$;  the vector 
 $\vx_1$ contains the fastest variables, whereas  $\vx_m$ contains the slowest variables.

\subsection{Tropical Geometry Constraints on the Scaling}
\label{sec:trop}

\subsubsection{General Considerations}

In the previous sections, the orders
$\vect{d}=(d_1, \dots, d_n)$  of the species concentrations 
are chosen \cor{arbitrarily.} However, theories of singular perturbations
and normally hyperbolic invariant manifolds imply that after a fast transient period,
the dynamics of CRNs is confined
to one low dimensional normally hyperbolic
invariant manifold \cite{fenichel1979geometric,radulescu2012reduction}, where it remains for 
a long period, after which it switches to another invariant manifold eventually. Normally hyperbolic invariant manifolds
generalize the notion of hyperbolic fixed points \cite{wiggins1994normally}. Although normally hyperbolic manifolds can have both contracting and expanding directions, here we are only concerned with the fully attractive case. Invariant manifolds with 
expanding directions, such as saddle connections, are important for switching between attractive invariant manifolds, a phenomenon that will not be addressed in this paper. 

Quasi-steady state (QSS) \cor{\cite{bodenstein1913theorie,semenov1956some,segel1989quasi} } and 
quasi-equilibrium (QE) \cite{gorban2001corrections} conditions provide lowest order
approximations to these normally hyperbolic, attractive invariant manifolds.
\cor{Although a QSS manifold may loose normal hyperbolicity at singular points 
\cite{eilertsen2020quasi},
the QSS approximation is valid in the stable region of this manifold.}
For CRNs with rational or polynomial rate functions the QSS and QE conditions read as 
systems of polynomial equations and the lowest order
approximations of invariant manifolds are algebraic 
varieties \cite{rvg}. 

Tropical geometry is the natural framework to study limits of
algebraic varieties depending on one parameter. These limits are based
on the Litvinov--Maslov dequantization of real numbers
leading to degeneration of complex algebraic varieties
into tropical varieties \cite{litvinov2007maslov,viro2001dequantization}.
The name dequantization is inspired by the analogy with Shrödinger's 
dequantization in quantum mechanics where the small parameter is
the Planck's constant $h$ and the limit $h \to 0$ allows to obtain
classical mechanics as a limit of quantum mechanics. 
By dequantization, multivariate polynomials become piecewise-linear functions (min-plus polynomials). Furthermore, null sets of multivariate polynomials become tropical hypersurfaces, defined
as the set of points where the piecewise-linear functions are not smooth, i.e.~where the minimum in the min-plus polynomials 
is obtained for at least  two monomials.

\subsubsection{Orders and Valuations} 
In algebraic geometry, tropical hypersurfaces, prevarieties and varieties establish a modern tool in the theory of Puiseux series \cite{bogart2007computing}.
Lowest orders in Puiseux series are called
valuations \cite{maclagan2009introduction}.

In our problem, let us assume that the 
normally hyperbolic invariant manifold confining the reduced dynamics (defined by the QSS or QE conditions, 
see also Section~\ref{sec:reduction})
can be approximate by the 
Puiseux series solutions 
$\vx(\epsilon)$ of the system 
\begin{equation}\label{eq:np}
f_1(\vx,\epsilon)=0,\, \ldots, \, f_m(\vx,\epsilon)=0,
\end{equation}
where 
$f_i(\vx,\epsilon) = \sum_{i=1}^{n_i} M_{ij}(\vx,\epsilon)$ are polynomials 
and $M_{ij}(\vx,\epsilon)$ are monomials in
$\vx$ and $\epsilon$.

Let us remind that a Puiseux series  is a power series with rational exponents and with a minimum,
eventually negative, exponent. By the multivariate version of the Newton-Puiseux theorem \cite{einsiedler2006non,bogart2007computing,maclagan2009introduction}
the  solutions of \eqref{eq:np} read
$$x_i(\epsilon) = \sum_{k=k_{0i}}^{\infty} c_{ik} \epsilon^{k/n},$$
where $n > 0$ and $k_{0i}$ are integers, and $c_{ik_{0i}}\ne0$. In general  
$x_i(\epsilon) \in \CC$, but here we are interested in real positive solutions, that is when
$x_i(\epsilon) \in \RR_{>0}$.

The valuation of $x_i(\epsilon)$ is the smallest exponent of this Puiseux series, namely
$V(x_i(\epsilon))=k_{0i}/n$.
The valuation can also be defined as the limit 
\begin{equation}
V(x_i(\epsilon)) = \lim_{\epsilon\to 0} \log_{\epsilon}(x_i(\epsilon)).  \label{valuation}
\end{equation}

\begin{remark}\label{rem:rrpp}
Concentration valuations \eqref{valuation} and valuation based scalings can be defined only for strictly positive concentrations. This is 
why we assume that all the concentration variables are in $\RRpp$.
\end{remark}

Valuations satisfy the tropical min-plus algebra:
\begin{eqnarray}\label{minplus}
V(x_1(\epsilon)+x_2(\epsilon)) &=& \min \{ V(x_1(\epsilon)),V(x_2(\epsilon))\}, \notag \\
V(x_1(\epsilon)x_2(\epsilon)) &=& V(x_1(\epsilon))+V(x_2(\epsilon)). 
\end{eqnarray}


According to the rules \eqref{minplus} of the min-plus algebra, valuations of monomials are linear 
functions of valuations of parameters 
$$V(k_j \vx^{\valpha_j} ) = e_j + \langle \vect{d},\valpha_j\rangle$$
and valuations of polynomials are min-plus polynomials, i.e.~piecewise-linear functions
$$V\big(\sum_{j=1}^r S_{ij}k_j \vx^{\valpha_j} \big) = \min\{ e_j + \langle \vect{d}, \valpha_j \rangle \mid 1\le j\leq r,\, S_{ij} \neq 0  \}. $$ 
Here we assume that elements of the stoichiometric matrix $S_{ij}$ have order $\Ord{\epsilon^0}$
and zero valuation.

In this paper, we will use valuations as a tool
for computing orders of magnitude. 

For instance, to each variable one can associate a characteristic time that is the reciprocal of $\D{\log(x_i)}{t} = \frac{\dot x_i}{x_i}$. The  valuation of  $\frac{\dot x_i}{x_i}$, which we call timescale order, reads as
\begin{equation}\label{eq:timescale}
\begin{split}
& V\big(\frac{\dot x_i}{x_i} \big) = V(\dot{x_i}) - d_i = \min\{ e_j + \langle \vect{d},\valpha_j \rangle \mid 1\le j\leq r, \, S_{ij} \neq 0  \} -d_i =\\&= \min\{ \psi_{ij} + \langle \vect{d},\valpha_j \rangle \mid 1\le j\leq r, \, S_{ij} \neq 0  \},
\end{split}
\end{equation}
where  $\psi_{ij}$ is defined by \eqref{rescaled}.


\subsubsection{The Tropical Equilibration Conditions}
For Puiseux series solutions, the valuations represent the orders of magnitude introduced in the previous section. In particular, $V(x_i) = d_i$.

By a theorem of Kapranov \cite{maclagan2009introduction} the valuations are rational points on the tropical hypersurface, defined as the locus of points where the piecewise-linear tropical polynomials 
$$V(f_i(\vx,\epsilon))=\min \{ V (M_{ij}(\vx,\epsilon)) \mid 1\leq j\leq n_i \}$$
are non-differentiable. 
In other words, the tropical hypersuface is the locus of points where the minimal
valuation in $V(f_i(\vx,\epsilon))$ is attained for at least two monomials.

In the case of Puiseux series of 
real positive solutions, which are of interest 
for our problem, in addition to the non-smoothness we need also
a sign condition. In this case, we obtain
the {\bf tropical equilibration condition}: the minimal valuation 
is obtained for at least two monomials of opposite signs
\cite{noel2012tropical,noelgvr,radulescu2012reduction,rvg,samal2016geometric,sgfr}.

Let us note that orders of magnitude satisfy the same properties as valuations.
The smallest order monomials 
are also the largest in absolute valueand therefore the tropical equilibration condition 
 means that each polynomial
equation should contain at least two dominant terms of minimal order and opposite signs. 
These conditions 
were justified  heuristically using the concept of compensation of 
dominant monomials \cite{noel2012tropical,noelgvr,radulescu2012reduction,rvg,samal2016geometric,sgfr}.

 In systems with 
multiple timescales, slow dynamics occurs only when 
for each dominant (i.e.~much larger than
the other) monomial on the right hand side of \eqref{eq:fi},  there is at least
one other monomial of the same order but with opposite sign. 


In singular perturbation theory, only fast variables
satisfy quasi-steady state or quasi-equilibrium equations that reduce the dimension of the dynamics. Slow variables need not to satisfy such constraints. 
Therefore, only the fast
variables and not necessarily the slow variables need to satisfy tropical equilibration conditions. This condition was called partial tropical equilibration in \cite{samal2016geometric,desoeuvres_CMSB2022}.

Therefore, in order to obtain the constraints satisfied
by the orders, we first partition the variables into two disjoint sets: fast equilibrated species variables with indices $F$ and slow non-equilibrated species variables with indices $S$, where  $F \cup S =  \{1,\ldots n\}$ and $F \cap S =\emptyset$. 
This splitting is a priori arbitrary and changes
from one partial equilibration solution to another. 

To summarize, orders of variables have to satisfy several types of constraints:
\begin{description}
\item{\bf 1. Tropical equilibration of fast variables} 

The equilibration condition follows from the properties~\eqref{minplus} of valuations and reads as
\begin{equation}
\label{ttotal}
\min \{ \psi_{ij} \mid 1\leq j \leq r, \, S_{ij} > 0 \} =
\min \{ \psi_{ij} \mid 1\leq j \leq r, \, S_{ij} < 0 \},
\end{equation}
for all $i \in F$.

\item{\bf 2.  Tropical equilibration of exact conservation laws}

Let us assume that the reaction network 
has one or several conservations laws $\phi_i(\vx)$, where $1 \leq i \leq n_c$. These can be
linear, monomial or polynomial conservation laws $$\phi_i(\vx) = \sum_{j=1}^{r_c} z_{ij} \vx^{\vect{\beta_{j}}}(\vx),$$ 
where $z_{ij}$ are integers and $\vect{\beta_{ij}} \in \NN^n$ are multi-indices. 
We consider the case that all $z_{ij}$ are positive and have order $\Ord{\epsilon^0}$.

Then the steady state variety must satisfy the polynomial 
equation $\phi_i(\vx) = \bar c_i \epsilon^{f_i}$. Again, the Kapranov theorem adapted
to real positive solutions leads to the condition
\begin{equation}
\min \{ \langle \vect{d},\vect{\beta_{j}} \rangle \mid 1 \leq j \leq r_c,\,  Z_{ij} \neq 0 \} = f_i,
\label{tcons}
\end{equation}
for all $1 \leq i \leq n_c $.



    
\item{\bf 3. Timescale conditions between fast and slow
    variables}

This condition simply means that the fastest slow variable
is slower than any fast variable 
\begin{equation}\label{ttime}
\min   \{ \psi_{ij} \mid S_{ij} \neq 0, \, i \in S, \, 1 \leq  j  \leq r \}
>
\min   \{ \psi_{ij} \mid S_{ij} \neq 0, \, 1\leq j \leq r \},
\end{equation}
for all $i \in F$.
\end{description}

If $F = \{ 1,\ldots,n\}$, then all variables are equilibrated and no timescale conditions are
required. 
In this case, the set of constraints \eqref{ttotal} and \eqref{tcons} define
{\bf total tropical equilibration solutions}. 
If not all variables are fast, then  
the set of constraints \eqref{ttotal}, \eqref{tcons} and \eqref{ttime} define
{\bf partial tropical equilibration solutions} (see \cite{desoeuvres_CMSB2022}).

\subsubsection{Importance of Concentration Valuations for the Scaling}

In many studies, scalings are applied only to parameters. 
This is because species concentrations are unknown and in this case
it is handy to assume that all the chemical species are present
in similar concentrations.
However, this method has limited validity, as concentrations of different species 
can have different orders of magnitudes.

To illustrate this common mistake, we consider the following example, adapted from \cite{SchneiderWilhelm:00a}.

\begin{example}
Let us consider the mass action CRN
$$\emptyset \xrightarrow{1}  A_2 \xrightarrow{1/\epsilon} A_1 \xrightarrow{1} \emptyset,
\quad A_2 + A_2 \xrightarrow{1/\epsilon} \emptyset.$$ 
The corresponding ODEs are 
$$\dot x_1 = x_2 / \epsilon - x_1, \quad \dot x_2 = - x_2/\epsilon - x_2^2/\epsilon + 1.$$

The model does not have exact conservation laws. 
By imposing total tropical equilibration conditions to the model we get 
$$d_2 - 1 = d_1,\quad \min \{d_2-1,2d_2-1\}=0.$$
These equations have the unique solution $d_1 = 0$, $d_2 =1$, meaning that the valuations of $x_1$ and $x_2$ are different. 
The corresponding scaling is $x_1= y_1 $, $x_2= y_2 \epsilon$ and the rescaled ODEs read
$$\dot y_1 = y_2  - y_1, \quad \dot y_2 = \epsilon^{-1} (- y_2 - \epsilon y_2^2 + 1).$$ 
This scaling shows that  $y_1$ and $y_2$ are slow and fast variables, respectively. The fast truncated
system 
$$\dot y_2 = \epsilon^{-1} (- y_2 + 1)$$ 
has a unique hyperbolic steady state $y_2 = 1$.  

This model is not an example of approximate conservation and standard
singular perturbation theory techniques (quasi-steady state approximation) 
can be applied for its reduction.

Note that \cite{SchneiderWilhelm:00a} used non-scaled concentrations for this example. 
By doing so, the 
truncated system is $\dot x_1 = x_2 / \epsilon$, $\dot x_2 = - x_2/\epsilon - x_2^2/\epsilon$, where 
both variables $x_1$ and $x_2$ are fast. This scaling is different from ours. It leads to the nonlinear approximate conservation law $\phi(x_1,x_2) = x_1 + \log (1 + x_2)$
that was interpreted as a slow variable
 in \cite{SchneiderWilhelm:00a}. 
This scaling, for which the two fast variables $x_1$, $x_2$ are not equilibrated, 
could describe the fast dynamics starting with initial concentrations $x_1(0)$, $x_2(0)$
of the same order, but does not apply to later stages of the dynamics.
\end{example}

\subsection{Reduction of the Michaelis-Menten Model under Quasi-equilibrium Conditions}

The Michaelis-Menten model has been used as a paradigmatic example as it allows to introduce  the
main concepts of model reduction. Both QSS and QE reductions were discussed in \cite{noelgvr,rvg,samal2015geometric}, for a two variable Michaelis-Menten
model obtained from the three variable one by exact reduction, using one exact linear conservation laws. In this subsection
we illustrate a slightly different approach that starts with the three variable model introduced in Example~\ref{MM}.

The scaling used to derive  \eqref{eqsMM} is based on the total tropical equilibration solution
$d_1 = d_2 = d_3 = 0$ and $\delta = \epsilon$.
More general scalings, leading to equivalent results, can be found
in \cite{noel2014tropicalization}.
According to this scaling all three variables $x_1$, $x_2$ and $x_3$ have the same timescale. 
As already shown in the Section~\ref{sec:exdef} the new variables $x_4=x_1+x_2$ and $x_5=x_2 + x_3$ are  
conservation laws (approximate and exact, respectively). 

We use the approximate and exact conservation laws to eliminate two out of the three variables 
$x_1$, $x_2$, $x_3$ and obtain
\begin{eqnarray}
    x_2 &=& x_4 - x_1, \notag \\
    x_3 &=& x_5 - x_4 + x_1.
\end{eqnarray}
The remaining variables satisfy
\begin{eqnarray}\label{eqsMMext}
\dot{x_1} &=& - k_1 x_1 (x_5 - x_4 + x_1) + k_2 (x_4-x_1), \notag \\
\dot{x_4} &=& - \delta k_3 (x_4 - x_1), \notag \\
\dot{x_5} &=& 0.
\end{eqnarray}
System \eqref{eqsMMext} shows that $x_4$ is a slow variable and $x_5$ is a conserved constant variable. 
The constant variable can be turned into a parameter $x_5=k_4$, which leads to
\begin{eqnarray}
\dot{x_1} &=& - k_1 x_1 (k_4 - x_4 + x_1) + k_2 (x_4-x_1), \notag \label{eqsMMF} \\
\dot{x_4} &=& \delta k_3 (x_4 - x_1) . \label{eqsMMS}
\end{eqnarray}

The system \eqref{eqsMMS} is typically a slow-fast system with $x_1$ the fast and $x_4$ the slow variable \cite{tikh,fenichel1979geometric}. 
The fast dynamics is described by 
\begin{eqnarray}\label{eqsMMfast}
\dot{x_1} &=& - k_1 x_1 (k_4 - x_4 + x_1) + k_2 (x_4-x_1)  
\end{eqnarray}
and has two hyperbolic steady states, where only one is positive and 
stable\footnote{all the eigenvalues of the Jacobian matrix computed in this state lie in the complex left half-plane}, namely
\begin{equation} \label{eqsMMsstate}
    x_1^* = \frac{-(k_1(k_4-x_4)+k_2) + \sqrt{(k_1(k_4-x_4)+k_2)^2 + 4k_1k_2x_4}}{2k_1}.
\end{equation}

It follows from singular perturbation theory \cite{tikh,hopp,fenichel1979geometric} that the solutions of system \eqref{eqsMMS} with appropriate initial 
conditions converge for $\delta \to 0$ to the solutions of the differential-algebraic system
\begin{eqnarray}\label{semiexplicit}
0 &=& - k_1 x_1 (k_4 - x_4 + x_1) + k_2 (x_4-x_1), \notag \\
x'_4 &=&  k_3 (x_4 - x_1), 
\end{eqnarray}
where the derivative of $x_4$ is with respect to the time $\tau = t\delta$.

Using the solution \eqref{eqsMMsstate}, the semi-explicit differential-algebraic system 
\eqref{semiexplicit}
can be transformed into the reduced ODE
\begin{equation}
  x'_4 =  k_3 \left( x_4 - \frac{-(k_1(k_4-x_4)+k_2) + \sqrt{(k_1(k_4-x_4)+k_2)^2 + 4k_1k_2x_4}}{2k_1} \right). 
\end{equation}
Figure~\ref{fig:figure1} illustrates the accuracy of this reduction.

\begin{figure}[ht!]
    \centering
    \includegraphics[width=\linewidth]{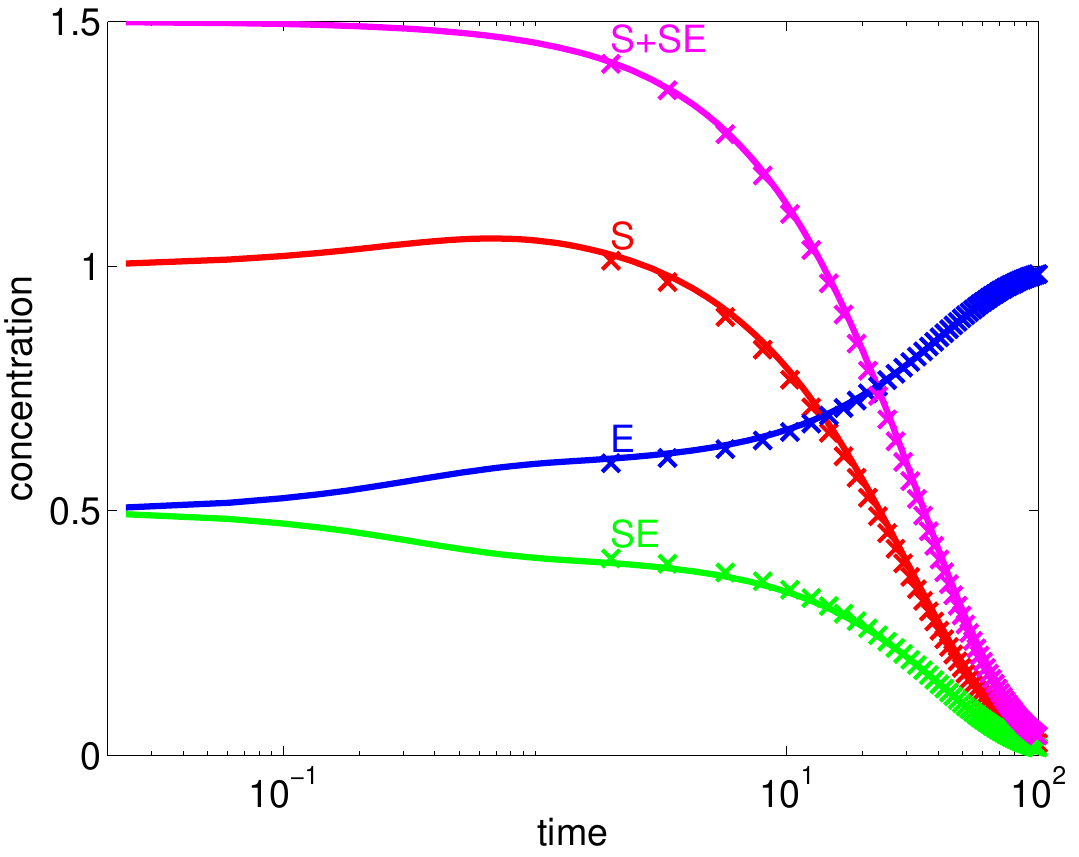}
    \caption{Comparison of numerical solutions obtained with the full Michaelis-Menten model (continuous lines) and 
    with the reduced model (crosses). The variables $S$, $E$ and $SE$ are fast and slaved 
    (in the reduced model, their values are obtained as solutions of algebraic equations) and the approximate conservation law $S+SE$ is slow. The initial values are such that $E$ and $S$ have  concentrations of the same order of magnitude; this is 
    not compatible with the quasi-steady state (QSS) approximation \cite{segel1989quasi,samal2015geometric}
    of the MM mechanism but is compatible with the quasi-equilibrium (QE) approximation that has been used here.
    As discussed in \cite{noelgvr,rvg,samal2015geometric} the QE and QSS approximations correspond to scalings obtained from different tropical
    equilibration solutions (the total equilibration for QE and a partial equilibration solution for QSS).
    }
    \label{fig:figure1}
\end{figure}

\subsection{Approximate Conservation Laws as Slow Variables}\label{sec:acslow}

We have seen in the previous section that the approximate 
linear conservation laws of the Michaelis-Menten model are
either exact conservation laws or slow variables.
In this section we show that this property is true in general for
any polynomial CRN model of type \eqref{eq:fi} and for all linear, monomial, 
or polynomial approximate conservation laws.





\subsubsection{Linear Approximate Conservation Laws as Slow Variables}

Linear approximate conservation laws correspond to ``pools'' of species that are conserved by fast cycling 
reactions. The exact and approximate linear conservation laws usually correspond to the total
number of copies of a certain type of molecule in the pool.
For instance, in Example~\ref{MM}, there
is an exact and an approximate linear conservation law corresponding to 
the total numbers of enzyme and substrate molecules, respectively. 
The fast part of the dynamics ends with the equilibration of all species in the fast pool. This state is named quasi-equilibrium (QE) \cite{gorban2010asymptotology}.


Let us consider an approximate linear conservation law
$$\phi(\vx)=\sum_{i=1}^n c_i x_i, \quad \mathrm{where} \ c_i \in \NN ,$$ 
which is conserved by the truncated ODE system \eqref{truncatedx}. 
Note that in the linear combination defining $\phi(\vx)$ we also admit that some of the coefficients 
$c_i$ are zero. 
Let $I = \{i \mid c_i \neq 0\}$ denote the set of indices
of species involved in the fast pool quantity.
Furthermore, in most applications the $c_i$ are small positive integers and we will assume that $c_i = \Ord{1}$. 
Using the Definitions~\ref{def:exactandapproxcons} and \ref{def:irreducibleConsLaw} we obtain the following result.


\begin{theorem}\label{th:slowlinear1}
  If $\phi(\vx)$ is  \cor{a simple} linear approximate conservation law, then 
  the variable $q = \phi(\vx)$ is either constant or slower
   than all the variables $\{ x_i \mid i \in I \}$
  of the system  \eqref{eq:fi}. Furthermore, if the timescales of all $x_i$ with $i \in I$ have the
  same order, then the concentrations of these variables have the same orders. 
\end{theorem}

Before we start with the proof of the theorem, we show the following lemma.
\begin{lemma}\label{lem:1}
Let $$\phi(\vx)=\sum_{i=1}^{n} c_i x_{i}$$ be  \cor{a simple} linear approximate conservation law and let $I$ be defined as above.
Then all the polynomials $f_i^{(1)}(\vk,\vx)$ in the truncated system \eqref{truncatedx}
have the same order in $\delta$, i.e.~$od_i + b_i=o d_{i'} + b_{i'}$ for all $i,\, i' \in I$. If furthermore, the timescales of the
variables $x_i$ have the same order in $\delta$ for all $i \in I$,
i.e.~$b_i=b_{i'}$ for all $i,\,i'\in I$,
then the concentrations $x_i$ have also the same
order in $\delta$, i.e.~$d_i=d_{i'}$ for all $i,\, i'\in I$.
\end{lemma}

\begin{proof}[Proof of Lemma~\ref{lem:1}]
In the variables $\vx$ the truncated system reads as  (see \eqref{truncatedx})
$$\dot x_i = \delta^{od_i} \dot y_i = f_i^{(1)} (\vk,\vx), \quad \mathrm{where} \  f_i^{(1)} (\vk,\vx) = \delta^{od_i+b_i+o\mu} \bar f_i^{(1)} (\vkbar,\vy).$$ 
Since $\phi(\vx)$ is conserved by system \eqref{truncatedx},
one has $$\sum_{i=1}^r c_i  f_i^{(1)} (\vk, \vx) = 0$$ for all $\vk$, $\vx$. 
This can only be satisfied if for all $i\in I$ there is 
at least one $j \in I$ such that $f_i^{(1)} (\vk,\vx)$ and $f_j^{(1)} (\vk,\vx)$
have a common monomial. Since $q$ is \cor{simple,}
then
for all $i,\, j\in I$ with $i\neq j$
either  $f_i^{(1)} (\vk,\vx)$ and $f_j^{(1)} (\vk,\vx)$ share a monomial
or there is finite sequence $i=i_0,\, i_1,\, \ldots,\, i_k = j$ such that
$f_{i_l}^{(1)} (\vk,\vx)$ and $f_{i_{l+1}}^{(1)} (\vk,\vx)$ share a monomial for
$0 \leq l \leq k$. 
Since by the definition of the truncated
system all the monomials in $f_i^{(1)} (\vk,\vx)$ have the same order, it follows that all the
polynomials  $f_i^{(1)} (\vk,\vx)$ for some $i\in I$ have the same order $\nu=od_i+b_i+o\mu$. 
The timescale order $\mu_i$ of $x_i$ is the order of $$\frac{\dot x_i}{x_i}= \frac{ f_i^{(1)} (\vk,\vx) }{x_i},$$
namely $\mu_i = \nu - o d_i = b_i+o\mu$. Thus, if all $b_i$ are equal, then all 
 $d_i$ are equal, for $i \in I$. 
\end{proof}

\begin{proof}[Proof of Theorem~\ref{th:slowlinear1} ]
From the definition of $q$ it follows that $$\dot q =  \sum_{i\in I}c_i \dot x_i.$$
Since $q$ is conserved by the truncated system \eqref{truncatedx}, we have that 
\begin{equation}\label{eq:lin1}
\sum_{i=1}^n c_i \delta^{b_i+od_i+o\mu} \bar f_i^{(1)}(\vkbar,\vy) = 0
\end{equation}
 for all $\vkbar$, $\vy$, $\delta$ and so
\[
\dot{q} = \sum_{i=1}^n c_i \delta^{b_i+od_i+o\mu} \left(\bar f_i^{(1)}(\vkbar,\vy)+\delta^{b'_i} \bar f_i^{(2)}(\vkbar,\vy,\delta) \right) =
\sum_{i=1}^n c_i \delta^{b_i+b'_i+ od_i+ o\mu}\bar f_i^{(2)}(\vkbar,\vy, \delta).
\]
Thus $\dot q=0$ if $f_i^{(2)}(\vkbar,\vy, \delta)$ vanishes identically. In this case the approximate conservation law is also an exact one. If this is not the case,
we can define $\mu_q$ as the order of the timescale of $q$, i.e.~the order of $\frac{\dot q}{q}$.
The order of $q$ is  $\min \{ o d_i \mid i\in I \}$ and so
\[
\mu_q = \min  \{ b_i + b_i' + od_i + o\mu \mid  i \in  I \} - \min \{ o d_i \mid i\in I \}.
\] 

We prove now that $q$ is slower than all $x_i$ with
$i\in I$, that is, we need to check that
$b_i + o \mu< \mu_q$ for all $i \in I$.
These conditions are equivalent to 
\[
b_i < \min  \{ b_j + b_j' + od_j \mid  j \in  I \} - \min \{ o d_j \mid j\in I \} \quad \mathrm{for} \ \mathrm{all} \  i\in I.
\]
According to the Lemma~\ref{lem:1}, 
$b_i = \nu -o\mu - o d_i$ for all $i\in I$.
Obviously, it is enough to prove that
\[
\nu - o\mu -\min \{ od_j \mid j\in I \} < 
\min  \{ b_j + b_j' + od_j \mid  j \in  I \} - \min \{ o d_j \mid j\in I \},
\]
which leads us to the inequation
$\nu -o \mu < \min  \{ b_j + b_j' + od_j \mid  j \in  I \}$.
Since $\nu = b_j + od_j + o \mu$ and $b_j' > 0$ for all $j \in I$,
all the above inequalities are satisfied.

The second part of the Theorem follows from the Lemma~\ref{lem:1}.
\end{proof}

Theorem~\ref{th:slowlinear1} suggests that there is a link between timescales 
and concentrations of species contributing to a linear approximate conservation law. This link can be made more precise by using the vectors $\vx_k$ introduced in Section~\ref{sec:formal_scaling}, which 
group variables with the same timescale orders $\mu_k$ for   
$1\leq k \leq m$. We have the following structure theorem for approximate linear conservation laws.
\begin{theorem}\label{th:structurelinear}
If $\phi(\vx_1,\ldots,\vx_l)$ is  \cor{a simple} linear approximate conservation law depending on variables having timescale orders smaller than or equal
to $\mu_l$, then 
\begin{equation}\label{eq:lindec1}
\phi(\vx_1,\ldots,\vx_l) = \sum_{k=1}^l \delta^{d_q + \mu_l - \mu_k} \langle \vc_k,\vz_k \rangle,
\end{equation}
where $\mu_1 < \mu_2 < \ldots < \mu_l$, $\vc_k \in \RR^{n_k}$  and $d_q$ is the 
order of the variable $q= \phi(\vx_1,\ldots,\vx_l)$. Thus, $$\phi(\vx_1,\ldots,\vx_l) = \phi^{(1)}(\vx_l) + 
\phi^{(2)}(\vx_1,\ldots,\vx_{l-1}),$$ where
$$\phi^{(1)}(\vx_l) = \delta^{d_q } \langle \vc_l,\vz_l \rangle \quad \mathrm{and} \quad \phi^{(2)}(\vx_1,\ldots,\vx_{l-1}) = \ord{\delta^{d_q }}.$$
$\phi^{(1)}(\vx_l)$ contains the dominant lowest order terms of $\phi$. 
In other words, variables of the same timescale orders contribute to
terms of the same order in the conservation law; the dominant terms in the
conservation law depend only on the slowest variables $\vx_l$. 
\end{theorem}
\begin{proof}
From Equation~\eqref{eq:lin1} it follows that 
\begin{equation}\label{eq:orders}
\mu_i + o d_i =  \mu_l + d_q 
\end{equation}
for all $i$ with $ c_i \neq 0$, where $d_q = \min \left\{ o d_i \mid c_i \neq 0 \right\}$.
From regrouping the coefficients $c_i$ into vectors $\vc_k$ corresponding to variables of the same timescales, we
obtain 
\begin{equation}\label{eq:lindec2}
\phi = \sum_{k=1}^l \langle \vc_k,\vx_k \rangle = \sum_{k=1}^l \delta^{o d_k} \langle \vc_k,\vz_k \rangle.
\end{equation}
Finally, Equation \eqref{eq:lindec1} follows from \eqref{eq:lindec2} and \eqref{eq:orders}.
\end{proof}

\begin{remark}
A set of independent \cor{simple} approximate conservation laws 
 can be obtained from the truncated stoichiometric matrix $\vect{S}^{(1)}$
by using algorithms for the computation of a basis of \cor{simple} 
vectors 
of the left kernel of a given integer coefficient matrix, \cor{i.e. vectors that
can not be decomposed as a sum of two non-zero kernel vectors having more zero elements} \cite{schuster1991determining}.
\end{remark}

\begin{example}  
  Consider the chemical reaction network
\[
  \begin{tikzcd}[row sep=huge, column sep=huge, every arrow/.append style={shift left=.75}]
   x_{1} \arrow[d, bend left,"k_{1}"] & x_{3} \arrow[d, bend left, "k_{3}"] \\
   x_{2} \arrow[u, bend left, "k_{2}"] \arrow[ur, "k_{5}"] & x_{4}
   \arrow[u, bend left, "k_{4}"]
\end{tikzcd}
\]

If the dynamics of this reaction network is of mass-action form, then
it is given by the system of ODEs
\begin{eqnarray*}
    \dot x_{1} &=& k_{2}x_{2}-k_{1}x_{1},\\
    \dot x_{2} &=& k_{1}x_{1}-(k_{2}+k_{5})x_{2},\\
    \dot x_{3} &=& k_{5}x_{2} + k_{4} x_4 - k_{3}x_{3}, \\
    \dot x_{4} &=& k_{3}x_{3}  - k_{4}x_{4}.
\end{eqnarray*}
Let us assume that the parameter orders are  $e_1=e_2=e_3=e_4=0$, $e_5=1$.
Then the total tropical equilibrations are solutions of the
system $d_1=d_2$, $\min\{ d_4, \, d_2+1\}=d_3$, $d_3=d_4$, that is 
$d_1=d_2 \geq d_4-1$, $d_3=d_4$. Assuming that the concentration orders are given by the
 total tropical equilibration $d_1=d_2=-1$ and $d_3=d_4=-2$,
 we obtain the rescaled system
 \begin{eqnarray*}
    \dot y_{1} &=& \bar k_{2}y_{2} - \bar k_{1}y_{1},\\
    \dot y_{2} &=& \bar k_{1}y_{1}-(\bar k_{2}+ \epsilon \bar k_{5})y_{2},\\
    \dot y_{3} &=& \epsilon^2 \bar k_{5}y_{2} + \bar k_{4} y_4 - \bar k_{3}y_{3}, \\
    \dot y_{4} &=& \bar k_{3}y_{3}  - \bar k_{4}y_{4}.
 \end{eqnarray*}
Since $\epsilon$ occurs only with integer powers we have $\delta=\epsilon$.
All species $x_i$ for $1\leq i \leq 4$
have the same timescale orders $\mu_i=0$ .
The truncated system is 
\begin{eqnarray*}
    \dot x_{1} &=&  k_{2}x_{2} -  k_{1}x_{1},\\
    \dot x_{2} &=&  k_{1}x_{1}-  k_{2} x_{2},\\
    \dot x_{3} &=&  k_{4} x_4 -  k_{3}x_{3}, \\
    \dot x_{4} &=&  k_{3}x_{3}  -  k_{4}x_{4} 
\end{eqnarray*}
and the truncated stoichiometric matrix reads as
$$\vect{S}^{(1)} = 
\begin{pmatrix}
-1 & 1 &0 &0 \\
1 & -1 &0 &0  \\
0 & 0  &-1& 1 \\
0 & 0  & 1 &-1
\end{pmatrix}.
$$
The truncated system has two \cor{simple} conservation laws
$\phi_1(\vx) = x_1+x_2$ and $\phi_2(\vx)=x_3+x_4$. These correspond to the species 
pools $q_1=x_1+x_2$ and $q_2=x_3 + x_4$ that have timescale orders
$\mu_{q_1} = 1$ and $\mu_{q_2} = 2$, respectively.
Thus, $q_1$ and $q_2$
are slower than the species $x_i$ with $1\leq i \leq 4$.
Note that the species concentrations in these
pools have equal orders $d_1=d_2$ and $d_3=d_4$, consistent
with the fact that in \cor{simple} pools, species with the same timescales have
the same concentration orders (see Lemma~\ref{lem:1}).
This CRN has also the exact conservation law $\phi_1(\vx)+\phi_2(\vx)$, but this conservation law is
not \cor{simple}. 
\end{example}

\subsubsection{Monomial Approximate  Conservation Laws as Slow Variables}


We consider now a monomial conservation law
$$\phi(\vx)=\prod_{i=1}^n x_i^{m_i}$$ of 
the truncated system \eqref{truncatedx}.
As in Section~\ref{sec:exdef}, 
we admit that some of the exponents $m_i$ can be zero and define
$I = \{ i \mid 1\leq i \leq n, \, m_i \neq 0 \}$.
This means that the variables $x_i$ with $i \notin I$ do not appear in the conservation law.

\begin{theorem}\label{th:monomialslow}
  If $\phi(\vx)$ is  \cor{a  simple} monomial approximate
  conservation law, then the variable $q=\phi(\vx)$ is slower 
  than all the variables
  $\{ x_{i} \mid i \in I \}$ of system \eqref{eq:fi}. Furthermore, all
  $x_i$ with $ i \in I$ have the same timescale orders.
\end{theorem}

\begin{proof}
Let us note that $$\dot q = q \sum_{i\in I}m_i \frac{\dot x_i}{x_i}.$$
Since $\phi(\vx)$ is conserved by the truncated system \eqref{truncatedx}, we have that
\begin{equation}\label{conditionmonomial}
\sum_{i\in I}  \frac{m_i}{x_i} 
 f_i^{(1)}(\vk,\vx) = 0,
 \end{equation}
for all $\vk$, $\vx$.
As $\frac{ f_i^{(1)} }{x_i}$ is a sum of rational 
monomials, \eqref{conditionmonomial}
 is only satisfied if for any $i\in I$ there is  $j\in I$ such that
$\frac{ f_i^{(1)} }{x_i}$ and $\frac{ f_j^{(1)}}{x_j}$
share a common monomial. Since $\phi(\vx)$ is \cor{simple}, for all
$i,\, j\in I$ either $\frac{ f_i^{(1)}}{x_i}$ and $\frac{ f_j^{(1)}}{x_j}$ share a common
monomial or there is a sequence  $i=i_0,\, i_1,\, \ldots,\,i_k=j$
such that $\frac{ f_{i_l}^{(1)}}{x_{i_l}}$ and $\frac{ f_{i_{l+1}}^{(1)}}{ x_{i_{l+1}} }$ share
a common monomial for $0 \leq l\leq k-1$. Hence, the orders 
 of $\frac{f_i^{(1)}(\vk,\vx)}{x_i}$ are the 
same for all $i \in I $. As
the timescale orders of $x_i$ are the orders of $\frac{f_i^{(1)}(\vk,\vx)}{x_i}$, it follows that all
$x_i$ with $i\in I$ have the same timescale orders.



In order to compute the timescale of $q$ we use
$$\frac{\dot q}{q} = 
\sum_{i\in I} m_i \frac{\dot x_i}{x_i} = 
\sum_{i\in I} m_i \delta^{ b_i + o \mu } 
\frac{ \bar f_i^{(1)}(\vkbar,\vy) + \delta^{b'_i}
\bar f_i^{(2)}(\vkbar,\vy,\delta)   }{y_i} =
\sum_{i\in I} m_i \delta^{ b_i + b'_i + o \mu } 
\frac{ f_i^{(2)}(\vkbar,\vy,\delta)   }{y_i},
$$
where the last equality follows from the fact that 
$\phi(\vx)$ is conserved by \eqref{truncatedx}. 

Let us denote by $\mu_q$ the order of the timescale of $q$, i.e.~the order of $\frac{\dot q}{q}$.
Assuming that all $m_i$ are small integers of order $\Ord{\delta^0}$, it
follows that
$$ \mu_q =  \min \{b_i+b'_i+ o\mu \mid i\in I \}.$$
To prove that $q$ is slower than each of the variables $\{x_i \mid i\in I\}$, we need to show that for all $i\in I$ we have $\mu_i<\mu_q$. Thus we need to prove that $b_i<\min \{b_j+b'_j \mid j\in I  \}$ for all $ i\in I$.
As for all $i,\, j\in I$ we have $b_i=b_j$ and $b'_j>0$, it follows that $ b_i=\min\{b_j \mid j \in I \} <\min\{b_j+b'_j \mid j\in I \}$ for all $ i\in I$.

\end{proof}


\begin{example}\label{example3}
The model 
$$\dot x_1 = x_1 (x_2 - x_1 ) - \delta x_1, \quad  \dot x_2 = x_2 (x_1 - x_2 )$$
is a mass action network described by 
$$A_1 + A_2 \xrightarrow{1} 2A_2,\quad A_1+A_2 \xrightarrow{1} \emptyset,\quad A_2 + A_2 \xrightarrow{1} \emptyset, \quad
A_1 \xrightarrow{\delta} \emptyset.$$
The truncated system 
$$\dot x_1 = x_1 (x_2 - x_1 ) ,\quad  \dot x_2 = x_2 (x_1 - x_2 )$$
has a continuous steady state variety $x_1=x_2$ and on which its Jacobian is singular.
This model has $\phi(x_1,x_2) =x_1 x_2$ as an approximate monomial \cor{simple} conservation law. 
The intersection of the steady state variety with the set $\phi = c_0$ is the point 
$x_1=x_2=c_0/2$. The monomial conservation law is complete,  
since the $2\times 2$ minor of the Jacobian
$$\mathrm{det}(D_{\vx} (x_1(x_2-x_1),x_1x_2)^T )  = - 2x_1^2$$
does not vanish for $x_1 > 0$. 
Including $q = x_1 x_2$ among the variables leads to the 
ODE system
$$\dot x_1 = x_1 (x_2 - x_1 ) - \delta x_1,\quad \dot x_2 = x_2 (x_1 - x_2 ), \quad \dot q = - \delta q.$$
We note that in agreement with the Theorem~\ref{th:monomialslow} $x_1$ and $x_2$ have the same timescale order and
that $q$ is a slower variable.
\end{example}

\subsubsection{Polynomial Approximate Conservation Laws as Slow Variables}


Let $$\phi(\vx)=\sum_{j=1}^r c_j \boldsymbol{x}^{\boldsymbol{m}_j}$$ be an approximate polynomial conservation law, that is conserved by system \eqref{truncatedx}, where $c_j \in  \RR \setminus \{0\}$ and $\boldsymbol{m}_j=(m_{1j},\dots , m_{nj}) \in\NN^n$.
Let $I = \{ i \mid m_{ij} \neq 0 \text{ for some $j$ with } 1\leq j \leq r\}$ so that $\phi(\vx)$ depends only 
on the variables $x_i $ with $i\in I$.

\begin{theorem}\label{th:polyslow}
  If $\phi(\vx)$ is  \cor{a simple} polynomial approximate conservation law, then the variable
   $q =\phi(\vx)$ is
  slower than all variables
  $x_i$ with $i \in I$ of system~\eqref{eq:fi}. Furthermore, if 
  the timescales of the variables $x_i$ have the same order in 
  $\delta$, i.e.~$b_i=b_{i'}$ for all $i,\, i'\in I$, then the 
  monomials in $\phi(\vx)$ have also the same order
  in $\delta$, i.e.~$\langle \vect{d},\boldsymbol{m}_j \rangle=\langle\vect{d},\boldsymbol{m}_{j'}\rangle$
  for all $1 \leq j,\, j' \leq r  $ such that $m_{i,j} \neq 0$ and $m_{i',j'} \neq 0$ for some
  $i,\, i' \in I$.
\end{theorem}

\begin{proof}
We note that $$\dot q = \sum_{j=1}^r c_j x^{\vect{m}_j} ( \sum_{i=1}^r m_{ij} \frac{\dot x_i}{x_i}) = 
\sum_{i=1}^n \frac{\dot x_i}{x_i} (\sum_{m_{i,j}\neq 0}    m_{ij}    c_j \vx^{\vect{m}_j}  ) .$$ 
Let us define the sums of rational monomials 
$$E_{i,j}(\vk,\vx) = \frac{f_i^{(1)}(\vk,\vx) }{x_i}     m_{ij}    c_j
 \vx^{\vect{m}_j}.$$
Then the expressions $E_{i,j}(\vk,\vx)$ have the orders 
$$b_i + o \mu + o\langle\vect{d},\vect{m}_j\rangle.$$
As $\phi(\vx)$ is conserved by \eqref{truncatedx} it follows that
\begin{equation}\label{eq:sumEij}
  \sum_{1\leq i\leq n,\,1\leq j\leq r,\,m_{i,j}\neq 0} E_{i,j}(\vk,\vx) 
=0 
\end{equation}
for all $\vk$, $\vx$. This is only possible if for any pair $(i, j)$ with $m_{i,j}\neq 0$, there is a pair $(i', j')$ with
 $m_{i',j'}\neq 0$ such that
$E_{i,j}$ and $E_{i',j'}$ share a common monomial. 
Since $\phi(\vx)$ is \cor{simple}, for all pairs $(i,j)$ and $(i',j')$ with $m_{i,j}\neq 0$ and $m_{i',j'}\neq 0$,
either $E_{i,j}$ and $E_{i',j'}$ share a common monomial or there is a sequence of pairs $$(i,j)=(i_0,j_0),\, (i_1,j_1),\, \ldots,\, (i_k,j_k)=(i',j')$$ 
such that 
$E_{i_l,j_l}$ and $E_{i_{l+1},j_{l+1}}$ share a common monomial for $0 \leq l \leq k-1$.
Thus, the expressions $E_{i,j}$
have the same order for all pairs $(i, j)$ with $m_{i,j}\neq 0$, i.e.~we have that 
$$b_i + o \mu + o\langle\vect{d},\vect{m}_j\rangle = b_{i'} + o \mu + o\langle\vect{d},\vect{m}_{j'}\rangle$$
for all pairs $(i,j)$ and $(i',j')$ with $m_{i,j}\neq 0$ and $m_{i',j'}\neq 0$.
In particular, if all variables $x_i$ have the same timescale order $b_i$
for $i\in I$, it follows that the scalar products $\langle \vect{d},\vect{m}_j \rangle$ are equal 
for all $1\leq j < r$ with $m_{i,j}\neq 0$ for some $i \in I$. This proves the second part of the theorem.

As $q$ is conserved by system \eqref{truncatedx}, it follows that
$$\dot q = \sum_{i,j,m_{i,j}\neq 0} \delta^{b_i + b'_i +o \mu}\frac{\bar f_i^{(2)}(\vkbar,\vy,\delta) }{y_i}     m_{ij}    c_j
\delta^{o\langle\vect{d},\vect{m}_j\rangle} \vy^{\vect{m}_j}.$$
The timescale order of $q$ is 
$$\mu_q = \min \{ b_i + b'_i + o \mu + o\langle\vect{d},\vect{m}_j \rangle \mid m_{ij} \neq 0 \} - \min \{ o\langle\vect{d},\vect{m}_j\rangle \mid 1\leq j\leq r \}.$$
Let $i\in I$ and $j$ such that $m_{i,j}\neq 0$.
Using $$b_i + o \mu + o\langle\vect{d},\vect{m}_j \rangle= \min \{ b_i + o \mu + o<\vect{d},\vect{m}_j\rangle \mid m_{i,j} \neq 0 \}$$
and $b'_i > 0$, we obtain 
$$b_i + o \mu = \min \{ b_i + o \mu + o\langle\vect{d},\vect{m}_j\rangle \mid m_{i,j} \neq 0 \} - o\langle\vect{d},\vect{m}_j\rangle < \mu_q,$$
meaning that $q$ is slower than all  $x_i$ with $i\in I$.

\end{proof}
 
As in the case of linear conservation laws, i.e.~Theorem~\ref{th:structurelinear},
we have a structure theorem for approximate polynomial conservation laws.
\begin{theorem}\label{th:structurepolynomial}
If $\phi(\vx_1,\ldots,\vx_l)$ is  \cor{a simple} polynomial approximate conservation law depending on variables having timescale orders smaller than or equal
to $\mu_l$, then 
\begin{equation}\label{eq:poldec1}
\phi(\vx_1,\ldots,\vx_l) = \sum_{k=1}^l \delta^{d_q + \mu_l - \mu_k} \langle \vc_k, \vect{\varphi}(\vz_k)\rangle,
\end{equation}
where $\mu_1 < \mu_2 < \ldots < \mu_l$, 
$\vc_k \in \RR^{n_k}$,
$\vect{\varphi}(\vz_k) \in \RR^{n_k}[\vz_k]$  and $d_q$ is the 
order of the variable $q= \phi(\vx_1,\ldots,\vx_l)$. Thus, 
$$\phi(\vx_1,\ldots,\vx_l) = \phi^{(1)}(\vx_l) + \phi^{(2)}(\vx_1,\ldots,\vx_{l-1}),$$ 
where
$$\phi^{(1)}(\vx_l) = \delta^{d_q } \langle \vc_l, \vect{\varphi}_l(\vz_l) \rangle \quad \mathrm{and} \quad \phi^{(2)}(\vx_1,\ldots,\vx_{l-1}) = \ord{\delta^{d_q }}.$$
$\phi^{(1)}(\vx_l)$ contains the dominant lowest order terms of $\phi$.
In other words, variables of the same timescale orders contribute to
monomials of the same order in the conservation law; the dominant monomials in the
conservation law correspond to the slowest variables. 
\end{theorem}
\begin{proof}
It follows from \eqref{eq:sumEij} that 
\begin{equation}\label{eq:orderspol}
    \mu_i + o \langle \vect{d},\vect{m}_j \rangle =   \mu_l + d_q 
\end{equation}
for all pairs $(i,j)$ with  $c_j \neq 0$ and $m_{i,j}\neq 0$,
where $d_q = \min \left\{ o \langle \vect{d},\vect{m}_j \rangle \mid c_j \neq 0\right\}$.
From \eqref{eq:orderspol} we obtain that the variables $x_i$, $x_{i'}$ appearing in the same monomial $c_j \vx^{\vect{m}_j}$ of $\phi$ must
have the same timescale orders, i.e.~we have $\mu_i = \mu_{i'}$.
By regrouping the coefficients $c_i$ into vectors $\vc_k$ corresponding to variables of the same timescales and using again \eqref{eq:orderspol}, we
obtain \eqref{eq:poldec1}.
\end{proof}

\subsection{The Model Reduction Algorithms}
\label{sec:reduction}
To wrap up all the above developed concepts,  
 we propose in this section several model 
 reduction algorithms. These take into
 account approximate conservation laws and
  are applicable to CRN models with multiple
timescales and polynomial rate functions.
We consider two types of reductions:
\begin{enumerate}
\item[(i)]  Reduction at the slowest timescale.
\item[(ii)]  Nested reductions at 
intermediate timescales. 
\end{enumerate}

In case (i) all the variables except the slowest one are eliminated during the reduction
procedure. The reduced model is an ODE
for the slowest variable. 
The elimination of fast variables proceeds
hierarchically, the fastest variables being eliminated first.  

In case (ii) all fast variables up to
the $(l-1)$-th fastest one satisfy
polynomial quasi-steady state equations and can be
eliminated.
The remaining variables 
satisfy a reduced system of ODEs. 
The reduced dynamics takes place on the  
 normally hyperbolic invariant manifold
that is close to the critical manifold defined 
 by the quasi-steady state equations. 
 Changing $l$ from $l=1$ to $l=m$ one obtains
$m$ nested attractive normally hyperbolic 
invariant manifolds along which the reduced
dynamics evolves at successively slower timescales. 
Of course (i) follows from (ii) with $l=m$.

For both types of reduction (i) and (ii) 
the elimination of the fast variables is possible 
only if the
truncated system at the $k$-th timescale (defined by the vector fields $\vfbar_1^{(1)},\ldots,\vfbar_k^{(1)}$)  
has non-degenerate steady states (with $1\leq k \leq m$ in case 
(i) and $1\leq k \leq l$ in case (ii)).
When there are approximate conservation
laws which are conserved by the truncated system,
the non-degeneracy condition is not fulfilled and 
the standard reduction algorithm 
proposed in \cite{kruff2021algorithmic}
does not apply. 
Our solution to this problem is to 
add approximate conservation laws to the
set of variables, eliminate some of the fast variables and 
obtain a modified system that has no approximate conservation laws and 
satisfies the hyperbolicity  
condition.

\subsubsection{The Slowest Timescale Reduction}

As in \cite{kruff2021algorithmic} 
we introduce the small parameters $\delta_{l-1} = \delta^{b_{l}-b_{l-1}}$, for $2\leq l \leq m$
and the vector $\vect{\bar \delta} =(\delta_1,\ldots,\delta_{m-1})$.
Let us change the time variable to  $\tau'=\tau \delta_1 \delta_2 \ldots \delta_{m-1}$, the slowest timescale of the model. Then system \eqref{final} becomes
\begin{eqnarray}  \label{slowest}
\delta_1 \delta_2 \ldots \delta_{m-1} \vz'_1 &=&    \vfbar^{(1)}_1(\vkbar,\vz) +   \vgbar_1(\vkbar,\vz,\vect{\bar \delta}), \notag\\
&\vdots&  \notag\\
\delta_{m-1} \vz'_{m-1} &=&   \vfbar^{(1)}_{m-1}(\vkbar,\vz) +  
\vgbar_{m-1}(\vkbar,\vz,\vect{\bar \delta}), \notag \\
\vz'_m &=&    \vfbar^{(1)}_m(\vkbar,\vz) +  \vgbar_m(\vkbar,\vz,\vect{\bar \delta}),
\end{eqnarray}
where  $\vgbar_k(\vkbar,\vz,\vect{\bar \delta})= \delta^{b'_k}  \vfbar^{(2)}_k(\vkbar,\vz,\vect{\bar \delta})$
satisfy $\vgbar_k(\vkbar,\vz,0)=0$ for $1 \leq k \leq m$.
We assume that the functions $\vgbar_k$ are smooth in all their arguments. The smoothness
in $\vect{\bar \delta}$ can be tested algorithmically with methods introduced in 
\cite{kruff2021algorithmic}.


By setting $\vect{\bar \delta} = \vect{0}$ in system \eqref{slowest} we obtain  the {\bf slowest
timescale reduced system}
\begin{eqnarray}  \label{slowesttruncated}
0 &=&   \vfbar^{(1)}_1(\vkbar,\vz) , \notag\\
&\vdots&  \notag\\
0 &=&   \vfbar^{(1)}_{m-1}(\vkbar,\vz), \notag \\
\vz'_m &=&    \vfbar^{(1)}_m(\vkbar,\vz) .
\end{eqnarray}
For $\vk \in \RRpp^r$ and 
$\vz_m \in \RRpp^{n_m}$ 
a state $(\vz_1,\ldots,\vz_{m-1})$ satisfying the system of equations
\begin{equation}\label{eq:qss}
\vfbar_1^{(1)}(\vkbar,\vz_1,\ldots,\vz_{m-1},\vz_{m})=0,\, \ldots, \, \vfbar_{m-1}^{(1)}(\vkbar,\vz_1,\ldots,\vz_{m-1},\vz_{m})=0
\end{equation}
is called a {\bf quasi-steady state}. System \eqref{eq:qss} is called the
{\bf quasi-steady state condition}.  

Assume that system \eqref{eq:qss}
can be solved for 
$(\vz_1,\ldots,\vz_{m-1})$
in the following  hierarchical way.
First, there is a differentiable function 
$ \vftilde_1(\vkbar,\vz_2,\ldots,\vz_m)$ such that
$$\vfbar_1^{(1)}(\vkbar,\vftilde_1(\vkbar,\vz_2,\ldots,\vz_m),\vz_2,\ldots,\vz_m)=0.$$ 
Next,
consider that there is a differentiable function 
$ \vftilde_2(\vkbar,\vz_3,\ldots,\vz_m)$  such that
$$\vfbar_2^{(1)}(\vkbar,
\vftilde_1(\vkbar,
\vftilde_2(\vkbar,\vz_3,\ldots,\vz_m),\ldots,\vz_m),
\vftilde_2(\vkbar,\vz_3,\ldots,\vz_m),\ldots,\vz_m)=0.$$
Assuming that the procedure can go on, consider finally
that there is a function 
$\vftilde_{m-1}(\vkbar,\vz_m)$
such that 
$$\vfbar_{m-1}^{(1)}(\vkbar,\vz_1,\vz_2,\ldots,\vz_{m-1},\vz_m)=0,$$
where $\vz_1$, $\vz_2$, $\ldots$, $\vz_{m-1}$ are recursively replaced
by $\vftilde_1(\vkbar,\vz_2,\ldots,\vz_m)$, 
$\vftilde_2(\vkbar,\vz_3,\ldots,\vz_m)$, $\ldots$, 
$\vftilde_{m-1}(\vkbar,\vz_m)$, respectively. 

Consider the reduced system
\begin{equation}\label{slowestreduced}
\vz'_m =   \vf^*_m (\vkbar,\vz_m) ,
\end{equation}
where  $\vf^*_m (\vkbar,\vz_m)$ is obtained 
from
$\vfbar_m^{(1)} (\vkbar,\vz_1,\ldots,\vz_m)$
by substituting
$\vz_1$, $\vz_2$, $\ldots$, $\vz_{m-1}$ 
as above. 
 
Solutions of system \eqref{slowest} in the limit $\vect{\bar \delta} \to 0$ were studied by Tikhonov \cite{tikh}, 
Hoppensteadt \cite{hopp} and O'Malley \cite{o1971initial}. They showed that under appropriate conditions (roughly speaking the non-degeneracy and hyperbolicity of the 
quasi-steady states $\vz_{k-1}=\vftilde_{k-1}(\vkbar,\vz_k,\ldots,\vz_m)$ for $2\leq k \leq m$, see Section~\ref{sec:chains}), the solutions of the  
system \eqref{slowest} with initial conditions $\vz_i(0) = g_i(\delta)$, where the 
$g_i(\delta)$ are differentiable functions, converge  for $\vect{\bar \delta} \to 0$ to the solutions
of the system \eqref{slowestreduced} with initial conditions $\vz_m(0) = g_m(0)$.
By this reduction, called {\bf quasi-steady state approximation},
all variables faster than the slowest one are eliminated and the reduced model describes the dynamics
at the slowest timescale. 
This type of reduction is the most popular one in applications, for instance in 
physical chemistry (where it is known as the Semenov-Bodenstein quasi-steady state 
approximation \cite{bodenstein1913theorie,semenov1956some}) and in 
computational systems biology \cite{radulescu2012reduction}.


\subsubsection{Nested Intermediate Timescale Reductions} \label{sub:nested}
This type of reduction 
was proposed by Cardin and Teixera \cite{cardin2017fenichel} and algorithmically formalized by Kruff et al.~in
\cite{kruff2021algorithmic}. We call it nested, since  
the reduced dynamics at intermediate timescales are embedded
in  normally hyperbolic invariant manifolds that form a nested family 
(manifolds of slower variables are included in manifolds of faster
variables). 
In \cite{cardin2017fenichel} both manifolds were considered, the stable and unstable one. 
As in \cite{kruff2021algorithmic} we only consider here the stable case. The stable normally hyperbolic invariant manifolds attract and sequentially confine the dynamics 
of the system, with rates from fastest to slowest, and are used to obtain reductions valid at intermediate timescales.


We provide here the formal description of the reduction at an intermediate 
timescale of order 
$\delta^{b_l} =\delta^{b_1} \delta_1 \delta_2 \ldots \delta_{l-1}$, where  
$\delta_l,\, 1\leq l \leq m$ are defined as in the previous section.

Redefining  
the time into $\tau' = \tau \delta^{b_l}$, where $1 \leq l \leq  m$, leads to the system
\begin{eqnarray}  \label{multiscale}
\delta_1 \delta_2 \ldots \delta_{l-1} \vz'_1 &=&   ( \vfbar_1^{(1)}(\vkbar,\vz) +  \vgbar_1(\vkbar,\vz,\vect{\bar \delta})), \notag\\
&\vdots&  \notag\\
\delta_{l-1} \vz'_{l-1} &=&   ( \vfbar_{l-1}^{(1)}(\vkbar,\vz) +  \vgbar_{l-1}(\vkbar,\vz,\vect{\bar \delta})), \notag \\
\vz'_l &=&   ( \vfbar_l^{(1)}(\vkbar,\vz) + \vgbar_l(\vkbar,\vz,\vect{\bar \delta})), \notag \\
\vz'_{l+1} &=&  \delta_{l} ( \vfbar_{l+1}^{(1)}(\vkbar,\vz) +
\vgbar_{l+1}(\vkbar,\vz,\vect{\bar \delta})), \notag \\
&\vdots&  \notag\\
 \vz'_m &=&  \delta_{l}\delta_{l+1}\ldots  \delta_{m-1} ( \vfbar_m^{(1)}(\vkbar,\vz) + \vgbar_m(\vkbar,\vz,\vect{\bar \delta})), 
\end{eqnarray}
 where $\vect{\bar \delta} = (\delta_{1},\ldots,\delta_{m-1})$ and $\vgbar_k(\vkbar,\vz,\vect{\bar \delta})$ is a smooth functions with  
 $\vgbar_k(\vkbar,\vz,0)=0$ for all $1\leq k \leq m$.
 In the following we call the system
\begin{eqnarray}  \label{truncatedmultiscale}
0 &=&   \vfbar_1^{(1)}(\vkbar,\vz), \notag\\
&\vdots&  \notag\\
0 &=&    \vfbar_{l-1}^{(1)}(\vkbar,\vz), \notag 
\end{eqnarray}
\begin{eqnarray}
\vz'_l &=&    \vfbar_l^{(1)}(\vkbar,\vz), \notag \\
\vz'_{l+1} &=&  \delta_{l} ( \vfbar_{l+1}^{(1)}(\vkbar,\vz) +
\vgbar_{l+1}(\vkbar,\vz,\vect{\bar \delta})), \notag \\
&\vdots&  \notag\\
 \vz'_m &=&  \delta_{l}\delta_{l+1}\ldots  \delta_{m-1} ( \vfbar_m^{(1)}(\vkbar,\vz) + \vgbar_m(\vkbar,\vz,\vect{\bar \delta})), 
\end{eqnarray}
where $\vect{\bar \delta} = (0,\ldots,0,\delta_{l},\ldots,\delta_{m-1})$, which we obtained from \eqref{multiscale} by setting $\delta_1, \, \delta_2, \, \ldots , \, \delta_{l-1}$ to zero, the 
{\bf reduced system at the $l$-th fastest time or slower}.
In this case as well, the system of 
equations $\vfbar_1^{(1)}(\vkbar,\vz)=0, \, \ldots,\,\vfbar_{l-1}^{(1)}(\vkbar,\vz)=0$
corresponds to the {\bf quasi-steady state condition}.
In \cite{kruff2021algorithmic} we have also defined the simpler reduced system  
obtained by setting $\delta_1, \, \delta_2\, \ldots , \, \delta_{m-1}$ in \eqref{multiscale} to zero, namely
\begin{eqnarray}  \label{truncatedmulticalesimpler}
0 &=&   \vfbar_1^{(1)}(\vkbar,\vz), \notag\\
&\vdots&  \notag\\
0 &=&    \vfbar_{l-1}^{(1)}(\vkbar,\vz), \notag \\
\vz'_l &=&    \vfbar_l^{(1)}(\vkbar,\vz), \notag \\
\vz'_{l+1} &=&  0, \notag \\
&\vdots&  \notag\\
 \vz'_m &=&  0.
\end{eqnarray}
This reduced model emphasizes
three groups of variables: 
 slaved variables $\vz_1,\ldots,\vz_{l-1}$, which are faster than $\delta^{b_l}$,
 driving variables $\vz_l$, which have timescale $\delta^{b_l}$, 
and quenched variables $\vz_{l+1},\, \ldots, \, \vz_m$, which are slower than $\delta^{b_l}$.
If regularity and hyperbolicity conditions are satisfied (see \cite{cardin2017fenichel} and Section~\ref{sec:chains}), then the solutions 
of system \eqref{multiscale} converge to the solutions
of system \eqref{truncatedmultiscale} when $\epsilon_k \to 0$
for $ 1 \leq k \leq l-1$ (see the Corollary of Theorem A in 
\cite{cardin2017fenichel}).

The limit $(\delta_l,\ldots,\delta_{m-1}) \to 0$ leading from system \eqref{truncatedmultiscale}
to system \eqref{truncatedmulticalesimpler} can be treated in the simpler framework of 
regular perturbations. Using the same regularity and hyperbolicity conditions,  
one can show that there is a time $T>0$ such that 
the solutions of system \eqref{multiscale} converge for $\vect{\bar \delta} \to 0$ uniformly on any close
subinterval of $(0,T]$
to the solutions of system \eqref{truncatedmulticalesimpler} (cf.~Theorem 1 of \cite{kruff2021algorithmic}).
This result implies that 
the reduction \eqref{truncatedmulticalesimpler} is valid 
on a time interval $[t_1\delta^{-b_l},t_2\delta^{-b_l}]$ with $[t_1,t_2] \subset (0,T]$. The reduction $\eqref{truncatedmultiscale}$
has a broader validity including times longer than $\delta^{-b_l}$.

\subsubsection{Hyperbolically Attractive Chains and Quasi-steady State Conditions}\label{sec:chains}
The two types of reductions presented in the previous section are based on 
hierarchical elimination of fast variables, previously 
discussed in \cite{kruff2021algorithmic}. We revisit here this construction, using the
concept of the Schur complement.

Let us define 
$$\vZ_k = \begin{pmatrix} \vz_1 \\ \vdots \\ \vz_k \end{pmatrix} \quad \mathrm{and} \quad \vFbar_k^{(1)} (\vkbar,\vz)  = \begin{pmatrix} \vfbar^{(1)}_1(\vkbar,\vz) \\ \vdots \\ \vfbar^{(1)}_k(\vkbar,\vz) \end{pmatrix}.$$ 
For any $1 \leq k \leq m $, the system of equations $\vFbar_k^{(1)} (\vkbar,\vz) = 0$ defines  
the $k$-th quasi-steady variety.
The set of positive solutions $\vz$ of 
$\vFbar_k^{(1)} (\vkbar,\vz) = 0$ is denoted by ${\mathcal M}_k$ and represents the intersection
of the $k$-th quasi-steady state variety with the first orthant. 
For $1 \leq l \leq m $ we call the chain $\RRpp^n={\mathcal M}_0 \supseteq  {\mathcal M}_1 \supseteq \ldots \supseteq {\mathcal M}_l$ of nested quasi-steady state varieties
lying in the first orthant an {\bf $l$-chain}.

We solve the equations $\vFbar_k^{(1)} (\vkbar,\vz) = 0$ by successive elimination of variables,
starting
with $\vz_1$ and ending with $\vz_k$. During the elimination process,  
 intermediary functions  $\vf_k^* (\vkbar,\vz_k,$ $\ldots,\vz_m)$,
 $\vftilde_k(\vkbar,\vz_{k+1},\ldots,\vz_m)$ and 
  $\vFtilde_k(\vkbar,\vz_{k+1},\ldots,\vz_m)$
 are defined recursively. More precisely, let
 $$\vf_1^* (\vkbar,\vz_1,\ldots,\vz_m) =\vfbar_{1}^{(1)} (\vkbar,\vz_{1},\ldots,\vz_m)$$
 and let $\vz_1 = \vftilde_1(\vkbar,\vz_{2},\ldots,\vz_m)$ be the locally unique solution of $\vf_1^* (\vkbar,\vz_1,\ldots,\vz_m)=0$  
and set
$$\vFtilde_1(\vkbar,\vz_{2},\ldots,\vz_m)=
\vftilde_1(\vkbar,\vz_{2},\ldots,\vz_m).$$ Giving this initial data one continues then recursively for $2\leq k \leq m-1$ in the following way: Define 
\begin{equation}
 \label{eqfstar}
\vf_{k}^* (\vkbar,\vz_{k},\ldots,\vz_m)=\vfbar_{k}^{(1)} (\vkbar,\vFtilde_{k-1}(\vkbar,\vz_{k},\ldots,\vz_m),\vz_{k},\ldots,\vz_m),   
\end{equation}
determine the locally unique solution 
\begin{equation}
\vz_k = \vftilde_k(\vkbar,\vz_{k+1},\ldots,\vz_m) 
 \label{eqvftilde} 
\end{equation}
of $\vf_k^* (\vkbar,\vz_k,\ldots,\vz_m)=0$ and then set 
\begin{equation}
\vFtilde_k(\vkbar,\vz_{k+1},\ldots,\vz_m) = 
\begin{pmatrix} 
\vFtilde_{k-1}(\vkbar, \vftilde_k(\vkbar,\vz_{k+1},\ldots,\vz_m)  ,\vz_{k+1},\ldots,\vz_m) \\
\vftilde_k(\vkbar,\vz_{k+1},\ldots,\vz_m)
\end{pmatrix} .  
\end{equation}
From the recursion the following proposition follows immediately.
\begin{proposition} \label{propvFbar}
The vector 
$\vz=(\vFtilde_k(\vkbar,\vz_{k+1},\ldots,\vz_m),\vz_{k+1},\ldots,\vz_m)$ 
is a 
solution of $\vFbar_k^{(1)} (\vkbar,\vz)=0$.
\end{proposition}
The existence of the implicit functions 
$\vftilde_k$ and $\vFtilde_k$ needs the following non-degeneracy condition. 
\begin{condition} \label{cond:nondegeneracy}
The solution $\vz_{k}$ of the equation 
$\vf_{k}^* (\vkbar,\vz_{k},\ldots,\vz_m)=0$
is non-degenerate, i.e.~$\mathrm{det}(D_{\vz_k} \vf_k^*) \neq 0$. 
\end{condition}
This condition can be written more conveniently.
\begin{theorem}\label{thregularity}
For $1 \leq k \leq l$ the implicit functions $\vftilde_k(\vkbar,\vz_{k+1},\ldots,\vz_m)$ and $\vFtilde_k(\vkbar,\vz_{k+1},$ $\ldots,\vz_m)$ 
exist and are differentiable,
if and only if 
\begin{equation}\label{condreg}
\mathrm{det}(D_{\vZ_k} \vFbar_k^{(1)}) \neq 0 \text{ for all } \vz \in {\cal M}_k .
\end{equation}
\end{theorem}
In order to prove Theorem\ref{condreg} we need the Schur complement, which occurs naturally during the
Gaussian elimination of variables (see, for instance, \cite{zhang2006schur}). 
\begin{definition}
Let   $M = \begin{pmatrix} A & B \\ C & D \end{pmatrix}$ be a block matrix with 
$A$ invertible. The matrix $M/A = D - C A^{-1} B$ is called the Schur complement of the block $A$ of $M$. 
\end{definition}
The Schur complement can be obtained from the following successive computations:
\begin{itemize}
    \item Solve $A X + B Y = 0$ for $X$.
    \item Substitute $X$ in $C X + D Y$, 
    which leads to $(M/A) Y$.
\end{itemize}
Moreover, the Schur complement
has the follwoing two
simple properties \cite{zhang2006schur}:
\begin{eqnarray}
\mathrm{det}(M)  &=& \mathrm{det}(M/A) \, \mathrm{det}(A ) \quad \quad \text{(Schur formula), }  \label{schur} \\
\mathrm{rk}(M)  &=&   \mathrm{rk}(M/A) + \mathrm{rk}(A) \quad \quad \text{(Guttman rank additivity formula). } \label{guttman} 
\end{eqnarray}

Returning to our problem, we prove now the following lemma.
\begin{lemma} \label{lemmaschur}
The  matrix $D_{\vz_{k+1}} \vf_{k+1}^*$ is a Schur complement. More precisely, we have 
\begin{equation}\label{eqschur}
D_{\vz_{k+1}} \vf_{k+1}^* = D_{\vZ_{k+1}} \vFbar_{k+1}^{(1)} / D_{\vZ_k} \vFbar_k^{(1)},
\end{equation}
 for all $1 \leq k \leq m-1$ and $\vz \in {\cal M}_{k}$ such that $D_{\vZ_k} \vFbar_k^{(1)}$ is invertible.
\end{lemma}
\begin{proof}
We differentiate \eqref{eqfstar} with respect to $\vz_{k+1}$ and obtain
$$ D_{\vz_{k+1}} \vf_{k+1}^*  =  (D_{\vZ_{k}} \vfbar_{k+1}^{(1)} ) D_{\vz_{k+1}} \vFtilde_{k} + D_{\vz_{k+1}} \vfbar_{k+1}^{(1)}.  $$
From Proposition~\ref{propvFbar} we have that 
$$\vFbar_k^{(1)} (\vkbar,\vFtilde_k(\vkbar,\vz_{k+1},\ldots,\vz_m),\vz_{k+1},\ldots,\vz_m)=0.$$
Differentiating the last equation with respect to $\vz_{k+1}$ we obtain
 $$(D_{\vZ_{k}} \vFbar_{k}^{(1)}) D_{\vz_{k+1}} \vFtilde_k + D_{\vz_{k+1}} \vFbar_{k}^{(1)} =0$$
 for $\vz \in {\cal M}_k$.
 Finally, we get
 $$ D_{\vz_{k+1}} \vf_{k+1}^* = - (D_{\vZ_{k}} \vfbar_{k+1}^{(1)} ) 
 (D_{\vZ_{k}} \vFbar_{k}^{(1)})^{-1}
 D_{\vz_{k+1}} \vFbar_{k}^{(1)}
 + D_{\vz_{k+1}} \vfbar_{k+1}^{(1)} =  D_{\vZ_{k+1}} \vFbar_{k+1}^{(1)} / D_{\vZ_k} \vFbar_k^{(1)}. $$
\end{proof}
Now we can prove Theorem~\ref{thregularity}. 
\begin{proof}[Proof of Theorem~\ref{thregularity} ]
Using Schur formula \eqref{schur} and Lemma~\ref{lemmaschur}, we obtain that
$$\mathrm{det}(D_{\vz_{k}} \vf_{k}^*) \, \mathrm{det}(D_{\vZ_{k-1}} \vFbar_{k-1}^{(1)}) = \mathrm{det}(D_{\vZ_{k}} \vFbar_{k}^{(1)})
$$
for $2\leq k \leq l$ and that 
$$\mathrm{det}(D_{\vz_{1}} \vf_{1}^*)  = \mathrm{det}(D_{\vZ_{1}} \vFbar_{1}^{(1)}).  $$
This implies that $\mathrm{det}(D_{\vz_{k}} \vf_{k}^*)\neq 0 $ for $1 \leq k \leq l$
is equivalent to $\mathrm{det}(D_{\vZ_{k}} \vFbar_{k}^{(1)}) \neq 0$ for $1 \leq k \leq l$.
\end{proof}


As discussed in \cite{kruff2021algorithmic} and \cite{o1971initial}, the validity 
of the quasi-steady state approximation depends also on the
 following hyperbolicity condition:
\begin{condition}[Hyperbolicity]\label{cond:hyp}
For all $1\leq k \leq m $ the solution $\vz_{k}=\vftilde_{k}(\vkbar,\vz_{k+1},\ldots,\vz_m)$ is a hyperbolically stable steady state
of the ODE $$\vz_k' = \vf^*_k  (\vkbar, \vz_k,\vz_{k+1},\ldots,\vz_m),$$ 
where $\vf^*_k$ is defined as in the subsection~\ref{sec:chains}.
Here, by hyperbolically stable we mean 
that all eigenvalues of the Jacobian matrix at the steady state
have strictly negative real parts. 
The validity of the nested reduction at the $l$-th fastest time or slower requires that Condition \ref{cond:hyp}
is fulfilled  for all $1\leq k  \leq l $.
\end{condition}
As discussed in \cite{kruff2021algorithmic}, an important concept for the geometric theory of singular perturbations
is the hyperbolically attractive chain.
\begin{definition}\label{def:chain}
An $l$-chain of nested quasi-steady state varieties ${\mathcal M}_0 \supseteq  {\mathcal M}_1 \supseteq \ldots \supseteq {\mathcal M}_l$ 
is called a {\bf hyperbolically attractive $l$-chain} if for all $1\leq k \leq l$ all eigenvalues of 
$D_{\vz_{k}} \vf_{k}^*$ have strictly negative real parts for 
$\vz \in {\mathcal M}_k$. In this case we write
${\mathcal M}_0 \rhd {\mathcal M}_1 \rhd \ldots \rhd {\mathcal M}_l$. 
\end{definition}


Summarizing,
the nested reduction  \eqref{truncatedmultiscale} is valid up to the $l$-th timescale if the $l$-chain ${\mathcal M}_0 \rhd {\mathcal M}_1 \rhd \ldots \rhd {\mathcal M}_l$ is hyperbolically attractive (see \cite{cardin2017fenichel,kruff2021algorithmic}). The slowest timescale quasi-steady state reduction \eqref{slowestreduced} is 
valid if the $m$-chain ${\mathcal M}_0 \rhd {\mathcal M}_1 \rhd \ldots \rhd {\mathcal M}_m$ is hyperbolically attractive.

From Lemma~\ref{lemmaschur} we obtain the following proposition. 
\begin{proposition}\label{prophyper}
An $l$-chain is hyperbolically attractive if and only if
$\mathrm{det}(D_{\vZ_{k}} \vFbar_{k}^{(1)}) \neq 0$ for all $1 \leq k \leq l$,
all eigenvalues of
$D_{\vz_{1}} \vfbar_{1}^{(1)}$ have negative real parts for all $\vz \in {\cal M}_1$
and for all $2\leq k \leq l$ all eigenvalues of
$D_{\vZ_{k}} \vFbar_{k}^{(1)} / D_{\vZ_{k-1}} \vFbar_{k-1}^{(1)}$ have negative 
real parts for all $\vz \in {\cal M}_{k-1}$.
\end{proposition}

\subsubsection{Approximate Conservation Laws and the Quasi-equilibrium Condition}

Linear approximate conservation laws were already proposed as
a tool for model reduction of CRNs when the so-called {\bf quasi-equilibrium} (QE) condition \cite{gorban2010asymptotology,radulescu2012reduction} is satisfied.  At QE the direct and reverse rates of fast reversible reactions compensate each other and the net rates of change of reactants and products are negligible. Products or reactants of fast reactions are fast species. However, although
concentrations of fast species are equilibrated, these variables can not be eliminated by using the quasi-steady state (QSS)
equations. In the case of QE linear combinations of concentrations
of fast species are conserved by the fast dynamics and the QSS equations have degenerate solutions indexed by the values of the conserved quantities \cite{gorban2010asymptotology,radulescu2012reduction}. 

In this paper we show that the degeneracy of solutions of
QSS equations is valid more generally, for any approximate conservation 
laws.


Let us denote by $\vx_l$ and $\vX_l=(\vx_1,\ldots,\vx_l)$ the set of variables
of system \eqref{eq:fi} having time-scales of order $\mu_l$ and timescales equal
to or faster than $\mu_l$, respectively. 
Furthermore, let  $\vf_l(\vk,\vx)$, $\vf_l^{(1)}(\vk,\vx)$, 
$$\vF_l(\vk,\vx) = \begin{pmatrix} \vf_1(\vk,\vx) \\ \vdots \\ \vf_l(\vk,\vx)  \end{pmatrix} \quad \mathrm{and} \quad 
\vF_l^{(1)}(\vk,\vx) = \begin{pmatrix} \vf_1^{(1)}(\vk,\vx) \\ \vdots \\ \vf_l^{(1)}(\vk,\vx)  \end{pmatrix} $$
be the full and truncated vector fields whose flows have timescales of order $\mu_l$, and   timescales equal
to or faster than $\mu_l$, respectively. In other words, these vector fields are the unscaled versions of the vector fields
 $\vfbar_l(\vkbar,\vz)$,  $\vfbar_l^{(1)}(\vkbar,\vz)$,
  $\vFbar_l(\vkbar,\vz)$ and  $\vFbar_l^{(1)}(\vkbar,\vz)$ introduced in Section~\ref{sec:formal_scaling}.

Let $\phi_l(\vX_l)$ be a linear, monomial or polynomial  approximate conservation law depending only on 
the variables $\vX_l$ satisfying $D_{\vX_l} \phi_l(\vX_l) \vF_l^{(1)} = 0$.
This approximate conservation law can eventually also be exact, in 
which case it also satisfies $D_{\vX_l} \phi_l(\vX_l) \vF_l = 0$.
The existence of such an approximate conservation law implies the failure of 
Condition~\eqref{condreg} in  Theorem~\ref{thregularity} as in the following proposition.
\begin{proposition}
\label{prop:quasieq}
Let us assume that there is an approximate conservation law $\phi_l(\vX_l)$, where $1\leq l \leq m$.
Then 
$\mathrm{det}(D_{\vZ_l} \vFbar^{(1)}_l )=0$
if $\vz \in {\cal M}_l$. 
\end{proposition}
\begin{proof}
The definition of approximate conservation laws yields
\begin{equation}\label{phil}
D_{\vX_l} \phi_l(\vX_l) \vF^{(1)}_l(\vk,\vx)=0.
\end{equation}
Since, up to the change of variables $\vkbar \leftarrow \vk$ and $\vz \leftarrow \vx$, the polynomials $\vF^{(1)}_l$ and $\vFbar^{(1)}_l$ are identical, 
 we find
\begin{equation}\label{consl}
D_{\vZ_l} \phi_l (\vZ_l) \vFbar^{(1)}_l(\vkbar,\vz)=0.
\end{equation}
Differentiating \eqref{consl} we obtain 
$$\DD{\phi_l}{\vZ_l}\vFbar^{(1)}_l  +  \D{ \phi_l}{\vZ_l}\D{\vFbar^{(1)}_l}{\vZ_l}=0,$$ where  
$\DD{ \phi_l}{\vZ_l}$ is the second derivative of $ \phi_l$
with respect to $\vZ_l$.
If $\vz \in {\cal M}_l$, then $\vFbar^{(1)}_l(\vkbar,\vz)=0$ and therefore
$$\D{ \phi_l}{\vZ_l}\D{\vFbar^{(1)}_l}{\vZ_l}=0.$$
Thus, the left kernel of the
matrix $\D{\vFbar^{(1)}_l}{\vZ_l}$ contains the covector $\D{ \phi_l}{\vZ_l} \neq 0$ and so  
$\mathrm{det}(D_{\vZ_l} \vFbar^{(1)}_l )=0$. 
\end{proof}

In this case the quasi-steady state condition $\vF_l^{(1)}(\vk,\vX_l,\vx_{l+1},\ldots,\vx_m)=0$  
can not be used to eliminate the fast variables $\vX_l$.
However, even in this case the CRN~\eqref{eq:fi} 
can be transformed to an equivalent CRN that fulfils the nondegeneracy 
Condition~\eqref{condreg}  for $1\leq k \leq l$. The details of the transformation are presented below.

Let $\vect{\Phi}_l(\vX_l) = (\phi_{1l}(\vX_l),\ldots,\phi_{s_ll}(\vX_l))^T$ be a set of $s_l$ approximate conservation laws 
dependent on $\vX_l$ and consider the equation
\begin{equation}\label{tzm}
\vx_{l}^c = \vect{\Phi}_l(\vX_l).
\end{equation}
Assuming that the conservation laws $\vect{\Phi}_l (\vX_l)$
are independent as functions of $\vx_l$, namely that
$$ \mathrm{rk}  (D_{\vx_l} \vect{\Phi}_l(\vX_l) ) = s_l,$$  
 we have, up to a relabelling of the components of $\vX_{l}$,
the splitting $\vX_{l}=(\vX_{l-1}, \check \vx_{l}, \hat \vx_{l})$, where $\vX_{l-1}\in \RR^{n_1 + \ldots + n_{l-1}}$,
$\check{\vx}_{l} \in \RR^{s_{l}}$, $\hat{\vx}_{l} \in \RR^{n_l - s_{l}} $  
and $\mathrm{det}(D_{\check{ \vx}_{l}}\vect{\Phi}_l)\neq 0$.
Hence,
\eqref{tzm} defines the implicit function
$\check{ \vx}_{l}= \vect{\Psi}_{l}(\vX_{l-1},\hat \vx_{l},\vx_{l}^c )$, which allows to eliminate the variables $\check{\vx}_l$.

The above splitting of $\vX_l$ induces the splittings
$\vF_l = (\vF_{l-1}, \vfcheck_l,\vfhat_l)$ and
$\vF^{(i)}_l = (\vF^{(i)}_{l-1},$ $ \vfcheck^{(i)}_l,\vfhat^{(i)}_l)$ for $1 \leq i \leq 2$.
Let us define the functions   
\begin{eqnarray}
&\vF_k^{red} (\vk,{\vX}_{l-1},\hat{\vx}_l,  \vx_{l}^c, \vx_{l+1},  \vx_{l+2} , \ldots,\vx_m)   = \notag \\
&\vFhat_k (\vk,{\vX}_{l-1},  \vect{\Psi}_{l}( \vX_{l-1},\hat{\vx}_l,\vx_{l}^c), \hat{\vx}_l,\vx_{l+1},  \vx_{l+2} , \ldots,\vx_m),   \label{eq:redfun}
\end{eqnarray}
where $\vFhat_k= \vF_k$ for $1 \leq k \leq l-1$, and $\vFhat_l= (\vF_{l-1}, \vfhat_l)$.
The transformed 
model obtained from the substitution $\check{ \vx}_{l}= \vect{\Psi}_{l}(\vX_{l-1},\hat \vx_{l},\vx_{l}^c )$  is
\begin{eqnarray}
\dot{\hat{\vX}}_{l} &=& \vF_l^{red} 
(\vk,{\vX}_{l-1},\hat{\vx}_l,  \vx_{l}^c, \vx_{l+1},  \vx_{l+2} , \ldots,\vx_m), 
\label{eq:transformed1} \\
\dot \vx_l^c &=&  ((D_{\vX_l} \vect{\Phi}_l ) \vF_l^{(2)} )
(\vk,{\vX}_{l-1},  \vect{\Psi}_{l}( \vX_{l-1},\hat{\vx}_l,\vx_{l}^c), \hat{\vx}_l,\vx_{l+1},  \vx_{l+2} , \ldots,\vx_m), \label{eq:new} \\
\dot{\vx}_{k} &=& \vf_k(\vk,{\vX}_{l-1},  \vect{\Psi}_{l}( \vX_{l-1},\hat{\vx}_l, \vx_{l}^c),\hat{\vx}_l,\vx_{l+1},  \vx_{l+2} , \ldots,\vx_m), \, k \in 
\{l+1,\ldots,m\}.  \label{eq:transformed2}
\end{eqnarray}

We have seen in Theorems~\ref{th:structurelinear} and \ref{th:structurepolynomial}) that $\vect{\Phi}_l =  \vect{\Phi}_l^{(1)} (\vx_l) + \ord{\delta^{q_f}}$,
 where $\vect{\Phi}_l^{(1)} (\vx_l) = \Ord{\delta^{q_f}}$ is the lowest order (dominant) part of  $\vect{\Phi}_l$.
 Furthermore, the dominant part $\vect{\Psi}^{(1)}_{l}$ of $\vect{\Psi}_{l}$ satisfies the equation
\begin{equation}\label{eq:truncons}
\vx_{l}^c = \vect{\Phi}_l^{(1)}(\vect{\Psi}^{(1)}_{l}( \hat{\vx}_l,\vx_{l}^c),\hat{\vx_l}).
\end{equation}
Thus, the truncated versions of the functions $\vF_k^{red}$ are
\begin{eqnarray}
&\vF_k^{red,1} (\vk,{\vX}_{l-1},\hat{\vx}_l,  \vx_{l}^c, \vx_{l+1},  \vx_{l+2} , \ldots,\vx_m)   = \notag \\
&\vFhat^{(1)}_k (\vk,{\vX}_{l-1},  \vect{\Psi}^{(1)}_{l}( \hat{\vx}_l,\vx_{l}^c), \hat{\vx}_l,\vx_{l+1},  \vx_{l+2} , \ldots,\vx_m),   \label{eq:redfuntrunc}
\end{eqnarray}
where $\vFhat^{(1)}_k= \vF^{(1)}_k$ for $1 \leq k \leq l-1$, and  
$\vFhat^{(1)}_l= (\vF^{(1)}_{l-1}, \vfhat_l)$.


We can state now the main result of this section.  Let us assume that the following conditions are satisfied.
\begin{condition}
\label{cond:cons}
$\,$
\begin{enumerate}
    \item For any $\vk\in \RRpp^r$ there exist $\vx\in \RRpp^n$ such that $\vF^{(1)}_l(\vk,\vx)=0$. For all 
    $\vx \in \RRpp^n$ with $\vF^{(1)}_l(\vk,\vx)=0$, we have 
$\mathrm{det}(D_{\vX_k}{ \vF^{(1)}_k}) \neq 0$ 
for all $1 \leq k \leq l-1$ 
and $\mathrm{det}(D_{\vX_l}{ \vF^{(1)}_l}) = 0$.
\item
There is a set of $s_{l}$ \cor{simple} approximate conservation laws 
$$\vect{\Phi}_l(\vX_l) =\transpose{ (\phi_{1l}(\vX_l),\ldots,\phi_{s_{l}l}(\vX_l))}$$ 
depending only on $\vX_l$ such that 
 $(D_{\vX_l} \vect{\Phi}_l )\vF^{(1)}_l = 0$, , where $0 < s_{l} \leq n_l$.
For all $\vk \in \RRpp^r$ and $\vx \in \RRpp^n$ such that $\vF^{(1)}_l(\vk,\vx)=0$ we have that
\begin{equation} \label{eq:rankflphil}
   \mathrm{rk} \left( D_{\vX_l} \begin{pmatrix} 
      \vect{\Phi}_l^{(1)} \\\vFhat^{(1)}_l 
    \end{pmatrix} \right)  = n_1 + \ldots + n_l.
\end{equation}

\item
The  conservation laws $\vect{\Phi}_l (\vX_l)$
are independent as functions of $\vx_l$, namely 
\begin{equation}\label{eq:rankphil}
\mathrm{rk}  (D_{\vx_l} \vect{\Phi}_l(\vX_l) ) = s_l.
\end{equation}
\end{enumerate}
\end{condition}



\begin{theorem}\label{th:qe}
If Condition~\ref{cond:cons} is fulfilled and
$$\vF_l^{red,1}(\vk,{\vX}_{l-1},\hat{\vx}_l,\vx_{l}^c,\vx_{l+1},\vx_{l+2}, \ldots,\vx_m)=0 ,$$ then
$\mathrm{det}(D_{\hat{\vX}_k}\vF_k^{(red,1)}) \neq 0$ for all $1\leq k \leq l$,
where $\hat{\vX}_k={\vX}_k$
for $1\leq k\leq l-1$ and $\hat{\vX}_l=({\vX}_{l-1},\hat{\vx}_l)$.
\end{theorem}


The proof of Theorem~\ref{th:qe} uses the following lemma.
\begin{lemma}\label{lemma:fred}
We have
$$D_{\hat{\vX}_l} \vF_l^{red,1} = D_{{\vX}_l} \left.
\begin{pmatrix}\vect{\Phi}_l^{(1)} \\ \vFhat^{(1)}_l  \end{pmatrix} 
\right/ D_{\check{\vx}_l} \vect{\Phi}_l^{(1)} .$$ 
\end{lemma}
\begin{proof}
Differentiating \eqref{eq:truncons} with respect to $ \hat \vX_l$ we obtain 
$$D_{\hat \vX_l}\vect{\Phi}_l^{(1)} + D_{\check{\vx}_l}\vect{\Phi}_l^{(1)} D_{\hat \vX_l}\vect{\Psi}_l^{(1)} =0.$$
It follows from \eqref{eq:redfuntrunc} that
$$D_{\hat{\vX}_l} \vF_l^{red,1} = D_{\hat{\vX}_l} \vFhat^{(1)}_l + D_{\check{\vx}_l} \vFhat^{(1)}_l
D_{\hat \vX_l}\vect{\Psi}_l^{(1)} = D_{\hat{\vX}_l} \vFhat^{(1)}_l - D_{\check{\vx}_l} \vFhat^{(1)}_l
(D_{\check{\vx}_l}\vect{\Phi}_l^{(1)})^{-1} D_{\hat \vX_l}\vect{\Phi}_l^{(1)},$$ which completes the proof. 
\end{proof}

\begin{proof}[Proof of Theorem~\ref{th:qe}]
We prove that 
$D_{\hat{\vX}_k} \vF_k^{red,1}$ is invertible for all $1\leq k \leq l$.
Using the structure Theorems~\ref{th:structurelinear} and \ref{th:structurepolynomial} we find that $D_{\vx_l}\vect{\Phi}_l=D_{\vx_l}\vect{\Phi}^{(1)}_l$. Thus,  it follows from \eqref{eq:rankphil} that 
$\mathrm{rk}  (D_{\vx_l} \vect{\Phi}_l^{(1)} ) = s_l$.
Using \eqref{eq:rankflphil}, the Guttman rank additivity formula \eqref{guttman} and Lemma~\ref{lemma:fred}, we find that 
$$\mathrm{rk} \left(D_{\hat{\vX}_l} \vF_l^{red,1} \right) = n_1+n_2+\ldots+n_l - s_l$$ and so   
$D_{\hat{\vX}_l} \vFhat_l^{red,1}$ is invertible.
Since $\vect{\psi}_l^{(1)}$ does not depend on $\vX_k$ for $1\leq k \leq l-1$, it follows that
$D_{\hat{\vX}_k} \vF_k^{red,1} = D_{{\vX}_k} \vF_k^{(1)}$. Since $D_{{\vX}_k} \vF_k^{(1)}$ is invertible, the same is true for $D_{\hat{\vX}_k} \vF_k^{red,1}$. 

\end{proof}

\begin{remark}
Because conservation laws are used to eliminate 
the variables $\check{\vx}_l$, the rank conditions 
 \eqref{eq:rankflphil} and \eqref{eq:rankphil} 
 are satisfied if
$$\mathrm{det}\left( D_{\vX_l} \begin{pmatrix} 
 \vect{\Phi}_l^{(1)} \\   \vFhat^{(1)}_l  
    \end{pmatrix}  \right) \neq 0  $$
 and $ \mathrm{det}( D_{\check{\vx}_l} \vect{\Phi}_l ) \neq 0$, respectively. 
\end{remark}


\begin{remark}
Theorem~\ref{th:qe} allows to define $l$-chains in the case of a quasi-equilibrium. The set of positive solutions
of the set of equations 
$$\vF_k^{red,1} (\vk,{\vX}_{l-1},\hat{\vx}_l,\vx_l^c,\vx_{l+1},   \ldots,  \vx_m)   = 0,$$ 
that is equivalent to 
$$\vFhat_k^{(1)} (\vk,{\vX}_{l-1},{\vx}_l,\vx_l^c,\vx_{l+1},   \ldots,  \vx_m)   = 0,\, \vect{\Phi}_l^{(1)}({\vx}_l)=\vx_l^c,$$
is denoted
${\cal M}_k^{QE}$ and
represents the $k$-th
{\bf quasi-equilibrium variety} (intersected with the first orthant). These sets satisfy
$$\RRpp^n= {\cal M}_0^{QE}  \supset {\cal M}_1^{QE} \supset \ldots \supset {\cal M}_l^{QE}.$$ 
The concept of hyperbolically attractive $l-$chains is applicable to quasi-equilibrium varieties as well. 
\end{remark}

By the results of Section~\ref{sec:acslow} the new variables $\vx_l^c$ are slower than $\hat{\vX}_l$ 
and by Theorem~\ref{th:qe}
the transformed model satisfies the non-degeneracy conditions 
$\mathrm{det}(D_{\hat{\vX}_k}F_k^{(red,1)})\neq 0$
up to order $l$.
If the approximate conservation laws are also exact, the new variables $\vx_l^c$  are
constant and stand for new parameters. Then, the new equations
\eqref{eq:new} are not added to the transformed ones \eqref{eq:transformed1} and \eqref{eq:transformed2}.
In this case as well, the transformed model \eqref{eq:transformed1} and \eqref{eq:transformed2} satisfies the non-degeneracy conditions up to order $l$.

\subsubsection{Algorithmic Solution for Eliminating the Approximate Conservation Laws}
The previous section allows us to define an algorithm which transforms the CRN \eqref{eq:fi} into an equivalent one 
that does not have approximate conservation laws and that can be further reduced using the method introduced in \cite{kruff2021algorithmic}. 
During this transformation, some old variables 
are substituted by new ones, representing approximate
conservation laws that are not exact. Also, each exact conservation law leads to the creation of a new 
parameter and to the elimination of one 
variable together with the corresponding ODE.

Algorithm~\ref{alg:transform}
transforms the CRN into another CRN that satisfies the condition
$\mathrm{det}(D_{\vX_k}$ $F_k)\neq 0$ for all $1\leq k\leq l$ up to the $l$-th timescale. 
It further iterates the procedure for increasing $l$, computes a rescaled
and truncated version of the CRN at each step by using the 
algorithm ScaleAndTruncate introduced in \cite{kruff2021algorithmic}.

\begin{algorithm}[ht!] 
\begin{algorithmic}[1]
    \caption{\label{alg:transform}$\operatorname{TransformCRNexplicit}$}
 \REQUIRE 
A CRN given by a polynomial vector field $\vF(\vk,\vx)$.  
 \smallskip
\ENSURE A transformed CRN given by a modified polynomial vector field. 
\STATE{ScaleAndTruncate.}
\STATE{$l:=0$}
\WHILE{$l < m$}
\STATE{$l:=l+1$}
\WHILE{$\mathrm{det}(D_{\vX_l}F_l^{(1)}) \neq 0$}
\STATE{$l:=l+1$}
\ENDWHILE
\STATE{Find a complete set $\vect{\Phi}_l$ of independent conservation laws for $\vF^{(1)}_l$
satisfying conditions \eqref{eq:rankflphil} and \eqref{eq:rankphil}.}
\STATE{\label{phisolve} Compute the solution  $\check \vx_l =  \vect{\psi}_{l} (\vX_{l-1},\hat\vx_l,\vx_l^c)$ of the 
equation $\vx_l^c = \vect{\Phi_l}(\vX_{l-1},\hat\vx_l,\check \vx_l)$.}
\FOR{$i: = 1$ \TO $s_l$}
\IF{$\Phi_{il}$ is not exact}
\STATE{Replace the ODE satisfied by $\check{x}_{il}$ by
$\dot{x}_{il}^c = (D_{\vx}{\Phi}_{il}) \vF(\vk,\vx)$.}
\ELSE
\STATE{Delete the ODE satisfied by $\check{x}_{il}$.}
\STATE{Define the new constant $k_{il}^c$ and concatenate it with 
$\vk$.}
\STATE{Substitute $x_{il}^c \leftarrow k_{il}^c$.}
\ENDIF
\STATE{Substitute $\check{x}_{il} \leftarrow \psi_{il} (\vX_{l-1},\hat\vx_l,\vx_l^c)$. }
\ENDFOR
\STATE{ScaleAndTruncate.}
\ENDWHILE
\end{algorithmic}
\end{algorithm}

\begin{algorithm}[ht!] 
\begin{algorithmic}[1]
    \caption{\label{alg:transformimplicit}$\operatorname{TransformCRNimplicit}$}
 \REQUIRE 
A CRN given by a polynomial vector field $\vF(\vk,\vx)$.  
 \smallskip
\ENSURE A differential algebraic CRN given by a modified polynomial vector field and a set of algebraic constraints. 
\STATE{ScaleAndTruncate.}
\STATE{$l:=0$}
\WHILE{$l < m$}
\STATE{$l:=l+1$}
\WHILE{$\mathrm{det}(D_{\vX_l}F_l^{(1)}) \neq 0$}
\STATE{$l:=l+1$}
\ENDWHILE
\STATE{Find a complete set $\vect{\Phi}_l$ of independent conservation laws for $\vF^{(1)}_l$
satisfying conditions \eqref{eq:rankflphil} and \eqref{eq:rankphil}.}
\FOR{$i: = 1$ \TO $s_l$}
\IF{$\phi_{il}$ is not exact}
\STATE{Replace the ODE satisfied by $\check{x}_{il}$ by
$\dot{x}_{il}^c = D_{\vx}{\phi}_{il} \vF(\vk,\vx)$.}
\STATE{Add $\vx_l^c = {\phi_{il}}(\vx)$ to the set of algebraic constraints.}
\ELSE
\STATE{Delete the ODE satisfied by $\check{x}_{il}$}
\STATE{Define new constants $k_{il}^c$ and concatenate them to 
$\vk$.}
\STATE{Add $k_{il}^c = {\phi_{il}}(\vx)$ to the set of algebraic constraints.}
\ENDIF
\ENDFOR
\STATE{ScaleAndTruncate.}
\ENDWHILE
\end{algorithmic}
\end{algorithm}

If none of the approximate conservation laws used in the transformation are exact, then
the resulting CRN has the same number of variables, ODEs and parameters as the initial one. 
Any exact conservation law used in the transformation
reduces the numbers of variables and ODEs 
by one and increases the number of parameters by one. 

Because at each step $l$ the total number of variables can only decrease, the total number of 
variables having timescales slower than $\vx_l$ and remaining to be treated is 
strictly decreasing with $l$.
Therefore, the algorithm terminates in a finite number of steps. 

The applicability of Algorithm~\ref{alg:transform} is limited by the possibility of 
solving the equation $\vx_l^c = \vect{\Phi}_l(\vX_l)$ symbolically (elimination step \ref{phisolve} of the algorithm). This 
is always possible when all the approximate conservation laws $\vect{\Phi}_l$ are linear, but may not be easy when 
the completeness condition \eqref{eq:rankflphil} 
can not be fulfilled without some polynomial conservation laws.
However, for most biochemical CRN models used in 
computational biology, this situation does not arise:
linear conservation laws are enough to obtain completeness. 

If one wants to avoid the elimination step   (for
instance, when there are polynomial conservation laws) there is another possible algorithmic 
solution whose output is 
a differential algebraic system. 
More precisely, at each step $l$ one
considers the truncated vector field $\vF^{(1)}_l(\vk,\vx)$ and the conservation 
law $\vect{\Phi}_l(\vx)$. The former is used for the ODE part of the transformed model
and the latter defines the algebraic constraint \eqref{tzm}.
This choice is implemented in Algorithm~\ref{alg:transformimplicit}.
One should note that the symbolic 
reduction algorithms introduced in \cite{kruff2021algorithmic} also 
use an implicit formulation of the fast variables elimination that leads to differential algebraic reduced systems.
Using Lemma~\ref{lemma:fred}
and Proposition~\ref{prophyper} it follows that the hyperbolicity test 
justifying the reduction of the transformed model should be performed on the eigenvalues of the
Schur complement
$$ \left. \left( D_{{\vX}_l} \left.
\begin{pmatrix} 
\vect{\Phi}_l^{(1)}\\
\vFhat^{(1)}_l    \end{pmatrix} 
\right/ D_{\check{\vx}_l} \vect{\Phi}_l^{(1)}  \right)
\right/
\left( D_{{\vX}_{l-1}} \left.
\begin{pmatrix}
\vect{\Phi}_{l-1}^{(1)} \\
\vFhat^{(1)}_{l-1} 
\end{pmatrix} 
\right/ D_{\check{\vx}_{l-1}} \vect{\Phi}_{l-1}^{(1)}   \right).$$

\section{Conclusions}
In this paper we showed how to transform a system of polynomial 
ODEs with approximate conservation laws into an equivalent system without any 
approximate conservation laws. This allowed us to reduce the transformed 
system using a previously introduced method which uses geometric singular 
perturbation theory for multiple timescales \cite{kruff2021algorithmic}.

The output of our reduction algorithm depends on the choice of a tropical equilibration solution. 
Changing this solution may lead to different timescale orderings of the variables, truncated systems, approximate conservation laws and reduced models. However,  
continuous branches of tropical solutions lead to the
same truncated systems, approximate conservation laws and
ordering of timescales. A branch of tropical equilibation solutions corresponds to a polyhedral domain in the space of logarithms
of species concentrations \cite{desoeuvres_CMSB2022}.
Furthermore, the validity of a given reduction can be extended to neighborhoods of such
polyhedra in logarithms of species concentrations.
For these reasons, the 
reductions based on orders of magnitude 
comparison (including those discussed in this paper) are
 robust 
 \cite{gr08,radulescu2008robust,radulescu2012reduction}.
However, biochemical CRNs are often 
excitable and their trajectories explore very large
domains of the species concentrations space. 
In such cases, the CRN may change the
branch of tropical solutions several times 
along the same trajectory. Thus,
scalings, truncated systems and even approximate
conservation laws may change and several different reductions
must be used along the same trajectory 
\cite{radulescu2015symbolic,samal2016geometric,desoeuvres_CMSB2022}. 
The study of switching between different reduced models asks 
for different mathematical methods such as 
{\em blow-ups} \cite{krupa2001extending} and
will be treated in future work. 

\cor{In our method we consider that fast dynamics relaxes to
a quasi-equilibrium or quasi-steady state. Approximate conservation 
laws can also be relevant in situations when fast dynamics is
periodic. This situation, needing averaging techniques,
has been discussed for perturbed 
Hamiltonian systems in \cite{freidlin2022perturbation}. The long-time behavior of such systems turns to be 
universal and corresponds to slow random motion on the 
graph of connected components of the Hamiltonian level sets \cite{freidlin2022perturbation}. }

The general usefulness of reduced models follows from their reduced
number of variables and parameters. A reduced model can be
more easily simulated, analysed and learned from data. 
Beyond these benefits, the model reduction process 
unfolds useful information about the full model. 
First, it provides a classification of the parameters,
according to their identifiability, that is very useful for machine
learning applications \cite{radulescu2008robust,radulescu2012reduction}. 
Parameters of the full model, not occurring in the reduced model 
are {\em sloppy} in the sense of a lack of sensibility of model
properties with respect to them. Other parameters of the 
full model, occurring in the reduced model in a grouped 
manner, for instance as monomials, are not identifiable
independently. 
The reduction process also outputs  timescales of different
variables. 
This is important for understanding the dynamics of the system and
in certain cases can be used to gain biological understanding. 
In particular, slow variables are involved into 
memory mechanisms, important in learning processes and for the maintenance of the biological identity, whereas 
fast variables are important for complex
responses needed for
adaptation to external changes.


\section{A case study: reduction of a signaling pathway model }
\label{sec:SM1}
\subsection{Model and its scaling}
The TGF-$\beta$ signaling model including including transcriptional repression of SMAD transcription factors 
by TIF1-$\gamma$ is described by 21 ODEs  \cite{andrieux2012dynamic}:
\begin{eqnarray}
\dot x_1  &=& k_2x_2 - k_1x_1 - k_{16}x_1x_{11} \notag \\
\dot x_2  &=& k_1x_1 - k_2x_2 + k_{17}k_{36}x_6 \notag \\
\dot x_3  &=& k_3x_4 - k_3x_3 + k_7x_7 + k_{33}k_{38}x_{20} - k_6x_3x_5 \notag \\
\dot x_4  &=& k_3x_3 - k_3x_4 + k_9x_8 - k_8x_4x_6   \notag \\
\dot x_5  &=& k_5x_6 - k_4x_5 + k_7x_7 + 2k_{11}x_9 - 2k_{10}x_5^2 - k_6x_3x_5 + k_{16}x_1x_{11} \notag \\
\dot x_6 &=& k_4x_5 - k_5x_6 + k_9x_8 + 2k_{13}x_{10} + k_{35}x_{21} - 2k_{12}x_6^2 - k_{17}k_{36}x_6 - k_8x_4x_6 \notag \\
\dot x_7 &=& k_6x_3x_5 - x_7(k_7 + k_{14}) \notag \\
\dot x_8 &=& k_{14}x_7 - k_9x_8 + k_8x_4x_6 - k_{31}x_8x_{17} \notag \\
\dot x_9 &=&  k_{10}x_5^2 - x_9(k_{11} + k_{15}) \notag \\
\dot x_{10} & =& k_{15}x_9 - k_{13}x_{10} + k_{12}x_6^2 \notag \\
\dot x_{11} &=& k_{23}x_{14} - k_{30}x_{11} \notag \\
\dot x_{12} &=& k_{18} - x_{12}(k_{20} + k_{26}) + k_{30}x_{11} + k_{27}x_{15} - k_{22}k_{37}x_{12}x_{13} \notag \\
\dot x_{13}  &=& k_{19} - x_{13}(k_{21} + k_{28}) + k_{30}x_{11} + k_{29}x_{16} - k_{22}k_{37}x_{12}x_{13} \notag \\
\dot x_{14}  &=& k_{22}k_{37}x_{12}x_{13} - x_{14}(k_{23} + k_{24} + k_{25}) \notag \\
\dot x_{15}  &=& k_{26}x_{12} - k_{27}x_{15} \notag \\
\dot x_{16}  &=& k_{28}x_{13} - k_{29}x_{16} \notag 
\end{eqnarray}
\begin{eqnarray}
\dot x_{17}  &=& k_{35}x_{21} - k_{31}x_{8}x_{17} \notag \\
\dot x_{18}  &=& k_{31}x_{8}x_{17} - k_{34}x_{18} \notag \\
\dot x_{19}  &=& k_{34}x_{18} - k_{32}x_{19} \notag \\
\dot x_{20}  &=& k_{32}x_{19} - k_{33}k_{38}x_{20} \notag \\
\dot x_{21}  &=& k_{34}x_{18} - k_{35}x_{21} 
\end{eqnarray}

This model is particularly interesting because it contains multiple exact and approximate conservation laws, and many timescales. 
The model has three exact linear conservation laws 
$x_{17} + x_{18} + x_{21}=k_{39}$, $x_1 + x_2 + x_5 + x_6 + x_7 + x_8 + 2x_9 + 2x_{10} + x_{18} + x_{21}=k_{40}$, $x_3 + x_4 + x_7 + x_8 + x_{18} + x_{19} + x_{20}=k_{41}$, whose constant values 
$k_{39}$,$k_{40}$,$k_{41}$ can be interpreted as the total amounts of TIF1-$\gamma$, SMAD2, and SMAD4, respectively.

We propose a reduction based on the total tropical equilibration 
$$d=(-2,    -1,    -2,    -2,     0,     0,   -1,    -1,     1,     1,     1,     1,     4,     2,     0,     3,    -1,    -1,    -1),$$ computed for $\epsilon=1/11$. 
This total equilibration solution is the closest, in logarithmic coordinates, to the steady state of the TGFb model. 

The rescaled system of ODEs is
\begin{eqnarray}\label{eq:tgfbrescaled}
\dot y_1  &=&	 \epsilon^2(\bar k_2 y_2- \bar k_1 y_1 - \epsilon^2 \bar k_{16} y_1 y_{11}) \notag \\
\dot y_2  &=&	 \epsilon^1(\bar k_1 y_1 + \epsilon^2 \bar k_{17} \bar k_{36} y_6-\bar k_2 y_2) \notag \\
\dot y_3  &=&	 \epsilon^2(\bar k_3 y_4 + \epsilon \bar k_7 y_7 + \epsilon \bar k_{33} \bar k_{38} y_{20}- \bar k_3 y_3 - \epsilon \bar k_6 y_3 y_5) \notag \\
\dot y_4  &=&	 \epsilon^2(\bar k_3 y_3 + \epsilon \bar k_9 y_8- \bar k_3 y_4 - \epsilon \bar k_8 y_4 y_6) \notag \\
\dot y_5  &=&	 \epsilon^1(\bar k_5 y_6 + \bar k_7 y_7 + 2 \epsilon^2 \bar k_{11} y_9 + \epsilon \bar k_{16} y_1 y_{11}- 2 \epsilon^2 \bar k_{10} y_5^2- \notag \\
&&- \epsilon \bar k_4 y_5 - \bar k_6 y_3 y_5) \notag \\
\dot y_6  &=&	 \epsilon^1(\bar k_9 y_8 + \bar k_{35} y_{21} + 2 \epsilon^2 \bar k_{13} y_{10} + \epsilon \bar k_4 y_5- \bar k_5 y_6 -  \notag \\
&&- 2 \epsilon^2 \bar k_{12} y_6^2 - \bar k_8 y_4 y_6 - \epsilon \bar k_{17} \bar k_{36} y_6) \notag \\
\dot y_7  &=&	 \epsilon^2(\bar k_6 y_3 y_5- \bar k_7 y_7 - \bar k_{14} y_7) \notag \\
\dot y_8  &=&	 \epsilon^2(\bar k_{14} y_7 + \bar k_8 y_4 y_6- \bar k_9 y_8 - \bar k_{31} y_8 y_{17}) \notag \\
\dot y_9  &=&	 \epsilon^2(\bar k_{10} y_5^2- \bar k_{11} y_9 - \bar k_{15} y_9) \notag \\
\dot y_{10}  &=&	 \epsilon^2(\bar k_{15} y_9 + \bar k_{12} y_6^2-\bar k_{13} y_{10}) \notag \\
\dot y_{11}  &=&	 \epsilon^3(\bar k_{23} y_{14}-\bar k_{30} y_{11}) \notag \\
\dot y_{12}  &=&	 \epsilon^2(\epsilon \bar k_{18} + \bar k_{27} y_{15} + \epsilon \bar k_{30} y_{11}- \bar k_{26} y_{12} - \epsilon \bar k_{20} y_{12} - \epsilon \bar k_{22} \bar k_{37} y_{12} y_{13}) \notag \\
\dot y_{13}  &=&	 \epsilon^0(\bar k_{19} + \bar k_{30} y_{11} + \epsilon^2 \bar k_{29} y_{16}- \epsilon^3 \bar k_{21} y_{13} - \epsilon^2 \bar k_{28} y_{13} - \bar k_{22} \bar k_{37} y_{12} y_{13}) \notag \\
\dot y_{14}  &=&	 \epsilon^2(\bar k_{22} \bar k_{37} y_{12} y_{13}- \bar k_{23} y_{14} - \bar k_{25} y_{14} - \epsilon \bar k_{24} y_{14}) \notag 
\end{eqnarray}
\begin{eqnarray}
\dot y_{15}  &=&	 \epsilon^3(\bar k_{26} y_{12}-\bar k_{27} y_{15} ) \notag \\
\dot y_{16}  &=&	 \epsilon^3(\bar k_{28} y_{13}-\bar k_{29} y_{16}) \notag \\
\dot y_{17}  &=&	 \epsilon^2(\bar k_{35} y_{21} -\bar k_{31} y_{8} y_{17}) \notag \\
\dot y_{18}  &=&	 \epsilon^2(\bar k_{31} y_{8} y_{17}-\bar k_{34} y_{18}) \notag \\
\dot y_{19}  &=&	 \epsilon^2(\bar k_{34} y_{18}-\bar k_{32} y_{19}) \notag \\
\dot y_{20}  &=&	 \epsilon^1(\bar k_{32} y_{19}-\bar k_{33} \bar k_{38} y_{20}) \notag \\
\dot y_{21}  &=&	 \epsilon^2(\bar k_{34} y_{18}-\bar k_{35}y_{21}), 
\end{eqnarray}
and the truncated rescaled system is 
\begin{eqnarray}\label{eq:tgfbtruncated}
\dot y_1  &=&	 \epsilon^2(\bar k_2 y_2- \bar k_1 y_1 ) \notag \\
\dot y_2  &=&	 \epsilon^1(\bar k_1 y_1 -\bar k_2 y_2) \notag \\
\dot y_3  &=&	 \epsilon^2(\bar k_3 y_4 - \bar k_3 y_3 ) \notag \\
\dot y_4  &=&	 \epsilon^2(\bar k_3 y_3 - \bar k_3 y_4 ) \notag \\
\dot y_5  &=&	 \epsilon^1(\bar k_5 y_6 + \bar k_7 y_7  - \bar k_6 y_3 y_5) \notag \\
\dot y_6  &=&	 \epsilon^1(\bar k_9 y_8 + \bar k_{35} y_{21}  - \bar k_8 y_4 y_6 ) \notag \\
\dot y_7  &=&	 \epsilon^2(\bar k_6 y_3 y_5- \bar k_7 y_7 - \bar k_{14} y_7) \notag \\
\dot y_8  &=&	 \epsilon^2(\bar k_{14} y_7 + \bar k_8 y_4 y_6- \bar k_9 y_8 - \bar k_{31} y_8 y_{17}) \notag \\
\dot y_9  &=&	 \epsilon^2(\bar k_{10} y_5^2- \bar k_{11} y_9 - \bar k_{15} y_9) \notag \\
\dot y_{10}  &=&	 \epsilon^2(\bar k_{15} y_9 + \bar k_{12} y_6^2-\bar k_{13} y_{10}) \notag \\
\dot y_{11}  &=&	 \epsilon^3(\bar k_{23} y_{14}-\bar k_{30} y_{11}) \notag \\
\dot y_{12}  &=&	 \epsilon^2(\bar k_{27} y_{15} - \bar k_{26} y_{12} ) \notag \\
\dot y_{13}  &=&	 \epsilon^0(\bar k_{19} + \bar k_{30} y_{11} - \bar k_{22} \bar k_{37} y_{12} y_{13}) \notag \\
\dot y_{14}  &=&	 \epsilon^2(\bar k_{22} \bar k_{37} y_{12} y_{13}- \bar k_{23} y_{14} - \bar k_{25} y_{14} ) \notag \\
\dot y_{15}  &=&	 \epsilon^3(\bar k_{26} y_{12}-\bar k_{27} y_{15} ) \notag \\
\dot y_{16}  &=&	 \epsilon^3(\bar k_{28} y_{13}-\bar k_{29} y_{16}) \notag \\
\dot y_{17}  &=&	 \epsilon^2(\bar k_{35} y_{21} -\bar k_{31} y_{8} y_{17}) \notag \\
\dot y_{18}  &=&	 \epsilon^2(\bar k_{31} y_{8} y_{17}-\bar k_{34} y_{18}) \notag \\
\dot y_{19}  &=&	 \epsilon^2(\bar k_{34} y_{18}-\bar k_{32} y_{19}) \notag \\
\dot y_{20}  &=&	 \epsilon^1(\bar k_{32} y_{19}-\bar k_{33} \bar k_{38} y_{20}) \notag \\
\dot y_{21}  &=&	 \epsilon^2(\bar k_{34} y_{18}-\bar k_{35}y_{21}), 
\end{eqnarray}
After this scaling four timescales are apparent,  in order from the fastest to the slowest: $\epsilon^0$, $\epsilon^1$, $\epsilon^2$, $\epsilon^3$.
The corresponding groups of variables 
having these timescales
are, in order from the fastest to the slowest: $\vz_1=y_{13}$, $\vz_2 = (y_2,y_5,y_6, y_{20})$, 
$\vz_3=(y_1,y_3,y_4,y_7,y_8,y_9,y_{10},y_{12},y_{14},y_{17}$ $,y_{18},y_{19},y_{21})$,$\vz_4=(y_{11},y_{15},y_{16})$.
Thus $n_1=1$, $n_2=4$, $n_3=13$, $n_4=3$.

\subsection{Elimination of the conservation laws}
We now transform the model into an equivalent one that
has no conservation laws. 
At the first iteration, step 5 of the Algorithm~\ref{alg:transform}, we find
$|D_{\vX_1} \vF^{(1)}_1| = -k_{22}k_{37}x_{12} \neq 0$,  
$|D_{\vX_2} \vF^{(1)}_2| = -k_2k_6k_{22}k_{33}k_{37}k_{38}x_3x_{12}(k_5 + k_8x_4) \neq 0 $, but
$|D_{\vX_3} \vF^{(1)}_3| = 0$. Thus $l=3$ 
at the step 8.

At the step 8, building a stoichiometric matrix $\vect{S}^{(1)}_3$ from $\vF^{(1)}_3$ we find 
four linear, independent, approximate conservation laws:
$x_1 +x_2$, $x_3+x_4$, $x_5 + x_6 + x_7 + x_8 + x_{18} + x_{21}$, $x_{17} + x_{18} + x_{21}$. The last one 
is an exact conservation law that we have already interpreted. The first three 
approximate conservation laws 
can be interpreted as: the total free  SMAD2,
the total free SMAD4, and the total phosphorylated SMAD2 free or forming heterodimers 
(excluding pS22c, pS22n that are homodimers, and pS24nTIF, pS2nTIF that are trimers), respectively. 
We have $s_3=4$.



At the step 9, we choose $\check \vx_3=(x_1,x_3,x_8,x_{21})$, that at step 13
are substituted as
$x_1 \leftarrow x_1-x_2$ ($x_1^c=x_1+x_2$ is renamed $x_1$), 
$x_3 \leftarrow x_3-x_4$ ($x_2^c=x_3+x_4$ is renamed $x_3$),
$x_8 \leftarrow -x_5 - x_6 - x_7 + x_8 - x_{18} - x_{21}$ ($x_3^c = x_5 + x_6 + x_7 + x_8 + x_{18} + x_{21}$ is renamed $x_8$),
$x_{21} \leftarrow k_{39}-(x_{17} + x_{18})$ ($x_4^c=x_{17} + x_{18} + x_{21}$ is renamed $k_{39}$, a parameter because the last conservation law is exact). 
Because the old variables 
are positive, the new variables must obey $x_1 \geq x_2$, $x_3 \geq x_4$,
$x_8 \geq x_5 + x_6 + x_7 + x_8 + x_{18} + x_{21}$ and $k_{39} \geq x_{17} + x_{18}$. 

After this substitution, the ScaleandTruncate step 20 reveals a fifth, slower timescale  of order $\epsilon^4$, that results from approximate conservation
laws. 

At the step 8 of the second iteration we get $l=4$, $|D_{\vX_k} \vF^{(1)}_k| \neq 0, k\in\{1,3\}$
 and $|D_{\vX_4} \vF^{(1)}_4| = 0$.
We find then two linear approximate conservation laws
$x_3 - x_5 - x_6 + x_{17} + x_{18} + x_{19} + x_{20}$,
$x_{12} + x_{15}$. In initial variables,
the first one corresponds to 
$x_3+x_4 - x_5 - x_6 + x_{17} + x_{18} + x_{19} + x_{20}$.
At this iteration $\check \vx_4=(x_3,x_{15})$
are substituted as
$x_3 \leftarrow x_3 + x_5 + x_6 - x_{17} - x_{18} - x_{19} - x_{20}$,
$x_{15} \leftarrow -x_{12} + x_{15}$.
All the new variables have timescales $\epsilon^4$.

At  the step 8 of the third iteration we find $l=5$, $|D_{\vX_k} \vF^{(1)}_k| \neq 0, k\in\{1,4\}$
but $|D_{\vX_5} \vF^{(1)}_5| = 0$. 
We get two new approximate conservation laws
$x_3+x_8$ and $x_3-x_1$. The first one is  an exact conservation as in initial
variables is 
$x_3+x_4   + x_7 + x_8  + x_{18} + x_{19} + x_{20} + x_{17} + x_{18} + x_{21}=k_{39}+k_{41}$.
At this iteration $\check \vx_5=(x_1,x_3)$
are substituted as
$x_1 \leftarrow x_3 - x_1$, 
$x_3 \leftarrow k_{39} + k_{41} - x_8$.
After this iteration, a sixth timescale occurs, of order $\epsilon^5$ for the variable $x_1$. 

At  step 8 of the fourth iteration $l=6$, 
 $|D_{\vX_k} \vF^{(1)}_k| \neq 0, k\in\{1,5\}$
but $|D_{\vX_6} \vF^{(1)}_6| = 0$.
We identify one more, exact, conservation law $2 x_{10} + 2 x_9 - x_1$ that in initial
variables represents $x_1 + x_2 + x_5 + x_6 + x_7 + x_8 + 2x_9 + 2x_{10} + x_{18} + x_{21} = k_{40}-k_{39}-k_{41}$. The variable $x_1$ is eliminated and the timescale  of order $\epsilon^5$
disappears. The substitution is $x_1 \leftarrow 2x_{10} + 2 x_9 - k_{40}+k_{39}+k_{41}$.
After the fourth iteration the full Jacobian matrix is regular and there are no more conservation law, approximate or exact. 
Five timescales remain, of orders $\epsilon^0$,$\epsilon^1$,$\epsilon^2$,$\epsilon^3$,$\epsilon^4$.



To summarize, 6 approximate and 3 exact 
conservation laws were used in this
transformation. The 3 exact conservation laws
were used to eliminate 3 of the initial 
system variables, see Table~\ref{table:transformed_steps}. Among 
the 6 approximate conservation laws, 2 
were kept as variables in the final transformed model,
the other being substituted at different steps
of the procedure. 
The final transformed
model has a reduced dimensionality (18 variables)
and no conservation laws. The transformed model
reads:
\begin{eqnarray}
\dot x_2 &=&  k_{17}  k_{36}   x_{6}  -  k_{1}  ( x_{2}  -  k_{40}  +  x_{8}  + 2  x_{9}  + 2  x_{10} ) -  k_{2}   x_{2},  \notag \\
\dot x_4 &=& 	 -  k_{3}  ( x_{4}  -  k_{41}  -  k_{39}  -  x_{5}  -  x_{6}  +  x_{8}  +  x_{17}  +  x_{18}  +  x_{19}  +  x_{20} ) -  k_{3}   x_{4}  -\notag \\
 &&- k_{9}  ( k_{39}  +  x_{5}  +  x_{6}  +  x_{7}  -  x_{8}  -  x_{17} ) -  k_{8}   x_{4}   x_{6},  \notag \\
\dot x_5 &=& 	  k_{5}   x_{6}  -  k_{4}   x_{5}  +  k_{7}   x_{7}  + 2  k_{11}   x_{9}  - 2  k_{10}   x_{5}^2 -  k_{16}   x_{11}  ( x_{2}  -  k_{40}  +  x_{8}  + 2  x_{9}+\notag \\
&& + 2  x_{10} )+k_{6}   x_{5}  ( x_{4}  -  k_{41}  -  k_{39}  -  x_{5}  -  x_{6}  +  x_{8}  +  x_{17}  +  x_{18}  +  x_{19}  +  x_{20} ), \notag \\
\dot x_6 &=& 	  k_{4}   x_{5}  -  k_{5}   x_{6}  + 2  k_{13}   x_{10}  - 2  k_{12}   x_{6}^2 -  k_{9}  ( k_{39}  +  x_{5}  +  x_{6}  +  x_{7}  -  x_{8}  -  x_{17} ) -   \notag \\
&&- k_{35}  ( x_{17}  -  k_{39}  +  x_{18} ) -k_{17}   k_{36}   x_{6}  -  k_{8}   x_{4}   x_{6},  \notag \\
\dot x_7 &=& 	 -  x_{7}  ( k_{7}  +  k_{14} ) -  k_{6}   x_{5}  ( x_{4}  -  k_{41}  -  k_{39}  -  x_{5}  -  x_{6}  +  x_{8}  +  x_{17}  +  x_{18}  + \notag \\
&& + x_{19 } +  x_{20} ), \notag \\
\dot x_8 &=& 	  k_{7}   x_{7}  -  x_{7}  ( k_{7}  +  k_{14} ) + 2  k_{11}   x_{9}  +  k_{14}   x_{7}  +
2  k_{13}   x_{10}  - 2  k_{10}   x_{5} ^2 - 2  k_{12}   x_{6} ^2 - \notag \\
&& - k_{16}   x_{11} 
( x_{2}  -  k_{40}  +  x_{8}  + 2  x_{9}  + 2  x_{10} ) -  k_{17}   k_{36}   x_{6},  \notag \\
\dot x_9 &=& 	  k_{10}   x_{5}^2 -  x_{9}  ( k_{11}  +  k_{15} ), \notag \\
\dot x_{10} &=& 	  k_{15}   x_{9}  -  k_{13}   x_{10}  +  k_{12}   x_{6}^2, \notag \\ 
\dot x_{11} &=& 	  k_{23}   x_{14}  -  k_{30}   x_{11} , \notag \\
\dot x_{12} &=& 	  k_{18}  -  x_{12 } ( k_{20}  +  k_{26} ) +  k_{30}   x_{11}  - 
k_{27}  ( x_{12}  -  x_{15} ) -  k_{22}   k_{37}   x_{12}   x_{13 }, \notag \\
\dot x_{13} &=& 	  k_{19}  -  x_{13}  ( k_{21}  +  k_{28} ) +  k_{30}   x_{11}  +  k_{29}   x_{16}  
-  k_{22}   k_{37}   x_{12}   x_{13},  \notag \\
\dot x_{14} &=& 	  k_{22}   k_{37}   x_{12}   x_{13}  -  x_{14}  ( k_{23}  +  k_{24}  +  k_{25} ), \notag \\
\dot x_{15} &=& 	  k_{18}  -  x_{12}  ( k_{20 } +  k_{26} ) +  k_{26 }  x_{12 } +  k_{30}   x_{11 } -  k_{22}   k_{37}   x_{12}   x_{13},  \notag \\
\dot x_{16} &=& 	  k_{28}   x_{13}  -  k_{29 }  x_{16 }, \notag \\
\dot x_{17} &=& 	  k_{31}   x_{17}  ( k_{39}  +  x_{5}  +  x_{6}  +  x_{7}  -  x_{8 } -  x_{17} ) -  k_{35}  ( x_{17}  -  k_{39}  +  x_{18} ), \notag \\
\dot x_{18} &=& 	 -  k_{34}   x_{18}  -  k_{31}   x_{17}  ( k_{39}  +  x_{5}  +  x_{6}  +  x_{7}  -  x_{8}  -  x_{17} ), \notag \\
\dot x_{19} &=& 	  k_{34}   x_{18}  -  k_{32}   x_{19},  \notag \\
\dot x_{20} &=& 	  k_{32}   x_{19}  -  k_{33}   k_{38}  x_{20},  
\end{eqnarray}
and the truncated rescaled transformed model reads
\begin{eqnarray}
     \dot y_2 &=& \epsilon^1(k_1 k_{40}-k_2 y_2) , \notag  \\ 
    \dot y_4 &=&  \epsilon^2(k_3 k_{41}-2 k_3 y_4) , \notag \\ 
	\dot y_5 &=&  \epsilon^1(k_5 y_6 + k_7 y_7 + k_6 y_4 y_5-k_6 k_{41} y_5) , \notag \\ 
\dot y_6 &=& 	 \epsilon^1(k_{35} k_{39} + k_9 y_8 + k_9 y_{17}- k_9 k_{39} - k_5 y_6 - k_9 y_7-\notag \\ 
&& - k_{35} y_{17} - k_{35} y_{18} - k_8 y_4 y_6) , \notag \\ 
\dot y_7 &=& 	 \epsilon^2(k_6 k_{41} y_5- k_7 y_7 - k_14 y_7 - k_6 y_4 y_5) , \notag \\ 
\dot y_8 &=& 	 \epsilon^3(k_{16} k_{40} y_{11}-k_{17} k_{36} y_{6}) , \notag \\ 
\dot y_9 &=& 	 \epsilon^2(k_{10} y_5^2- k_{11} y_9 - k_{15} y_9) , \notag 
 \end{eqnarray}
\begin{eqnarray}
\dot y_{10} &=& 	 \epsilon^2(k_{15} y_9 + k_{12} y_6^2-k_{13} y_{10}) , \notag \\ 
\dot y_{11} &=& 	 \epsilon^3(k_{23} y_{14}-k_{30} y_{11}) , \notag \\ 
\dot y_{12} &=& 	 \epsilon^2(k_{27} y_{15}-k_{26} y_{12}) , \notag \\ 
\dot y_{13} &=& 	 \epsilon^0(k_{19} + k_{30} y_{11}-k_{22} k_{37} y_{12} y_{13}) , \notag \\ 
\dot y_{14} &=& 	 \epsilon^2(k_{22} k_{37} y_{12} y_{13}- k_{23} y_{14} - k_{25} y_{14}) , \notag \\ 
\dot y_{15} &=& 	 \epsilon^4(k_{18} + k_{30} y_{11}- k_{20} y_{12} - k_{22} k_{37} y_{12} y_{13}) , \notag \\ 
\dot y_{16} &=& 	 \epsilon^3(k_{28} y_{13}-k_{29} y_{16}) , \notag \\ 
\dot y_{17} &=& 	 \epsilon^2(k_{35} k_{39} + k_{31} k_{39} y_{17} + k_{31} y_7 y_{17}- k_{35} y_{17} - k_{35} y_{18} - k_{31} y_{17}^2 -\notag \\ 
&&- k_{31} y_8 y_{17}) , \notag \\ 
\dot y_{18} &=& 	 \epsilon^2(k_{31} y_{17}^2 + k_{31} y_8 y_{17}- k_{34} y_{18} - k_{31} k_{39} y_{17} - k_{31} y_7 y_{17}) , \notag \\ 
\dot y_{19} &=& 	 \epsilon^2(k_{34} y_{18}-k_{32} y_{19}), \notag\\ 
\dot y_{20} &=& 	 \epsilon^1(k_{32} y_{19}-k_{33} k_{38} y_{20}). 
\label{eq:tgfbtransformedtruncated}
\end{eqnarray}
As can be seen from \eqref{eq:tgfbtransformedtruncated}, this method unravels one new timescale that was not 
apparent in the initial rescaled model \eqref{eq:tgfbtruncated}. 

The new variables of the transformed model can be expressed in the old variables of the initial model as  shown in Table~\ref{table:transformed_variables}.
Some of the variables remain the same after the transformation.
In order to find the inverse transformation, from new variable $x_i$ to old
variables $x_i^o$, we need to gather the definitions of the variables
that change, namely $x_8$ and $x_{15}$ and the definitions
of the three exact conservation laws
that were used to eliminate three old variables. More precisely,
we have to solve
\begin{eqnarray} \label{eq:varchange}
x_8 &=& x_5+x_6+x_7+x^o_8+x_{18}+x^o_{21}\notag \\
x_{15} &=& x_{12}+x^o_{15}\notag\\
k_{39}&=&x_{17}+x_{18}+x^o_{21}\notag\\
k_{40}&=&x^o_1+x_2+x_5+x_6+x_7+x^o_8+2x_9+2x_{10}+x_{18}+x^o_{21}\notag\\
k_{41}&=&x^o_3+x_4+x_7+x^o_8+x_{18}+x_{19}+x_{20},
\end{eqnarray}
leading to 
\begin{eqnarray}
x^o_1 &=& k_{40} - x_2 -x_8- 2x_9 - 2x_{10} \notag \\
x^o_{3} &=& k_{39} + k_{41} - x_4 + x_5 + x_6 - x_8 - x_{17} - x_{18} - x_{19} - x_{20} \notag\\
x^o_8 &=&  x_8 + x_{17} - x_5 - x_6 - x_7 - k_{39} \notag\\
x^o_{15}&=&x_{15} - x_{12}\notag\\
x^o_{21}&=&k_{39} - x_{17} - x_{18}.
\end{eqnarray}
Because all the old variables are positive, the new variables have to satisfy the following constraints:
\begin{eqnarray}
 k_{40} - x_2 -x_8- 2x_9 - 2x_{10} \geq  0 \notag \\
k_{39} + k_{41} - x_4 + x_5 + x_6 - x_8 - x_{17} - x_{18} - x_{19} - x_{20} \geq 0 \notag\\
 x_8 + x_{17} - x_5 - x_6 - x_7 - k_{39} \geq 0 \notag\\
x_{15} - x_{12} \geq 0 \notag\\
k_{39} - x_{17} - x_{18} \geq 0.
\end{eqnarray}
This result is well known for CRNs, as exact conservation laws are often used for reducing model reduction. The resulting CRNs have variables constrained to polytopes. In our case, the number of constraints is larger, because not only exact, but also approximate conservation laws, are used for the reduction.  

Contrary to the reduction by exact conservation laws elimination when the remaining variables are  chemical species, our transformed model
contains two variables representing pools of chemical species. According to \eqref{eq:varchange} $x_{15}$ represents the total type 1 free receptor RI,
and $x_{8}$ represents the total phosphorylated SMAD2, except those in homodimers or complexified with TIF1-$\gamma$.

\subsection{Reduced models}

The non-degeneracy condition being satisfied, the transformed
model can be now further reduced by successive elimination of 
the fast variables. The hyperbolicity condition can be tested
with methods exposed in \cite{kruff2021algorithmic}.
The reduced models at various last slow timescales
are summarized in the Table~\ref{table:reduced}.

\begin{sidewaystable}
  \begin{tabular}{|l|l|l|l|l|l|}
  \hline  \scriptsize
  \textbf{$l$}  &  \textbf{Groups of variables} & \scriptsize \textbf{ToV}  & \textbf{Conservation laws}   & \scriptsize \textbf{Interpretation} & \scriptsize \textbf{ToC}  \\ \hline
  \scriptsize 
 $1$ & $\vx_1 = x_{13}$ & $\epsilon^0$  & none & &  \\ \scriptsize
 $2$  & $\vx_2 = (x_{2},x_5,x_6,x_{20})$ & $\epsilon^1$  & none & &  \\ \scriptsize
 $3$     & $\vx_3 = (x_1,x_3,x_4,x_7,x_8,x_9,x_{10},$ & $\epsilon^2$  & $x_1+x_2$ &  \scriptsize total free SMAD2 & $\epsilon^4$ \\ \scriptsize
$3$     & $x_{14},x_{17},x_{18},x_{19},x_{21})$ & $\epsilon^2$  & $x_3+x_4$ &  \scriptsize total free SMAD4 & $\epsilon^3$ \\ \scriptsize
$3$     &  & $\epsilon^2$  & $x_5+x_6+x_7+x_8+x_{18}+x_{21}$ &  \scriptsize total pSMAD2  &  $\epsilon^3$\\ \scriptsize
$3$     &  & $\epsilon^2$  & $x_{17}+x_{18}+x_{21}$ &  \scriptsize total TIF & $\epsilon^{\infty}$ \\ \scriptsize
$4$          & $\vx_4 = (x_{11},x_{15},x_{16},x_3+x_4,$ & $\epsilon^3$  & $x_{12}+x_{15}$ &  \scriptsize total free RI  & $\epsilon^4$ \\ \scriptsize
$4$          & $x_5+x_6+x_7+x_8+x_{18}+x_{21})$ & $\epsilon^3$  & $x_3+x_4-x_5-x_6+x_{17}+x_{18}+x_{19}+x_{20}$ &  \scriptsize total pSMAD2  & $\epsilon^4$ \\ \scriptsize
$5$          & $\vx_5 = (x_1+x_2,x_3+x_4-x_5-x_6+$ & $\epsilon^4$  & $x_3+x_4+x_{17}+x_{19}+x_{20}-(x_1 +x_2 +x_5+x_6)$ &  \scriptsize total SMAD2 &$\epsilon^5$ \\ \scriptsize
$5$          & $+x_{17}+x_{18}+x_{19}+x_{20},x_{12}+x_{15})$ & $\epsilon^4$  & $x_3+x_4+x_7+x_8+x_{17}+2x_{18}+x_{19}+x_{20}$ 
&  \scriptsize total SMAD4 & $\epsilon^{\infty}$ \\ \scriptsize
$6$          & $\vx_6 = x_3+x_4+x_{17}+x_{19}+$ & $\epsilon^5$  & 
$x_1+x_2+x_5+x_6+2x_9+2x_{10}-$ & & $\epsilon^{\infty}$ \\
          & $+x_{20}-(x_1 +x_2 +x_5+x_6)$ &  & 
\hfill $-(x_3+x_4+x_{17}+x_{18}+x_{19}+x_{20})$ & & \\
      \hline
  \end{tabular}  
  \caption{Transformed model: variable sets and conservation laws at various iterations. The ToV column contains the  timescale order of the slowest variable included in the conservation laws (slower means higher order) and the ToC column contains the timescale order of the variable resulting from the conservation law. One can check that ToC$>$ToV: conservation laws are slower than the variables composing them.  Exact conservation laws have infinite timescale orders. \label{table:transformed_steps}}
\end{sidewaystable}

\begin{sidewaystable}
  \begin{tabular}{|l|l|l|l|}
  \hline  
  \textbf{Variable}  & \textbf{Definition in old variables} & \textbf{Timescale}  & \textbf{Interpretation}  \\ \hline
  \footnotesize 
 $x_{2}$  &  $x_{2}$ &  $\epsilon^1$ & SMAD2n\\
 $x_{4}$  &  $x_{4}$ &  $\epsilon^2$& SMAD4n\\
 $x_{5}$  &  $x_{5}$ &  $\epsilon^1$& pSMAD2c\\
 $x_{6}$  &  $x_{6}$ &  $\epsilon^1$& pSMAD2n\\
 $x_{7}$  &  $x_{7}$ &  $\epsilon^2$& pSMAD24c\\
  $x_{8}$  &  $x_5+x_6+x_7+x_8+x_{18}+x_{21}$ & $\epsilon^3$ & total pSMAD2 without pSMAD22 \\
   $x_{9}$  &  $x_{9}$ & $\epsilon^2$ & pSMAD22c \\
      $x_{10}$  &  $x_{10}$ & $\epsilon^2$ & pSMAD22n\\
         $x_{11}$  &  $x_{11}$ &$\epsilon^3$  & LRe \\
            $x_{12}$  &  $x_{12}$ &$\epsilon^2$  & RI\\
               $x_{13}$  &  $x_{13}$ &$\epsilon^0$ & RII\\
                  $x_{14}$  &  $x_{14}$ & $\epsilon^2$ & LR\\
      $x_{15}$  &  $x_{12}+x_{15}$ & $\epsilon^4$ & total free RI\\
      $x_{16}$  &  $x_{16}$ & $\epsilon^3$& RIIe\\ 
      $x_{17}$  &  $x_{17}$ &$\epsilon^2$ & TIF \\ 
      $x_{18}$  &  $x_{18}$ &$\epsilon^2$ & pSMAD24nTIF \\ 
      $x_{19}$  &  $x_{19}$ & $\epsilon^2$& SMAD4ubn\\ 
      $x_{20}$  &  $x_{20}$ & $\epsilon^1$& SMAD4ubc \\ 
      \hline
  \end{tabular}   
  \caption{Transformed model: final variables and their interpretation. 
  The variables of the transformed model are all positive and must also satisfy 
  $k_{40} - x_2 -x_8- 2x_9 - 2x_{10} \geq  0$,
$k_{39} + k_{41} - x_4 + x_5 + x_6 - x_8 - x_{17} - x_{18} - x_{19} - x_{20} \geq 0$,
 $x_8 + x_{17} - x_5 - x_6 - x_7 - k_{39} \geq 0$,
$x_{15} - x_{12} \geq 0$,
$k_{39} - x_{17} - x_{18} \geq 0$. \label{table:transformed_variables}
  }
\end{sidewaystable}

\begin{sidewaystable}
  \begin{tabular}{|l|l|l|}
  \hline  \scriptsize
  \textbf{T}  & \textbf{ODEs} & \textbf{Fast variables}  \\ \hline
  \scriptsize  $\epsilon^{4}$
   & \scriptsize 
  $\dot x_{15}=k_{18} - x_{12}(k_{20} + k_{26}) + $ 
   & \scriptsize 
   $x_2=\frac{k_1k_{40}}{k_2}, x_4=\frac{k_{41}}{2},
   x_5=\frac{2k_5k_{16}k_{19}k_{23}k_{40}(k_7 + k_{14})}{k_6k_{14}k_{17}k_{25}k_{30}k_{36}k_{41}},
x_6=\frac{k_{16}k_{19}k_{23}k_{40}}{k_{17}k_{25}k_{30}k_{36}},
x_7=\frac{k_5k_{16}k_{19}k_{23}k_{40}}{k_{14}k_{17}k_{25}k_{30}k_{36}},$ \\
&\scriptsize $+ k_{26}x_{12} + k_{30}x_{11} - k_{22}k_{37}x_{12}x_{13}$    &\scriptsize
$x_8=\frac{(2 k_5 (k_9 +  k_{14} ) +k_{8} k_{14}   k_{41} ) k_{16} k_{19} k_{23} k_{40} + 
  + k_{14} k_{17} k_{25} k_{30} k_{36}(
    2 (k_9 -   k_{35})  (k_{39} - x_{17})
    +2  k_{35}  x_{18})
    }{2 k_9 k_{14} k_{17} k_{25} k_{30} k_{36}},$
        \\
&&\scriptsize  
$a=2 k_{17} k_{25} k_{30} k_{31} k_{34} k_{35} k_{36},$ 
$c=-2 k_9 k_{17} k_{25} k_{30} k_{34} k_{35} k_{36} k_{39},$      \\
&&\scriptsize
$b=k_5 k_{16} k_{19} k_{23} k_{31} (2 +k_8k_{41})(k_{34}+k_{35}) k_{40}  + 
     k_{39}(c-a )$, $x_{17} = (-b + \sqrt{b^2 - 4 a c})/(2 a),$ \\
&&\scriptsize
$x_{18}=\frac{-k_{34}( (2 k_5+k_8k_{41})   k_{16} k_{19} k_{23} k_{31} k_{40}+ 2 (k_9-k_{31} k_{39}) k_{17} k_{25} k_{30} k_{35} k_{36}  )  x_{17} + 
    + a x_{17}^2 +c  }{k_{34}(2 k_{17} k_{25} k_{30} k_{35} k_{36} (k_9 + k_{31} x_{17}))},$
    \\
&&\scriptsize
$x_{19}=\frac{k_{34} x_{18}}{k_{32}},$  $x_{20}=\frac{k_{34} x_{18}}{k_{33} k_{38}},$  \\  
&&\scriptsize
$x_9=\frac{4 k_5^2 k_{10} k_{16}^2 k_{19}^2 k_{23}^2 k_{40}^2 (k_7 + k_{14})^2}{
    k_6^2 k_{14}^2 k_{17}^2 k_{25}^2 k_{30}^2 k_{36}^2 k_{41}^2 (k_{11} + k_{15})},$ 
  
$x_{13}=\frac{k_{19} (k_{23} + k_{25})k_{26} }{k_{22} k_{25} k_{27} k_{37} x_{15}},
x_{14}=\frac{k_{19}}{k_{25}},
x_{16}=\frac{k_{19} k_{26} k_{28} (k_{23} + k_{25})}{k_{22} k_{25} k_{27} k_{29} k_{37} x_{15}}$
    \\
&&\scriptsize
$x_{10}=\frac{k_{16}^2 k_{19}^2 k_{23}^2 k_{40}^2 
(4 k_5^2 k_{15} k_{10} (k_7+k_{14})^2 + k_6^2 k_{12} k_{14}^2 (k_{15}+k_{11}) k_{41}^2  
    )}
    {k_6^2 k_{13} k_{14}^2 k_{17}^2 k_{25}^2 k_{30}^2 k_{36}^2 k_{41}^2 (k_{11} + k_{15})},$
    $x_{11}=\frac{k_{19} k_{23}}{k_{25} k_{30}}, x_{12}=\frac{k_{27} x_{15}}{k_{26}}.$
    \\ \hline
 \scriptsize $\epsilon^3$  & \scriptsize
 $\dot x_8 = 
- 2 k_{10} x_5^2 - 2 k_{12} x_6^2 - k_{16} x_{11} (x_2 - k_{40} + x_8 +$
 & \scriptsize $x_2=\frac{k_1k_{40}}{k_2}, x_4=\frac{k_{41}}{2},$ 
 $c=-k_{34} k_{35} k_{39} (2 k_5 k_9 + 2 k_5 k_{14} + k_8 k_{14} k_{41}),$
 \\
  & \scriptsize
 $ + 2 x_9 + 2 x_{10}) +  2 k_{13} x_{10}  + 2 k_{11} x_9  - k_{17} k_{36} x_6,$
 & \scriptsize
 $a=k_{31} (2 k_5 k_{14} k_{34} + 2 k_5 k_{14} k_{35} + 2 k_5 k_{34} k_{35} + k_8 k_{14} k_{34} k_{41} + k_8 k_{14} k_{35} k_{41}),$
 \\
   & \scriptsize
 $ \dot x_{11} =   k_{23}x_{14} - k_{30}x_{11},$
 & \scriptsize
   $b = (a-2 k_5 k_{31} k_{34} k_{35}) x_8+(c-a) k_{39}-c (k_{39}^2+1)/k_{39},$
   $x_{17} = (-b + \sqrt{b^2-4ac})/(2a),$
 \\
    & \scriptsize
 $ \dot x_{15} =   k_{18} - x_{12}(k_{20} + k_{26}) + k_{26}x_{12} + k_{30}x_{11} -$
 & \scriptsize
 $x_5=\frac{      4 k_5 (k_9 + k_{31} x_{17}) (k_7 + k_{14}) (x_8 - k_{39} + x_{17}) }{k_6 k_{41} (2 k_5 k_9 + 2 k_5 k_{14} + k_8 k_{14} k_{41} + 2 k_5 k_{31} x_{17}) },$
 $x_6=\frac{2 k_{14} (k_9 + k_{31} x_{17}) (x_8 - k_{39} + x_{17})}
{2 k_5 k_9 + 2 k_5 k_{14} + k_8 k_{14} k_{41} + 2 k_5 k_{31} x_{17}},$
 \\
     & \scriptsize
     $- k_{22}k_{37}x_{12}x_{13},$
 & \scriptsize
 $x_7=\frac{2 k_5 (k_9 + k_{31} x_{17}) (x_8 - k_{39} + x_{17})}{2 k_5 k_9 + 2 k_5 k_{14} + k_8 k_{14} k_{41} + 2 k_5 k_{31} x_{17}},$
 $x_9=\frac{16 k_5^2 k_{10} (k_9 + k_{31} x_{17})^2 (k_7 + k_{14})^2 (x_8 - k_{39} + x_{17})^2}{k_6^2 k_{41}^2 (k_{11} + k_{15}) (2 k_5 k_9 + 2 k_5 k_{14} + k_8 k_{14} k_{41} + 2 k_5 k_{31} x_{17})^2},$
 \\
&\scriptsize $ \dot x_{16} =   k_{28}x_{13} - k_{29}x_{16}.$ &\scriptsize
$x_{18} = \frac{k_{35}k_{39} + (k_{31}k_{39}-k_{35})x_{17} + k_{31}(x_7-x_8)x_{17} - k_{31}x_{17}^2 }{k_{35}},$
$x_{19}=\frac{k_{34} x_{18}}{k_{32}},$  $x_{20}=\frac{k_{34} x_{18}}{k_{33} k_{38}}.$ \\
 \hline
 \scriptsize $\epsilon^2$  & \scriptsize
 $\dot x_4= - k_3 (x_4 - k_{41} - k_{39} - $
 & \scriptsize
 $x_2=\frac{k_1k_{40}}{k_2},$
 \\
 \scriptsize
 & \scriptsize $ - 
     x_5 - x_6 + x_8 + x_{17} + x_{18} + x_{19} + x_{20}) -$
 & \scriptsize 
 $x_5=\frac{k_5 (k_{35}  -  k_9) k_{39} + k_5 (k_7  -  k_9) x_7 + k_5 k_9 (x_8 + x_{17}) - k_5 k_{35} (x_{17} + x_{18}) + k_7 k_8 x_4 x_7}{k_6 (k_5 + k_8 x_4) (k_{41} - x_4)},$
 \\
  \scriptsize
 & \scriptsize $ - 
     k_3 x_4 - k_9 (k_{39} + x_5 + x_6 + x_7 - x_8 - x_{17}) - $
 & \scriptsize 
 $x_6=-\frac{(k_9 - k_{35})k_{39} + k_9(x_7 - x_8)  + (k_{35}-k_9)x_{17} + k_{35}x_{18}}{k5 + k8 y4},$
 \\
   \scriptsize
 & \scriptsize $ - k_8 x_4 x_6,\, \dot x_7= - x_7 (k_7 + k_{14}) - k_6 x_5 (x_4 - k_{41} - $
 & \scriptsize 
 $x_{13}=\frac{k_{19} + k_{30}x_{11}}{k_{22}k_{37}x_{12}},$
 \\
   \scriptsize
 & \scriptsize $ -k_{39}  - x_5 - x_6 + x_8 + x_{17} + x_{18} + x_{19} + x_{20}),
       $
 & \scriptsize 
 $x_{20}=\frac{k_{32}x_{19}}{k_{33}k_{38}}.$
 \\
    \scriptsize
 & \scriptsize $   \dot x_8 =  2 k_{11} x_9  + 2 k_{13} x_{10} - 2 k_{10} x_5^2 - 2 k_{12} x_6^2 -$
 & \scriptsize 
 \\
     \scriptsize
 & \scriptsize $-k_{16} x_{11} (x_2 - k_{40} + x_8 + 2 x_9 + 2 x_{10})- k_{17} k_{36} x_6, $
 & \scriptsize 
 \\
     \scriptsize
 & \scriptsize $ \dot x_9 = k_{10} x_5^2 - x_9 (k_{11} + k_{15}), $
 & \scriptsize 
 \\
      \scriptsize
 & \scriptsize $\dot x_{10} = k_{15} x_9 - k_{13} x_{10} + k_{12} x_6^2,\,\dot x_{11} = k_{23} x_{14} - $
 & \scriptsize 
 \\
       \scriptsize
 & \scriptsize $- k_{30} x_{11}, \,\dot x_{12}= k_{18} - x_{12} (k_{20} + k_{26}) + k_{30} x_{11} - $
 & \scriptsize 
 \\
       \scriptsize
 & \scriptsize $- k_{27} (x_{12} - x_{15}) - k_{22} k_{37} x_{12} x_{13}, $
 & \scriptsize 
 \\
       \scriptsize
 & \scriptsize $ \dot x_{14} = k_{22} k_{37} x_{12} x_{13} - x_{14} (k_{23} + k_{24} + k_{25}), $
 & \scriptsize 
 \\
        \scriptsize
 & \scriptsize $ \dot x_{15} = k_{18} - x_{12} (k_{20} + k_{26}) + k_{26} x_{12} + k_{30} x_{11} -$
 & \scriptsize 
 \\
        \scriptsize
 & \scriptsize $- k_{22} k_{37} x_{12} x_{13},\,\dot x_{16} = k_{28} x_{13} - k_{29} x_{16},$
 & \scriptsize 
 \\
         \scriptsize
 & \scriptsize $\dot x_{17}= k_{31} x_{17} (k_{39} + x_5 + x_6 + x_7 - x_8 - x_{17}) - $
 & \scriptsize 
 \\
          \scriptsize
 & \scriptsize $-k_{35} (x_{17} - k_{39} + x_{18}),\, \dot x_{18} = - k_{34} x_{18} -$
 & \scriptsize 
 \\
           \scriptsize
 & \scriptsize $- k_{31} x_{17} (k_{39} + x_5 + x_6 + x_7 - x_8 - x_{17}),$
 & \scriptsize 
 \\
 &\scriptsize  $\dot x_{19} = k_{34} x_{18} - k_{32} x_{19}.$&\\
 \hline
  \end{tabular}   
  \caption{Description of various reduced models. The T column contains the 
  timescale order of the slow variable satisfying ODEs; in the case of several slow variables, it represents the fastest one.
  The fast variable column expresses the concentrations of fast variables as functions of the slow ones. The variables $x_i$ are the transformed variables defined in the Table~\ref{table:transformed_variables}. \label{table:reduced}}
\end{sidewaystable}
In order to test numerically the accuracy of the reduction 
we have eliminated the fast variables up to timescale order $\epsilon^{q_k}$ by symbolically 
solving the algebraic truncated system 
$\vFbar^{(1)}_k(\vk,\vx)=0$, eliminating the variables $\vX_k$, and numerically solving the
system of ODEs for the remaining slow variables (nested reduction  \eqref{truncatedmultiscale}). The result is represented in the
Figure~\ref{fig:figure_reduction_TGFb}, for various choices of $k$. 
As it can be noticed, especially at shorter timescales there are
 few species that are predicted with errors by the reduced model.
 There are two reasons to this phenomenon. The first reason is that
 the values of fast species are based on the truncated system of equations. 
 Although all the terms neglected by truncation have orders larger than the
 dominant terms and therefore the reduction
 is justified in the limit $\epsilon \to 0$, for finite $\epsilon$ 
 the quality of the approximation can be low if the
 number of the neglected terms is large. 
 This source of error can be reduced by considering higher order terms 
 in the approximation, for instance higher order Puiseux series 
 to represent the fast variables. 
 Another reason for bad approximation
 is the choice of the tropical equilibration used for the reduction. 
 A tropical equilibration solution is valid in a domain in the space of
 concentrations but not for all species concentrations. Furthermore, 
several tropical equilibration solutions (a polytope in log scale) lead to the same 
reduced model, but again the corresponding polytope does not cover all the 
concentration. It is thus possible that the tropical equilibration solution and 
the reduced model has to change along a trajectory of the full model when this crosses 
polytopes corresponding to different reductions. This is 
the case for the transformed TGF-$\beta$ model, see Figure~\ref{fig:figure_testing_equilibration_TGFb}.
This source of error can be reduced by considering tropical equilibration solutions 
at the boundary between polytopes, leading to reductions valid for two or several
polytopes of solutions.

\begin{figure}[ht!]
    \centering
    \includegraphics[width=0.5\linewidth]{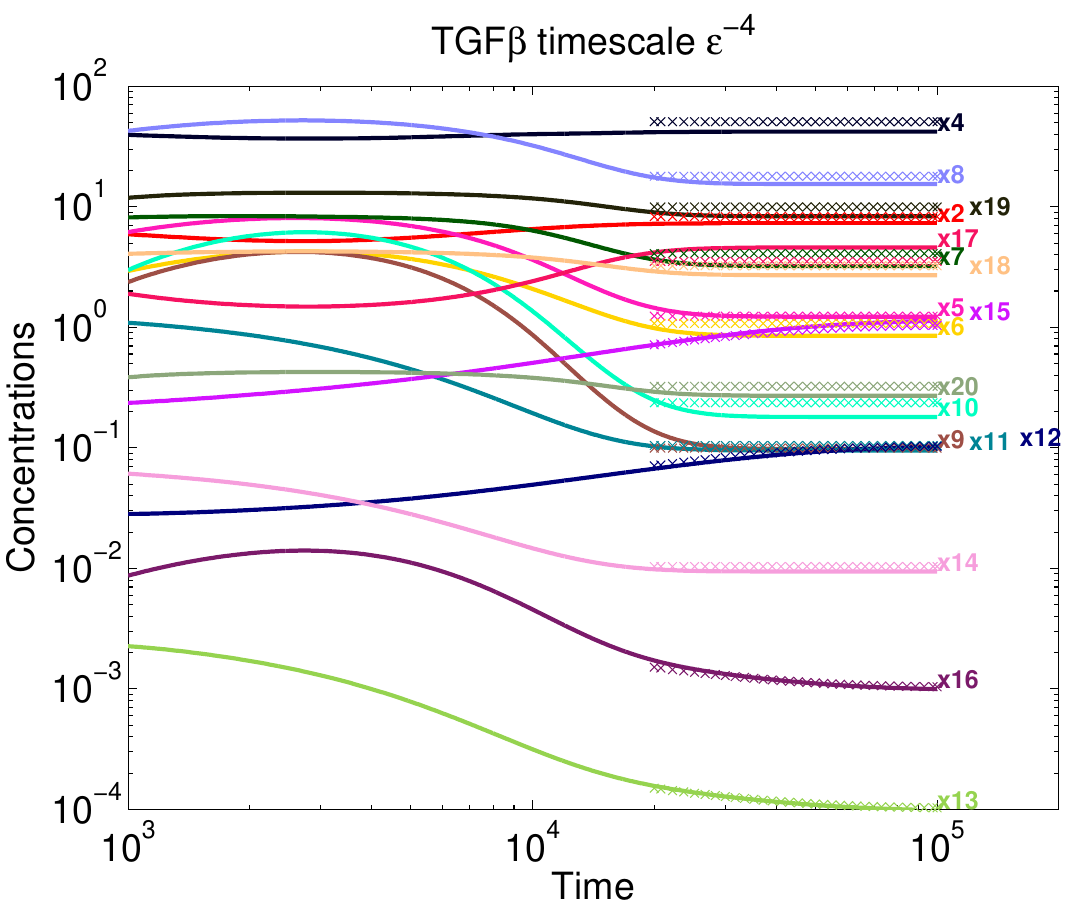}

    \includegraphics[width=0.5\linewidth]{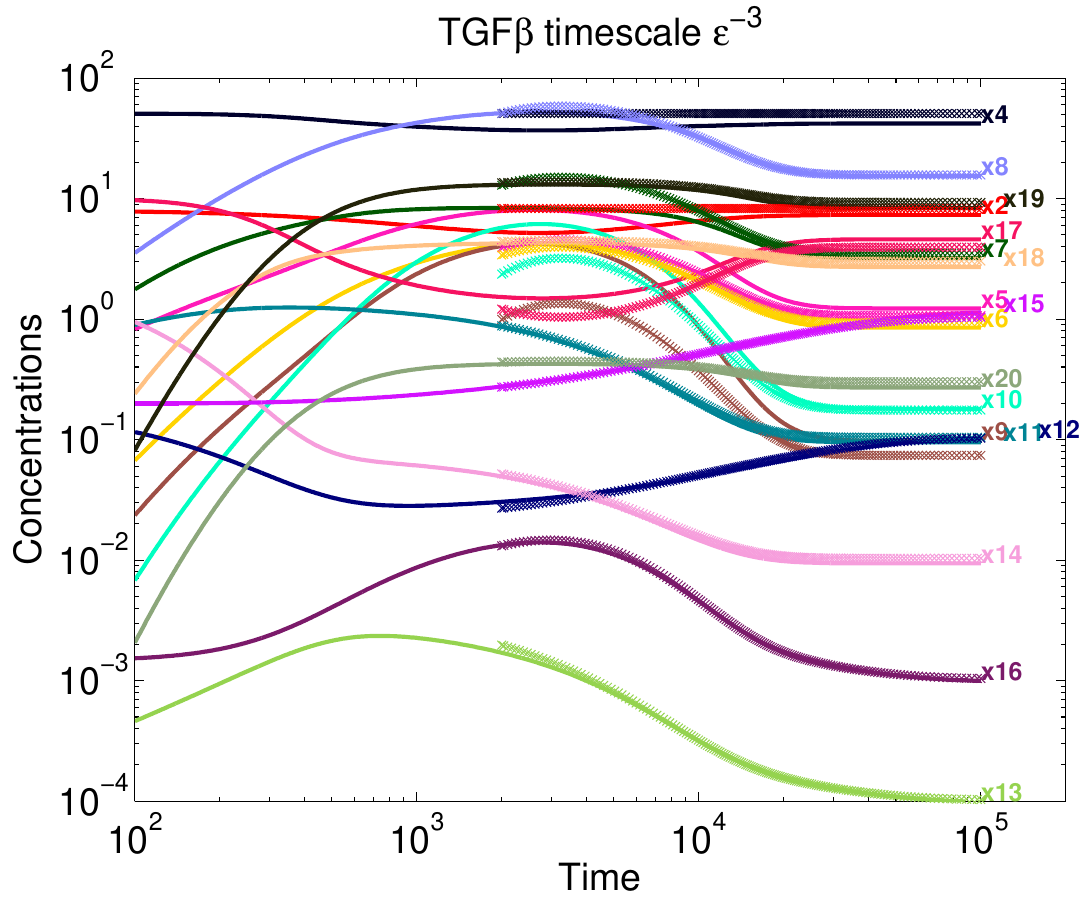}

    \includegraphics[width=0.5\linewidth]{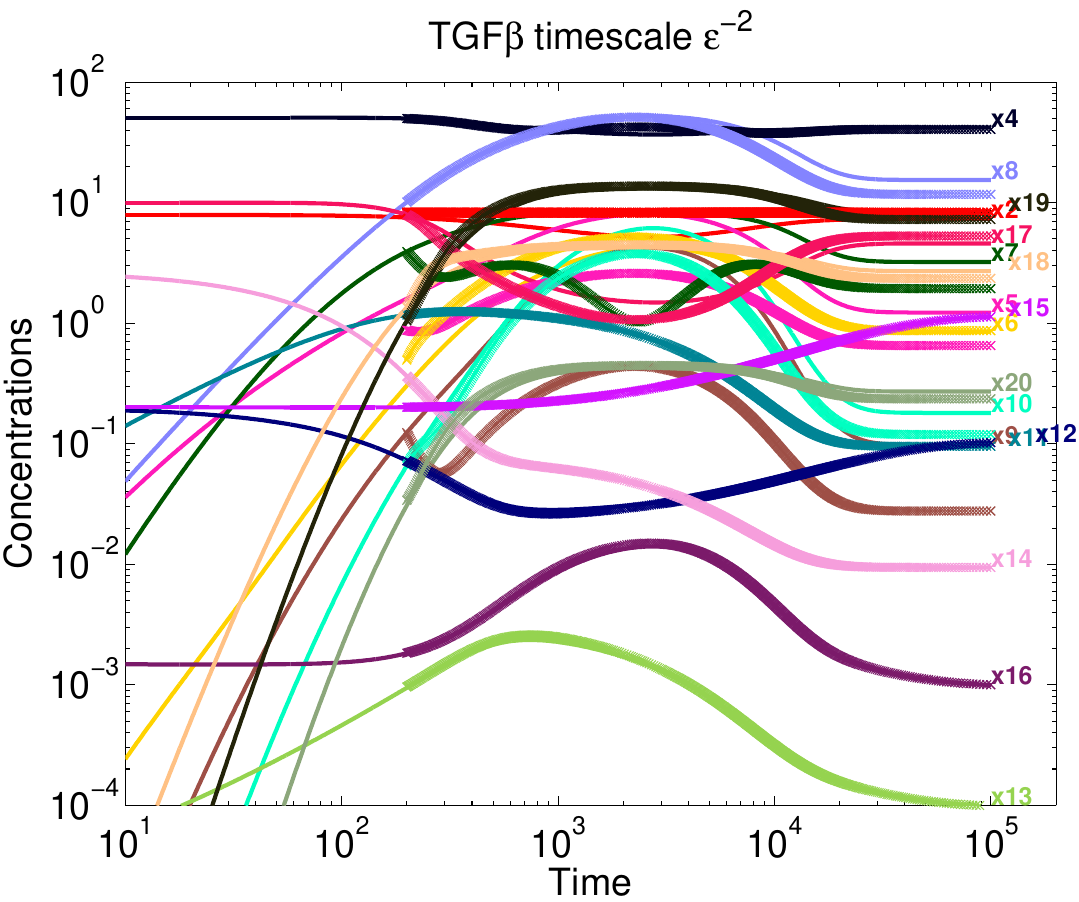}

    \caption{Comparison of numerical solutions obtained with the transformed TGF-$\beta$ model (continuous lines) and 
    with slowest timescale reduced models (crosses).
    For each reduced model, a small number of variables (slow) follow ODEs. The initial values
    of these were chosen the same as the values computed with the full transformed model
    at a large enough time. The remaining fast 
    variables were computed as functions of the slow variables. The large errors for a few species at times shorther than $10^4$ could be explained by lack of validity of the 
    tropical equilibration used for the reduction at these shorter timescales, see also
    Figure~\ref{fig:figure_testing_equilibration_TGFb}.
    }
    \label{fig:figure_reduction_TGFb}
\end{figure}

\begin{figure}[ht!]
    \centering
    \includegraphics[width=\linewidth]{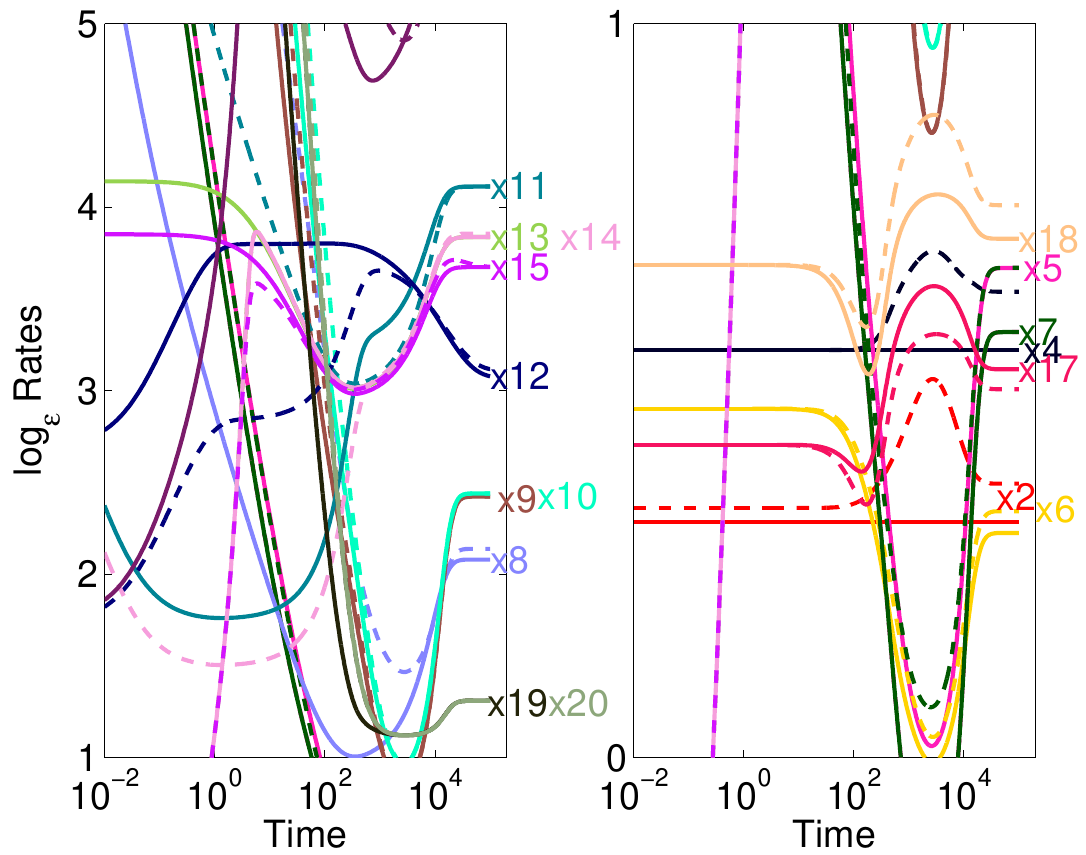}
    
    \caption{Testing tropical equilibration for various species
    of the transformed TGF-$\beta$ model. For each species we have plotted the
    $\log_\epsilon$ of positive (continuous line) and negative (dotted lines) 
    rates producing and consuming these species, respectively. For tropically equilibrated
    species the two rates must have the same order (the $\log_\epsilon$ values should 
    round up to the same integer for continuous and dotted curves of the same color). For times larger than $10^4$ this condition is valid
    for all species. For shorter times, a few species  are not equilibrated. 
    Furthermore,  some rates change abruptly at these timescales,
    suggesting that different tropical equilibration solutions should be considered at 
    shorter timescales.
    }
    \label{fig:figure_testing_equilibration_TGFb}
\end{figure}




\section*{Acknowledgement}
We thank 
Sebastian Walcher, Peter Szmolyan and Werner Seiler  for very helpful discussions.
The project SYMBIONT owes a lot to Andreas Weber who sadly left us in 
2020, but who is still present in our memories.


\end{document}